\documentclass[leqno,a4paper,11pt]{amsart}
\usepackage{fullpage}

\usepackage[x11names]{xcolor} 

\usepackage[T1]{fontenc}
\usepackage[utf8]{inputenc} 

\usepackage{amsfonts,amsmath,amssymb}
\usepackage[fleqn]{mathtools}
\usepackage{float}
\usepackage{bm}
\usepackage{upgreek}
\usepackage{slashed}
\usepackage{accents}
\usepackage{youngtab}

\usepackage[ddmmyyyy]{datetime}

\usepackage{tikz}
\usepackage{graphicx}
\usepackage[all]{xy}
\usepackage{array}
\usepackage{scalerel}
\usepackage{multirow}
\usepackage[shortlabels]{enumitem}
\usepackage{caption}
\usepackage{hyperref}

\usepackage{amsthm}
\theoremstyle{plain}
\newtheorem{thm}{Theorem}[section]
\newtheorem*{thm*}{Theorem}
\newtheorem{lem}[thm]{Lemma}
\newtheorem{prop}[thm]{Proposition}
\newtheorem{cor}[thm]{Corollary}
\theoremstyle{definition}
\newtheorem{defn}[thm]{Definition}

\newtheorem{exa}[thm]{Example}

\theoremstyle{remark}
\newtheorem{rem}[thm]{Remark}

\numberwithin{equation}{section}
\newcommand{\parderv}[2] {\frac{\partial #1}{\partial #2}}

\renewcommand{\d} {\mathrm{d}}

\newcommand{\ol} [1] {\overline{#1}}
\newcommand{\wh} [1]{\widehat{#1}}
\newcommand{\wt} [1]{\widetilde{#1}}
\newcommand{\mr}[1] {\mathring{#1}}

\newcommand{\mbf} [1]{\mathbf{#1}}

\newcommand{\mc} [1]{\mathcal{#1}}

\renewcommand{\i} {\mathrm{i}}
\newcommand{\e} {\mathrm{e}}

\newcommand{\Id} {\mathrm{Id}}

\renewcommand{\bigwedge}{\scaleobj{1.2}{\wedge}}

\renewcommand{\bigodot}{\scaleobj{1.2}{\odot}}

\newcommand{\hook}{\makebox[7pt]{\rule{6pt}{.3pt}\rule{.3pt}{5pt}}\,}

\newcommand{\Riem} {\mathsf{R}}
\newcommand{\Weyl} {\mathsf{W}}
\newcommand{\Ric} {\mathsf{Ric}}
\newcommand{\Sc} {\mathsf{Sc}}
\newcommand{\WbA} {\mathsf{A}}
\newcommand{\CQ} {\mathsf{Q}}
\newcommand{\CT} {\mathsf{T}}
\newcommand{\Rho} {\mathsf{P}}
\newcommand{\Bach} {\mathsf{B}}
\newcommand{\Cot} {\mathsf{Y}}

\newcommand{\pr} {\mathrm{pr}}

\newcommand{\Ann} {\mathrm{Ann}}

\newcommand{\Z} {\mathbf{Z}}

\newcommand{\R} {\mathbf{R}}
\newcommand{\C} {\mathbf{C}}

\newcommand{\SU}{\mathbf{SU}}

\begin{document}
	
	\title[Perturbations of Fefferman spaces over CR three-manifolds]{Perturbations of Fefferman spaces\\ over CR three-manifolds}
	
	\author{Arman Taghavi-Chabert}
	\address{Institute of Physics, {\L}\'{o}d\'{z} University of Technology, W\'{o}lcza\'{n}ska 217/221, 90-005 {\L}\'{o}d\'{z}, Poland; and Department of Geometry, Faculty of Mathematics and Computer Science, University of {\L}\'{o}d\'{z}, Stefana Banacha 22, 90-238 {\L}\'{o}d\'{z}, Poland}
	\email{arman.taghavi-chabert@p.lodz.pl}
	\thanks{The research leading to these results has received funding from the Norwegian Financial Mechanism 2014-2021 UMO-2020/37/K/ST1/02788.}
	
	
	\date{}
	
	\dedicatory{In memory of Jerzy Lewandowski (1959--2024)}
	
	\begin{abstract}
		We introduce a generalisation of Fefferman's conformal circle bundle over a contact Cauchy--Riemann three-manifold. These can be viewed as exact `perturbations' of Fefferman's structure by a semi-basic one-form, which encodes additional data on the CR three-manifold.
		
		We find conditions on the Weyl tensor and the Bach tensor for a Lorentzian conformal four-manifold equipped with a twisting non-shearing congruence of null geodesics to be locally conformally isometric to such a perturbed Fefferman space.
		
		We investigate the existence of metrics in the perturbed Fefferman conformal class satisfying appropriate subsystems of the Einstein equations, such as the so-called pure radiation equations. These metrics are defined only off cross-sections of Fefferman's circle bundle, and are conveniently expressed in terms of densities that generalise Gover's notion of almost Einstein scales. Our setup allows us to reduce the prescriptions on the Ricci tensor to the underlying CR three-manifold in terms of differential constraints on a complex-valued density and the CR data of the perturbation one-form. One such constraint turns out to arise from a non-linear, or gauged, analogue of a second-order differential operator on densities. A solution thereof provides a criterion for the existence of a CR function and, under certain conditions, for the realisability of the CR three-manifold. These findings are consistent with previous works by Lewandowski, Nurowski, Tafel, Hill, and independently, by Mason.
		
		We also provide an analysis of the Weyl curvature of such conformal structures in terms of the underlying CR data. In particular, we arrive at a CR formulation of the asymptotic Einstein condition by viewing conformal infinity as a cross-section of Fefferman's circle bundle.
	\end{abstract}
	
	\maketitle
	\tableofcontents
	
	\section{Introduction}\label{sec:intro}
	The focus of this article is the interaction between Lorentzian conformal geometry in dimension four and Cauchy--Riemann (CR) geometry in dimension three, and is motivated by two different approaches. On the one hand, there is Fefferman's classical construction \cite{Fefferman1976a}, which associates to a given contact CR three-manifold $(\mc{M},H,J)$ a conformal structure $\wt{\mbf{c}}$ of Lorentzian signature on the total space  of a circle bundle $\wt{\mc{M}}$ over $(\mc{M},H,J)$. In the formulation of \cite{Lee1986,Cap2008}, each contact form $\theta$ with Levi form $h$ and associated induced Weyl connection one-form $\wt{\omega}^{W}_{\theta}$ gives rise to a \emph{Fefferman metric}
	\begin{subequations}
		\begin{align}\label{eq:Fefmet}
			\wt{g}_{\theta} & = 4 \theta \odot \wt{\omega}^{W}_{\theta} + h \, ,
		\end{align}
		on $\wt{\mc{M}}$. A change of contact form induces a conformal change of Fefferman metric. In addition, $\wt{\mbf{c}}$ admits a \emph{null conformal Killing field}, vertical with respect to the fibration $\wt{\mc{M}} \rightarrow \mc{M}$.
		
		On the other hand, many exact solutions to the vacuum Einstein's field equations admit a so-called \emph{non-shearing congruence of null geodesics} --- the terminology \emph{shear-free} in place of non-shearing is also used. These are characterised by the property that the local leaf space of such a congruence is a CR three-manifold. Those congruences for which the CR three-manifold is contact are said to be \emph{twisting}. We may therefore rightly see our Fefferman space as a special case of such Lorentzian geometries, and one of the questions we will answer in this article is to what extent these two constructions differ.
		
		When it comes to the existence of \emph{global} Einstein metrics in the Fefferman class, the answer is rather disappointing: there are none. However, there \emph{can} be local ones defined off two cross-sections of $\wt{\mc{M}} \rightarrow \mc{M}$. These are given in the form
		\begin{align}\label{eq:FefEinsmet}
			\wt{g} & = \sec^2 \phi \cdot \wt{g}_{\theta} \, ,
		\end{align}
	\end{subequations}
	where $\theta$ is a contact form on $\mc{M}$, $\wt{g}_{\theta}$ is its associated Fefferman metric, and $\phi$ is a local fibre coordinate. Clearly $\wt{g}$ is defined only  when $\phi = \pm \frac{\pi}{2}$, but unfortunately, it turns out that for $\wt{g}$ to be Einstein, it must also be \emph{conformally flat} \cite{Lewandowski1988}. This also means that the underlying CR structure is locally CR equivalent to the Heisenberg group or the three-sphere \cite{Cap2008}.
	
	But there are of course many other, non-conformally flat, algebraically special Einstein metrics. In local coordinates $(u, \zeta, \ol{\zeta}, r)$, broadly following \cite{Debney1969,Stephani2003}, they may be cast as
	\begin{subequations}\label{eq:Debmet}
		\begin{align}
			\wt{g} & = 4 \mr{\varrho} \cdot \omega \odot \wt{\lambda}{}'_{\omega} + 2 (r^2 + \mr{\varrho}^2)   \cdot \omega^{1} \odot \ol{\omega}{} ^{\bar{1}}  \, .
		\end{align}
		where
		\begin{align}
			\omega & = \tfrac{1}{2 \mr{\varrho} } \left(\d u + Z \d \zeta + \ol{Z} \d \ol{\zeta} \right) \, , & \omega^1 & = \d \zeta \, , &
			\wt{\lambda}{}'_{\omega} & = \d r + \wt{W} \omega^1 + \ol{\wt{W}} \ol{\omega}{}^{\bar{1}} + \wt{U} \omega \, .
		\end{align}
	\end{subequations}
	Here, the coordinate $r$ is an affine parameter along the null geodesics of the congruence defined by the null vector field $\wt{k}$, and $(u, \zeta, \ol{\zeta})$ are coordinates on the leaf space, with $u$ real, and $\zeta$ complex. The nowhere vanishing complex-valued function $Z$ and real-valued function $\mr{\varrho}$ are defined on $\mc{M}$, and the one-forms $\omega$ and $\omega^1$ define the contact CR structure on the leaf space, with $\d \omega = \i \omega^1 \wedge \ol{\omega}{}^{\bar{1}} \pmod{\omega}$. The $r$-dependence of the complex-valued function $\wt{W}$ and real-valued function $\wt{U}$ is fully determined by virtue of the algebraic degeneracy of the Weyl tensor and (a subset of) Einstein's field equations, the	so-called \emph{reduced Einstein equations}. This form of the metric has the advantage of being amenable to asymptotic analysis as shown in \cite{Mason1998}.
	
	A more convenient form for these metrics was obtained by Lewandowski and Nurowski in \cite{Lewandowski1990}, who found that any algebraically special pure radiation metric can be cast in the form
	\begin{align}\label{eq:LNmet}
		\wt{g} & = \mr{\varrho}^2 \sec^2 (\phi - \mr{\phi} )\cdot \left( 4 \theta \odot \wt{\lambda} + 2 \theta^1 \odot \ol{\theta}{}^{\bar{1}} \right) \, ,
	\end{align}
	where $\phi$ is an affine parameter along the geodesics of the congruence, the one-forms $\theta$ and $\theta^1$ define the CR structure on the leaf space $\mc{M}$, with $\d \theta = \i \theta^1 \wedge \ol{\theta}^1$, and $\mr{\varrho}$ and $\mr{\phi}$ are smooth functions on $\mc{M}$, but, as argued in \cite{TaghaviChabert2022}, both $\mr{\varrho}$ and $\mr{\phi}$ can be eliminated by a rescaling of the contact form $\theta$ and a reparametrisation of the geodesics respectively. The components of the null one-form $\wt{\lambda}$ all admit finite Fourier expansions in $\phi$, so their coefficients are functions on $\mc{M}$. The coordinates $\phi$ and $r$ from our previous form are related by the transformation $r = \mr{\varrho} \tan (\phi - \mr{\phi})$. We can then see at once that each of the metrics \eqref{eq:LNmet} and \eqref{eq:FefEinsmet} with \eqref{eq:Fefmet} live on a circle bundle with some cross-sections removed, and this strongly suggests that they only differ from each other by some `perturbation' semi-basic one-form $\wt{\xi}$, i.e.\ $\wt{\lambda} = \wt{\omega}^W_{\theta} + \wt{\xi}$. One of the aims of this article will be to establish this fact formally. From the relation between the two coordinates $r$ and $\phi$, we see that $r \rightarrow \pm \infty$ as $\phi \rightarrow \pm \tfrac{\pi}{2}$, so one can interpret the hypersurfaces where the metric $\wt{g}$ blows up as \emph{conformal infinities}. To accommodate these singularity sets, we shall put Gover's notion of \emph{almost Einstein scale} into good use. These are certain densities $\wt{\sigma}$ on $\wt{\mc{M}}$ satisfying a conformally invariant second-order differential equation, which encapsulates the existence of an Einstein metric in the conformal class, but defined only off the singularity set of $\wt{\sigma}$, if non-empty.
	
	A by-product of the investigation of these non-shearing congruences of null geodesics, which is the focus of \cite{Lewandowski1990a,Hill2008}, is the discovery of analytic properties of CR three-manifolds arising from some subset of Einstein's field equations: these concern the existence of \emph{CR functions}, of which \emph{Kerr's complex coordinate} $\zeta$ in \eqref{eq:Debmet} is an example, or even the realisability of the CR manifold, under suitable conditions.
	
	More recently, the authors of \cite{Schmalz2019} revisited the work of \cite{Lewandowski1991,Hill2008} by focussing on a variant of Fefferman's construction, which admits non-conformally flat solutions of a certain subset of Einstein's field equations, and by appealing to Jacobowitz's work on the realisability of CR manifolds \cite{Jacobowitz1987}. The present article will delve further into these ideas.
	
	The overall aim of the article is two-fold:
	\begin{enumerate}
		\item To introduce and study CR invariant differential equations for contact CR manifolds, in connection with the existence of CR functions and CR embeddability. These naturally occur in the investigation of algebraically special Einstein Lorentzian four-manifolds.
		\item To show that algebraically special Lorentzian four-manifolds subject to Ricci curvature prescriptions can find a natural formulation in the context of Fefferman spaces. To be precise, these belong to the same conformal class as `perturbed' Fefferman metrics, i.e.\ they are of the form $\wt{g} = \sec^2 \phi \cdot \left( \wt{g}_{\theta} + 4 \theta \odot \wt{\xi} \right)$ where $\wt{g}_{\theta}$ is a Fefferman metric for some contact form $\theta$, $\wt{\xi}$ is a semi-basic one-form, determined by some CR data, and $\phi$ a local fibre coordinate arising from a distinguished trivialisation of the Fefferman bundle.
	\end{enumerate}	
	
	While the present paper remains by and large very much self-contained, it refers to some crucial results found in the companion article (or prequel) \cite{TaghaviChabert2023b}, which deals with perturbed Fefferman spaces in arbitrary even dimensions, with a focus on dimensions greater than four. It is also noteworthy to mention that the present results depend on curvature computations carried out by the author in arbitrary even dimensions in \cite{TaghaviChabert2022}.
	
	We presently outline the structure of the paper while highlighting its main results. Section \ref{sec:CR_geom} gives a review of CR geometry in any odd dimensions, but new interpretations are also obtained in connection with the realisability of \emph{contact} CR structures. In particular, we recast two results due to Jacobowitz \cite{Jacobowitz1987} in terms of solutions to a CR invariant second-order differential equation in Proposition	\ref{prop:New_Jacobowitz} and Theorem \ref{thm:real2dens}. In fact, this equation is related to a so-called \emph{first BGG operator} as pointed out in Remarks \ref{rem:BGG} and \ref{rem:CR-Eins}. A non-linear analogue of this equation, which we term a \emph{Webster--Weyl structure} provides, in dimension three, a criterion for the existence of CR functions --- Theorem \ref{thm:Webster--Weyl}. Theorem \ref{thm:lambda_exact} shows that this non-linear equation can always be `linearised' in a certain sense.
	
	Section \ref{sec:conf_geom} contains a review of conformal geometry and so-called optical geometries which are at the heart of non-shearing congruences of null geodesics. In Definitions \ref{defn:redEins} and \ref{defn:aESc}, we introduce formal definitions of certain metrics with prescribed Ricci curvature, which we term \emph{weakly half-Einstein}, \emph{half-Einstein} and \emph{pure radiation} metrics, and their \emph{almost Lorentzian scale} analogues as generalisations of almost Einstein scales --- Proposition \ref{prop:scale2metric}. We then give a brief revision of the notion of \emph{principal null directions of the Weyl tensor} and the \emph{Petrov types}, and in Theorems \ref{thm:GS} and \ref{thm:GSalpha}, we revisit two versions of the Goldberg--Sachs theorem using almost Lorentzian scales.
	
	We briefly recall the Fefferman construction in Section \ref{sec:Fefferman}, and define their perturbations by a semi-basic one-form and associated \emph{CR data} understood as a tuple of densities on the CR manifold.
	
	In Section \ref{sec:algsp}, we start our investigation of optical geometries with twisting non-shearing null geodesics in dimension four and their relation to perturbed Fefferman spaces. In particular, we prove Theorem \ref{thm:PetrovII+Bach}, which can be summarised as follows:
	\begin{thm*}
		Let $(\wt{\mc{M}},\wt{\mbf{c}})$ be an oriented and time-oriented Lorentzian conformal four-manifold with twisting non-shearing congruence of null geodesics $\wt{\mc{K}}$. Suppose that for any generator $\wt{k}$ of $\wt{\mc{K}}$, the Weyl tensor and the Bach tensor of $\wt{\mbf{c}}$ satisfy
		\begin{align*}
			\wt{\Weyl}(\wt{k},\wt{v},\wt{k},\cdot) & = 0 \, , & \mbox{for any $\wt{v} \in \Gamma(\langle \wt{k} \rangle^\perp)$,} \\
			\wt{\Bach}(\wt{k},\wt{k}) & = 0 \, ,
		\end{align*}
		respectively. Then $(\wt{\mc{M}},\wt{\mbf{c}})$ is locally conformally isometric to a perturbation of a Fefferman space over a contact CR three-manifold.
	\end{thm*}
	
	In Section \ref{sec:al_Lor_sc}, we investigate the consequences of the existence of distinguished almost Lorentzian scales satisfying differential conditions of increasing stringency --- these include almost (weakly) half-Einstein and pure radiation scales. Notably, Proposition \ref{prop:alwkhalfE}, Theorems \ref{thm:wkhlfEinsc}, \ref{thm:hlfEinsc} and \ref{thm:pertFeffpurad} show how these prescriptions can be translated into CR invariant differential equations. They provide CR invariant articulations of the results of \cite{Lewandowski1990,Mason1998}. Of interest is Proposition \ref{prop:al_half-Einstein}, which gives an alternative characterisation of algebraically special optical geometries as perturbed Fefferman spaces. Drawing from Section \ref{sec:CR_geom}, we also recover the results of \cite{Lewandowski1990a,Hill2008} regarding CR realisability in Proposition \ref{prop:wkhlfEinsc}, Theorems \ref{thm:realisableCRMax} and \ref{thm:realisableCRPetrovII}. On that matter of realisability, Corollary \ref{cor:wkhlfEinsc} adds new results in some specific cases of perturbed Fefferman spaces.

	In Proposition	\ref{prop:Weyl} of Section \ref{sec:Weyl}, we provide a detailed description of the Weyl tensor of an important class of perturbed Fefferman spaces in terms of weighted tensors on the base space. In Lemma \ref{lem:PetrovD} and Proposition \ref{prop:PetrovIII-N-O}, we consider the consequences of the algebraic degeneracy of the Weyl tensor on the CR data according to the Petrov types. These weighted tensors are shown to satisfy algebraic conditions in the presence of an almost half-Einstein scale or an almost pure radiation scale in Propositions \ref{prop:hEinasym} and \ref{prop:pureradasym}. The asymptotic behaviour of the curvature at conformal infinity is examined in Propositions \ref{prop:Wasym} and \ref{prop:confEins}.
	
	Three appendices are included: Appendix \ref{app:coframe2} contains some technical results regarding adapted coframes in dimension three, which are used in Section \ref{sec:CR_geom}. In Appendix \ref{app:Bach_prop}, we provide computations of parts of the Bach tensor used in Section \ref{sec:algsp}. In Appendix \ref{app:others}, we sketch how to recover the Debney--Kerr--Schild metric from our approach.
	
	\vspace{1mm}

	\section{CR geometry}\label{sec:CR_geom}
	We start our journey with a discussion of general CR manifolds in Section \ref{sec:genCR}, where we notably collect known results regarding the analytic aspects of CR structures of \emph{any} Levi signatures --- see \cite{Jacobowitz1990} for a pleasant introduction to the topic. From Section \ref{sec:contactCR} on, exploiting the Webster calculus, we restrict our attention to non-degenerate CR structures. This allows us to provide new geometric characterisations of some analytic properties of CR manifolds in terms of CR invariant differential equations. While the section is by-and-large self-contained, we shall later apply some of the results in Section \ref{sec:al_Lor_sc}.

	\subsection{Preliminaries}\label{sec:genCR}
	An \emph{almost Cauchy--Riemann (CR) structure} on a smooth $(2m+1)$-dimen-sional manifold ${\mc{M}}$ consists of a pair $({H},{J})$ where ${H}$ is a rank-$2m$ distribution and ${J}$ a bundle complex structure on ${H}$, i.e.\ ${J} \circ {J} = - {\Id}$, where ${\Id}$ is the identity map on ${H}$. This means that the complexification ${}^\C {H}$ of ${H}$ splits as ${}^\C {H} = {H}^{(1,0)} \oplus {H}^{(0,1)}$ where ${H}^{(1,0)}$ and ${H}^{(0,1)}$ are the rank-$m$ $\i$-eigenbundle and $-\i$-eigenbundle of ${J}$ respectively. The complex line bundle ${\mc{C}} := \bigwedge^{m+1} \Ann({H}^{(0,1)})$ is known as the \emph{canonical bundle} of $(\mc{M},H,J)$.
	
	Throughout, we shall assume that $({\mc{M}},{H},{J})$ is a \emph{CR manifold}, that is, $H^{(1,0)}$ is involutive, i.e.\ $[H^{(1,0)},H^{(1,0)}] \subset H^{(1,0)}$, which always holds in dimension three, i.e.\ $m=1$. In addition, $\mc{M}$ and $\Ann(H)$ will be assumed to be oriented. In the present context, a choice of nowhere vanishing one-form $\theta$ annihilating $H$ will be referred to as a \emph{pseudo-Hermitian structure}, and by convention, we will assume that $\theta$ agrees with the positive orientation on $\Ann(H)$. Its \emph{Levi form} is the Hermitian form $h$ on $H^{(1,0)}$ defined by $h(v,\ol{w}) = - 2 \i \d \theta (v, \ol{w}) $ for any $v, w \in \Gamma(H^{(1,0)})$. Under a positive rescaling of $\theta$, the Levi form $h$ transforms by a positive conformal factor. In particular, the \emph{signature} $(p,q)$ of $h$, where $p$ and $q$ are non-negative integers such that $p+q \leq m$ and denote the number of positive and negative eigenvalues respectively, provides an invariant of the CR structure. This convention implies that $h$ has $m-p-q$ zero eigenvalues. By extension, we define the \emph{Levi form} of $(H,J)$ to be the Hermitian form $\bm{h}$ on $H^{(1,0)}$ taking values in the complex line bundle ${}^\C T \mc{M}/{}^\C H$, given by $\bm{h}(v,\ol{w}) = 2 \i \pr([v,\ol{w}])$ for any $v, w \in \Gamma(H^{(1,0)})$, where $\pr$ is the natural projection from ${}^\C T \mc{M}$ to ${}^\C T \mc{M}/{}^\C H$.
	
	The CR structure is said to be \emph{contact} if the distribution $H$ is contact, i.e.\ maximally non-integrable. This means that, for any nowhere vanishing section $\theta \in \Gamma(\Ann(H))$, the $(2m+1)$-form $\theta \wedge (\d \theta)^m$ is nowhere vanishing. The Levi form is then non-degenerate, i.e.\ it has signature $(p,q)$ with $p+q=m$. A common term for a contact CR structure with everywhere positive definite Levi form is \emph{strictly pseudo-convex}. The \emph{Reeb vector field} of a contact form $\theta$ is the unique vector field $\ell$, satisfying ${\theta} ({\ell}) = 1$ and $\d {\theta} ( {\ell} , \cdot ) = 0$. Since there is in general no preferred pseudo-Hermitian structure on $(\mc{M},H,J)$, we obtain a subconformal structure ${\mbf{c}}_{{H}}$ on ${H}$ compatible with $J$.

	A \emph{(complex) infinitesimal symmetry of $H^{(1,0)}$} is a complex-valued vector field $v$ on $\mc{M}$ that preserves the distribution $H^{(1,0)}$ in the sense that $\mathsterling_v w \in \Gamma(H^{(1,0)})$ for any $w \in \Gamma(H^{(1,0)})$. An \emph{infinitesimal symmetry of $(\mc{M},H,J)$} is a real-valued vector field $v$ on $\mc{M}$ that preserves $H^{(1,0)}$. In particular, since it is real, it preserves $H^{(0,1)}$ too. An infinitesimal symmetry $v$ of either $H^{(1,0)}$ or $(H,J)$ is said to be \emph{transverse (to $H$)} if it satisfies $v \hook \theta \neq 0$ for any $\theta \in \Gamma(\Ann(H))$.
	
	\begin{rem}
		A complex infinitesimal symmetry of $H^{(1,0)}$ as we have defined it is not an infinitesimal symmetry in the strict sense of the term since it does not generate a one-parameter family of diffeomorphisms preserving the CR structure unless it is real. Its geometric interpretation will be given in Jacobowitz's Proposition \ref{prop:Jacobowitz} below.
	\end{rem}
	
	Real hypersurfaces in $\C^{m+1}$ are the prototypical examples of CR manifolds. Not every CR manifold arises in this way, but those that do are referred to as (locally) \emph{realisable} or \emph{embeddable}, in the sense that locally, there exists an embedding of $\mc{M}$ into $\C^{m+1}$ with the property that the bundle complex structure $J$ on $H$ is induced from the ambient complex structure on $\C^{m+1}$.
	
	A \emph{(local) CR function} on $({\mc{M}}, {H},{J})$ is a smooth complex-valued function $f$ defined in the neighbourhood $\mc{U}$ of a point in $\mc{M}$ that satisfies $v({f}) = 0$ for any section $v$ of ${H}^{(0,1)}$ on $\mc{U}$. We shall say that a set of CR functions $\{ f_1, \ldots f_r\}$ on some neighbourhood $\mc{U}$ of a point is \emph{independent} if they satisfy
	\begin{align*}
		\d f_1 \wedge \ldots \wedge \d f_r \neq 0 \, , 
	\end{align*}
	at every point of $\mc{U}$, and \emph{strongly independent} if they satisfy
	\begin{align*}
		\d f_1 \wedge \d \ol{f_1} \wedge \ldots \wedge \d f_r \wedge \d \ol{f_r} \neq 0 \, , 
	\end{align*}
	at every point of $\mc{U}$. Realisability is then equivalent to the existence of $m+1$ independent CR functions. We shall give further criteria for realisability below. Before that, we give the following criterion for the existence of a CR function:
	\begin{lem}[\cite{Hill2008}]\label{lem:HLN}
		Locally, a CR manifold $({\mc{M}}, {H},{J})$ admits a strongly independent CR function if and only if there exists a smooth complex-valued one-form $\alpha$ annihilated by $H^{(0,1)}$ such that $\alpha \wedge \d \alpha = 0$ and $\alpha \wedge \ol{\alpha} \neq 0$.
	\end{lem}
	More precisely, locally, one can find two smooth complex-valued functions $w$ and $z$ on $\mc{M}$, with $z$ nowhere vanishing, such that $\alpha = z^{-1} \d w$.

	\begin{prop}[\cite{Jacobowitz1987}]\label{prop:Jacobowitz}
		Locally, a CR manifold $({\mc{M}}, {H},{J})$ is realisable if and only if it admits a complex transverse infinitesimal symmetry of $H^{(1,0)}$.
	\end{prop}
	
	The existence of a (real) transverse infinitesimal symmetry of $(\mc{M},H,J)$ is sufficient but not necessary for realisability.

	On a realisable CR manifold of dimension $2m+1$, the pullback of any holomorphic $(m+1,0)$-form from $\C^{m+1}$ to $\mc{M}$ gives rise to a closed section of the canonical bundle. So realisability provides an abundance of such forms. For a `weaker' converse, having enough CR functions at disposal leads to the following result --- see also Remark \ref{rem:res+conj} below in relation to Trautman's conjecture.
	
	\begin{thm}[\cite{Jacobowitz1987}]\label{thm:Jacobowitz}
		Locally, a $(2m+1)$-dimensional CR manifold is realisable if and only if it admits $m$ strongly independent CR functions and a nowhere vanishing closed section of the canonical bundle. 
	\end{thm}
	
	\begin{rem}
		In dimension three, it follows that once we know that $(\mc{M},H,J)$ admits a CR function, realisability becomes equivalent to the existence of a nowhere vanishing closed section of the canonical bundle.
	\end{rem}	
	
	For ease of reference, we summarise the various criteria for the realisability of a CR manifold below. Given the following statements on a CR manifold $(\mc{M},H,J)$:
	\begin{enumerate}
		\item there is a (real) transverse symmetry of $(H,J)$;
		\item there is a (complex) transverse symmetry of $H^{(1,0)}$;
		\item the CR structure is realisable;
		\item there are $m+1$ independent CR functions;
		\item there are $m$ strongly independent CR functions and a nowhere vanishing closed section of $\mc{C}$;
		\item there is a nowhere vanishing closed section of $\mc{C}$;
	\end{enumerate}
	we have that
	\begin{align*}
		(1) & &  \Longrightarrow & & (2) & & \Longleftrightarrow & &  (3) & & \Longleftrightarrow & & (4) & & \Longleftrightarrow & & (5) & & \Longrightarrow & & (6) \, .
	\end{align*}
	
	\begin{rem}\label{rem:CRop}
		Local sections of $H^{(0,1)}$ can be interpreted as \emph{Cauchy--Riemann operators}, sometimes called \emph{Lewy operators}. CR functions thus correspond to smooth complex-valued functions that lie in the kernel of all such operators. For a CR manifold embedded as a real hypersurface in complex space, these operators arise from the restriction of the standard Cauchy--Riemann operators on $\C^{m+1}$. We say that a Cauchy--Riemann operator $\ol{\partial}$ is \emph{solvable} at a point if for each sufficiently smooth complex-valued function $\xi$ in its neighbourhood, there exists a nowhere vanishing complex-valued function $z$ in a possibly smaller neighbourhood such that
		\begin{align}\label{eq:Lewy}
			\xi = \i \ol{\partial} \log z \, .
		\end{align}
		In dimension three, even for realisable CR manifolds, there exist functions $\xi$ for which there is no $z$ solving \eqref{eq:Lewy} \cite{Lewy1957}. In other words, the $\ol{\partial}$-Poincar\'{e} Lemma does not hold in CR geometry in general.
	\end{rem}
	
	\begin{rem}\label{rem:res+conj}
		It is worth listing a number of known facts and conjectures regarding the embeddability of CR manifolds --- see also \cite{Trautman2002}:
		\begin{itemize}
			\item Not every strictly pseudo-convex CR \emph{three}-manifold is embeddable \cite{Nirenberg1973,Jacobowitz1982}. This essentially follows from Lewy's result \cite{Lewy1957} --- see Remark \ref{rem:CRop}.
			\item For $m\geq1$, not every CR manifold of dimension $2m+1$ and Levi signature $(m-1,1)$ is embeddable \cite{Jacobowitz1983}.			
			\item For $m\geq3$, every strictly pseudo-convex CR manifold of dimension $2m+1$ is embeddable, \cite{Kuranishi1982,Akahori1987,Webster1989a,Webster1989}.
			\item There exist CR \emph{seven}-manifolds of Levi signature $(1,1)$ that are not embeddable, yet admit nowhere vanishing closed sections of the canonical bundle \cite{LeBrun1984,Jacobowitz1987}.
			\item It is not known whether every strictly pseudo-convex CR \emph{five}-manifold is embeddable.
			\item Trautman conjectured \cite{Trautman1999a,Jacobowitz2020} that every strictly pseudo-convex three-manifold admitting a nowhere vanishing closed section of the canonical bundle is embeddable.
		\end{itemize}
	\end{rem}
	
	\subsection{Contact CR structures}\label{sec:contactCR}
	Henceforth, $(\mc{M},H,J)$ will be assumed to be a $(2m+1)$-dimensional contact CR manifold with signature $(p,q)$, $p+q=m$. There is a well-developed theory of such geometries available \cite{Tanaka1975,Webster1978,Gover2005,Cap2008,Cap2010}, which we shall exploit throughout this article.
	
	It will be convenient to assume that the canonical bundle $\mc{C}$ admits a $-(m+2)$-nd root, and define the line bundle ${\mc{E}}(1,0) := {\mc{C}}^{-\tfrac{1}{m+2}}$, its complex conjugate ${\mc{E}}(0,1) := \overline{{\mc{E}}(1,0)}$ and dual  ${\mc{E}}(-1,0) :=  {\mc{E}}(1,0)^*$. More generally, we define \emph{density bundles} ${\mc{E}}(w,w') := {\mc{E}}(1,0)^w \otimes \overline{{\mc{E}}(1,0)}{}^{w'}$ for any $w, w' \in \R$ such that $w - w' \in \Z$. By definition, it follows that $\overline{{\mc{E}}(w,w')} = {\mc{E}}(w',w)$.
	
	There are  canonical sections ${\bm{\theta}}$ of $T^* {\mc{M}} \otimes {\mc{E}}(1,1)$ and ${\bm{h}}$ of $({H}^{(1,0)})^* \otimes ({H}^{(0,1)})^* \otimes {\mc{E}}(1,1)$ with the property that, for each \emph{CR scale} $s$, that is, a (real-valued) section $0 < s \in \Gamma({\mc{E}}(1,1))$, we have that ${\theta} = s^{-1} {\bm{\theta}}$ is a contact form with Levi form ${h} = s^{-1}  {\bm{h}}$. This reflects the fact that $T \mc{M}/H$ injects into $\mc{E}(1,1)$. The weighted Levi form $\bm{h}$ allows us to identify ${H}^{(1,0)}$ with $({H}^{(0,1)})^*\otimes \mc{E}(1,1)$ and ${H}^{(0,1)}$ with $({H}^{(1,0)})^*\otimes \mc{E}(1,1)$.
	
	One can also understand the relation between densities and contact forms as follows: to each section $\zeta$ of $\mc{C} = \mc{E}(-(m+2),0)$, there is a unique contact form $\theta$ with Reeb vector field $\ell$ for which $\zeta$ is \emph{volume-normalised}, i.e.\
	\begin{align}\label{eq:volnorm}
		\theta \wedge (\d \theta)^m & = \i^{m^2} m! (-1)^q \theta \wedge (\ell \hook \zeta) \wedge (\ell \hook \ol{\zeta}) \, .
	\end{align}
	This equation tells us that at every point $p$ there is a circle of volume-normalised elements of the fibre $\mc{C}_p$ associated to a non-zero element of $\Ann(H_p)$.

	\subsection{Webster connections}
	Each choice of CR scale $s$ with contact form ${\theta}=s^{-1}\bm{\theta}$ and Reeb vector field $\ell$ induces a splitting
	\begin{align*}
		{}^\C T {\mc{M}} =  {}^\C \langle \ell \rangle \oplus {H}^{(1,0)} \oplus {H}^{(0,1)} \, .
	\end{align*}
	We shall adopt the following abstract index notation for CR geometry:
	\begin{align*}
		{\mc{E}}^{\alpha} &:= {H}^{(1,0)} \, , & {\mc{E}}^{\bar{\alpha}} & := {H}^{(0,1)}\, , & {\mc{E}}_0 & := \Ann(H) \, , \\
		{\mc{E}}_{\alpha} & := ({H}^{(1,0)})^* \, , & {\mc{E}}_{\bar{\alpha}} & := ({H}^{(0,1)})^* \, ,
		& {\mc{E}}^0 & := \left(\Ann(H)\right)^* \, .
	\end{align*}
	Indices will be raised and lowered using ${\bm{h}}_{\alpha \bar{\beta}}$, e.g.\ ${v}_{\alpha} = {\bm{h}}_{\alpha \bar{\beta}} {v}^{\bar{\beta}}$. Complex conjugation on ${}^\C {H}$ changes the index type, and, in general, unless stated otherwise, we shall write ${v}^{\bar{\alpha}}$ for $\overline{{v}^{\alpha}}$, and so on. These indices will also be used in a concrete way, in which case they run from $1$ to $m$.

	Complete ${\ell}$ to a frame $( \ell , {e}_\alpha , \overline{{e}}_{\bar{\beta}} )$, $\alpha, \bar{\beta} = 1, \ldots, m$, adapted to $({H}, {J})$, i.e.\ $( {e}_\alpha )$  and $( \overline{{e}}_{\bar{\beta}} )$, $\alpha, \bar{\beta} = 1, \ldots, m$, span ${H}^{(1,0)}$ and ${H}^{(0,1)}$ respectively. Denote by $( \theta , {\theta}^\alpha , \overline{{\theta}}{}^{\bar{\beta}} )$, the coframe dual to $( \ell , {e}_\alpha , \overline{{e}}_{\bar{\beta}} )$, ${\alpha, \bar{\beta} = 1, \ldots , m}$. Then the contact form ${\theta}$ satisfies
	\begin{align*}
		\rm{d} \theta & = \i {h}_{\alpha \bar{\beta}} {\theta}^{\alpha} \wedge \overline{{\theta}}{}^{\bar{\beta}}  \, ,
	\end{align*}
	where we realise the Levi form ${h}_{\alpha \bar{\beta}}=s^{-1} \bm{h}_{\alpha \bar{\beta}}$ of $\theta$ as a Hermitian matrix of functions. We refer to $(\theta^{\alpha})$ as an \emph{admissible} coframe for $(H^{(1,0)})^*$, and to $(\theta, \theta^{\alpha}, \ol{\theta}{}^{\bar{\alpha}})$ as an \emph{adapted coframe} for $(\mc{M},H,J)$. An admissible or adapted coframe is said to be \emph{unitary} if $h_{\alpha \bar{\beta}}$ is diagonal with respect to $(\theta^{\alpha}, \ol{\theta}{}^{\bar{\alpha}})$ with $p$ $+1$-eigenvalues and $q$ $-1$-eigenvalues. In particular, $h_{\alpha \bar{\beta}} = \delta_{\alpha \bar{\beta}}$ in the positive definite case. In dimension three, we shall sometimes qualify such a coframe as being \emph{of the first kind}, to distinguish it from an adapted coframe \emph{of the second kind} as described in Appendix \ref{app:coframe2}.
	
	For each contact form $\theta$, there is a unique linear connection ${\nabla}$ on $T {\mc{M}}$ and associated bundles with the property that, for a chosen adapted coframe $( \theta , {\theta}^{\alpha} , \overline{{\theta}}{}^{\bar{\beta}} )$, ${\nabla}$ satisfies
	\begin{align*}
		{\nabla} \theta & = 0 \, , & {\nabla} {\theta}^\alpha & = - {\Gamma}_{\beta}{}^{\alpha} \otimes {\theta}^{\beta} \, ,  & {\nabla} \ol{{\theta}}{}^{\bar{\alpha}} & = - \Gamma_{\bar{\beta}}{}^{\bar{\alpha}} \otimes \ol{\theta}{}^{\bar{\beta}} \, ,
	\end{align*}
	where ${\Gamma}_{\beta}{}^{\alpha}$ and $\Gamma_{\bar{\beta}}{}^{\bar{\alpha}}$ denote the connection one-form of ${\nabla}$ in that coframe. In addition, ${\nabla}$ is required to preserve the Levi form ${h}_{\alpha \bar{\beta}}$, from which it follows that
	\begin{align*}
		\d {h}_{\alpha \bar{\beta}} - \left( {\Gamma}_{\alpha}{}^{\gamma} {h}_{\gamma\bar{\beta}} + \Gamma_{\bar{\beta}}{}^{\bar{\gamma}} {h}_{\alpha \bar{\gamma}} \right)  & = 0 \, .
	\end{align*}
	For a unitary adapted coframe, we have ${\Gamma}_{\alpha \bar{\beta}} = - {\Gamma}_{\bar{\beta} \alpha}$.

	The torsion of this connection splits into two pieces: the Levi-form ${h}_{\alpha \bar{\beta}}$ and the \emph{Webster} or \emph{pseudo-Hermitian torsion tensor} ${\WbA}_{\alpha \beta}$ with conjugate ${\WbA}_{\bar{\alpha} \bar{\beta}}$. The latter satisfies the symmetries ${\WbA}_{\alpha \beta} = {\WbA}_{(\alpha \beta)}$. They can be identified within the Cartan structure equations:
	\begin{subequations}\label{eq:structure_CR}
		\begin{align}
			\rm{d} \theta & = \i {h}_{\alpha \bar{\beta}} {\theta}^{\alpha} \wedge \overline{{\theta}}{}^{\bar{\beta}}  \, , \\
			\rm{d} {\theta}^{\alpha} & = {\theta}^{\beta} \wedge {\Gamma}_{\beta}{}^{\alpha} +  {\WbA}^{\alpha}{}_{\bar{\beta}}   \theta \wedge \overline{{\theta}}{}^{\bar{\beta}} \, , \\
			\rm{d} \overline{{\theta}}^{\bar{\alpha}} & = \ol{{\theta}}{}^{\bar{\beta}} \wedge {\Gamma}_{\bar{\beta}}{}^{\bar{\alpha}} + {\WbA}^{\bar{\alpha}}{}_{\beta}  \theta \wedge {\theta}^{\beta}  \, .
		\end{align}
	\end{subequations}
	
	The Webster connection $\nabla$ extends to a connection on the density bundles $\mc{E}(w,w')$. Let us pick $\sigma \in \Gamma(\mc{E}(w,0))$, which we assume to be such that $\sigma^{-\tfrac{m+2}{w}} = \theta \wedge \theta^1 \wedge \ldots \wedge \theta^m$ for some adapted coframe $(\theta,\theta^\alpha, \ol{\theta}{}^{\bar{\alpha}})$. Then
	\begin{align}\label{eq:conn1_density}
		\nabla \sigma & = \tfrac{w}{m+2} \Gamma_{\alpha}{}^{\alpha} \sigma \, ,
	\end{align}
	where $\Gamma_{\alpha}{}^{\beta}$ is the connection one-form of $\nabla$ corresponding to $\theta$, and similarly for densities of weight $(0,w)$.
	A useful property of the weighted contact form $\bm{\theta}$ is that it is preserved by any Webster connection $\nabla$, i.e.\ $\nabla \bm{\theta} = 0$. This means that if $\theta = s^{-1} \bm{\theta}$ for some CR scale $s$, then $\nabla s =0$.
	
	\begin{rem}
		For future use, we record the commutators of the Webster connection when $(\mc{M},H,J)$ is of dimension \emph{three} only. For any $\lambda_{\alpha} \in \Gamma(\mc{E}_{\alpha})$, we have
		\begin{subequations}
			\begin{align}
				({\nabla}_{\alpha} {\nabla}^{\alpha} - {\nabla}^{\alpha} {\nabla}_{\alpha} ) {\lambda}_{\gamma} + \i  {\nabla}_0 {\lambda}_{\gamma}  & = - 4 {\Rho} {\lambda}_{\gamma}  \, , \label{eq:SchoutenSc} \\
				({\nabla}_{\alpha} {\nabla}_{0} - {\nabla}_{0} {\nabla}_{\alpha} ) {\lambda}_{\gamma} - {\WbA}_{\alpha \beta} {\nabla}^{\beta}  {\lambda}_\gamma & = -  {\lambda}_{\beta} {\nabla}^{\beta} {\WbA}_{\alpha \gamma}   \, , \\
				({\nabla}^{\alpha} {\nabla}_{0} - {\nabla}_{0} {\nabla}^{\alpha} ) {\lambda}_{\gamma} - {\WbA}^{\alpha \beta} {\nabla}_{\beta}  {\lambda}_\gamma & = ({\nabla}_{\gamma} {\WbA}^{\alpha \delta})  {\lambda}_{\delta}  \, ,
			\end{align}
		\end{subequations}
		where ${\Rho}$ is the \emph{Webster--Schouten scalar}. For any density $f \in \Gamma(\mc{E}(w,w'))$,
		\begin{subequations}\label{eq:CR_com}
			\begin{align}
				({\nabla}_{\alpha} {\nabla}^{\alpha} - {\nabla}^{\alpha} {\nabla}_{\alpha} ) {f} & = \tfrac{4}{3} (w -w') \Rho f - \i {\nabla}_{0} {f} \, , \label{eq:CR_com1} \\
				({\nabla}_{\alpha} {\nabla}_{0} - {\nabla}_{0} {\nabla}_{\alpha} ) {f} & = \tfrac{1}{3} (w -w') {\nabla}^{\beta}{\WbA}_{\alpha \beta} f + {\WbA}_{\alpha}{}^{\bar{\beta}} 
				{\nabla}_{\bar{\beta}} {f} \, . \label{eq:CR_com2}
			\end{align}
		\end{subequations}
		
		Finally, the \emph{Cartan tensor} is the fourth-order CR invariant of weight $(-1,-1)$ defined by
		\begin{align}\label{eq:Cartan_tensor}
			{\CQ}_{\alpha \beta} & := \i {\nabla}_0 {\WbA}_{\alpha \beta} - 2 \i {\nabla}_{\alpha} \CT_{\beta} + 2 {\Rho} {\WbA}_{\alpha \beta}  \, , & \mbox{where
				$\CT_{\alpha} = \tfrac{1}{3} \left( {\nabla}_{\alpha} {\Rho} - \i {\nabla}^{\gamma} {\WbA}_{\gamma \alpha} \right)$.}
		\end{align}
		Its vanishing is equivalent to \emph{CR flatness}, that is, ${\CQ}_{\alpha \beta} = 0$ if and only if $(\mc{M}, {H},{J})$ is locally CR equivalent to the CR three-sphere.
	\end{rem}
	
	\subsection{Gauged partial Webster connections}\label{sec:gauged_Webster}
	In what follows, $\mc{E}_{\bullet}$ will denote any multiple tensor products of $\mc{E}_{\alpha}$ and $\mc{E}_{\bar{\alpha}}$. Choose a Webster connection $\nabla$, and view $\nabla_{\alpha}$ and $\nabla_{\bar{\alpha}}$ as partial connections along sections of $\mc{E}^{\alpha}$ and $\mc{E}^{\bar{\alpha}}$. Then, for any choice of $\xi_{\alpha} \in \Gamma(\mc{E}_{\alpha})$, we define the partial connection
	\begin{align}\label{eq:gpconn}
		\accentset{\xi}{\nabla}_{\alpha} & : \Gamma(\mc{E}_{\bullet}(w,w')) \longrightarrow \Gamma(\mc{E}_{\alpha} \otimes \mc{E}_{\bullet}(w,w')) \quad : \bm{f} \mapsto  \accentset{\xi}{\nabla}_{\alpha} \bm{f} := \nabla_{\alpha} \bm{f} - \i (w - w') \xi_{\alpha} \bm{f} \, .
	\end{align}
	Clearly, $\accentset{\xi}{\nabla}_{\alpha}$ coincides with $\nabla_{\alpha}$ on any section of $\mc{E}_{\bullet}(w,w)$. We shall often use the short-hand $\accentset{\xi}{\nabla}_{\alpha} = \nabla_{\alpha} + \xi_{\alpha}$ to mean \eqref{eq:gpconn}.
	
	\begin{defn}\label{defn:GPconn}
		We shall call the partial connection defined by \eqref{eq:gpconn} the \emph{partial Webster connection gauged by $\xi_{\alpha}$}. If $\xi_{\alpha} = \i \nabla_{\alpha} \log z$ for some nowhere vanishing complex-valued function $z$, we refer to the partial gauge $\xi_{\alpha}$ as being \emph{exact}.
	\end{defn}
	In full analogy, a one-form $\xi$ can be used to gauge a Webster connection as
	\begin{align}\label{eq:gconn}
		\accentset{\xi}{\nabla} & : \Gamma(\mc{E}_{\bullet}(w,w')) \longrightarrow \Gamma(T^* \mc{M} \otimes \mc{E}_{\bullet}(w,w')) \quad : \accentset{\xi}{\nabla} \bm{f} := \nabla \bm{f} - \i (w - w') \xi \bm{f} \, .
	\end{align}
	This full gauged connection will play no r\^{o}le in this section, but it will find applications in Section \ref{sec:Weyl}. Writing $\xi = \xi_{\alpha} \theta^{\alpha} + \xi_{\bar{\alpha}} \ol{\theta}{}^{\bar{\alpha}} + \xi_{0} \theta$, its curvature is found to be
	\begin{align*}
		F & := \d \xi = F_{\alpha \beta}  \theta^{\alpha} \wedge \theta^{\beta} + F_{\bar{\alpha} \bar{\beta}} \ol{\theta}{}^{\bar{\alpha}} \wedge \ol{\theta}{}^{\bar{\beta}} + 2 F_{\alpha \bar{\beta}} \theta^{\alpha} \wedge \ol{\theta}{}^{\bar{\beta}} 
		+ 2 F_{0 \beta} \theta \wedge \theta^{\beta} 
		+ 2 F_{0 \bar{\beta}} \theta \wedge \ol{\theta}{}^{\bar{\beta}} \, ,
	\end{align*}
	where
	\begin{equation}\label{eq:curv_xi0}
		\begin{aligned}
			F_{\alpha \bar{\beta}} & = \tfrac{1}{2} \left( \nabla_{\alpha} \xi_{\bar{\beta}} - \nabla_{\bar{\beta}} \xi_{\alpha} + \i {h}_{\alpha \bar{\beta}} \xi_{0} \right) \, , & 	F_{\alpha \beta} & = \nabla_{[\alpha} \xi_{\beta]} \, , \\
			F_{0 \beta} & = \tfrac{1}{2} \left( \nabla_{0} \xi_{\beta} - \nabla_{\beta} \xi_{0} +  {\WbA}_{\beta}{} ^{\bar{\alpha}} \xi_{\bar{\alpha}} \right) \, .
		\end{aligned}
	\end{equation}
	
	\begin{rem}
		To carry out explicit computations, it is convenient to choose $\sigma \in \Gamma(\mc{E}(1,0))$ with associated Webster connection $\nabla$, i.e.\ $\nabla (\sigma \ol{\sigma})=0$, so that we can trivialise $\bm{f} \in \Gamma(\mc{E}_{\bullet}(w,w'))$ as the complex-valued function $f = \sigma^{-w} \ol{\sigma}^{-w'} \bm{f}$. Then
		\begin{align}\label{eq:gD_density}
			\sigma^{-w} \ol{\sigma}^{-w'}\accentset{\xi}{\nabla} \bm{f} & = \nabla f + (w -w') (\sigma^{-1} \accentset{\xi}{\nabla} \sigma)  f \, .
		\end{align}
	\end{rem}

	The following lemma follows from a straightforward computation:
	\begin{lem}\label{lem:degauge}
		For any two pairs $\sigma, \xi_{\alpha}$ and $\sigma', \xi'_{\alpha}$, where $\sigma, \sigma' \in \Gamma(\mc{E}(1,0))$ and $\xi_{\alpha}, \xi'_{\alpha} \in \Gamma(\mc{E}_{\alpha})$ are related by
		\begin{align} \label{eq:freedom}
			\sigma' & = z \sigma \, , & \xi'_{\alpha} = \xi_{\alpha} + \i \nabla_{\alpha} \log z \, ,
		\end{align}
		for some smooth nowhere vanishing complex-valued function $z$, we have
		\begin{align*}
			{\sigma'}{}^{-1} \accentset{\xi'}{\nabla}_{\alpha} \sigma' & = \sigma^{-1} \accentset{\xi}{\nabla}_{\alpha} \sigma \, .
		\end{align*}
	\end{lem}
	
	\begin{rem}
		Clearly, if the gauge is exact, one can always find a rescaling \eqref{eq:freedom} such that ${\sigma'}{}^{-1} {\nabla}_{\alpha} \sigma' = \sigma^{-1} \accentset{\xi}{\nabla}_{\alpha} \sigma$
	\end{rem}

	\subsection{Transformations under a change of contact form}
	Let $s$ and $\wh{s}$ be two CR scales related by $\wh{s} = \e^{-\varphi} s$, for some smooth function $\varphi$ on $\mc{M}$, so that the corresponding Levi forms are related by  $\wh{h}_{\alpha \bar{\beta}} = \e^{\varphi} h_{\alpha \bar{\beta}}$, and the induced change of adapted coframe is given by
	\begin{align}
		\wh{\theta} & = \e^{\varphi} \theta \, , & 
		\wh{\theta}{}^{\alpha} & = \theta^{\alpha} + \i \Upsilon^{\alpha} \theta \, , & 		\ol{\wh{\theta}}{}^{\bar{\alpha}} & = \ol{\theta}{}^{\bar{\alpha}} - \i \Upsilon{}^{\bar{\alpha}} \theta \, , \label{eq:contact_chg}
	\end{align}
	where $\Upsilon = \d \varphi$. The induced change of Webster connection will be denoted by the short-hand notation
	\begin{align}\label{eq:Webstrnsf}
		\wh{\nabla} & = \nabla + \Upsilon \, ,
	\end{align}
	the precise meaning of which is already given in \cite[equations (2.6) and (2.7), Proposition 2.3]{Gover2005}. The Webster torsion and Webster--Schouten scalar transform as
	\begin{align}\label{eq:A_chg} 
		\wh{{\WbA}}_{\alpha \beta} & = {\WbA}_{\alpha \beta} + \i {\nabla}_{\alpha} {\Upsilon}_{\beta}  - \i {\Upsilon}_{\alpha} {\Upsilon}_{\beta} \, , &
		\wh{{\Rho}} & = {\Rho} - \tfrac{1}{2} \left( {\nabla}^{\alpha} {\Upsilon}_{\alpha} + {\nabla}_{\alpha} {\Upsilon}^{\alpha} \right) - \tfrac{m}{2} {\Upsilon}^{\gamma} {\Upsilon}_{\gamma} \, , 
	\end{align}
	respectively.
	
	\begin{rem}
		Gauged (partial) Webster connections undergo the same transformation as their ungauged counterparts: if $\accentset{\xi}{\nabla}_{\alpha}$ and  $\accentset{\xi}{\wh{\nabla}}_{\alpha}$ are the partial Webster connections gauged by some $\xi_{\alpha} \in \Gamma(\mc{E}_{\alpha})$, and $\nabla$ and $\wh{\nabla}$ are related by \eqref{eq:Webstrnsf}, then
		\begin{align*}
			\accentset{\xi}{\wh{\nabla}}_{\alpha} \bm{f} & = 	\accentset{\xi}{\nabla}_{\alpha} \bm{f} + w \Upsilon_{\alpha} \bm{f} \, ,
		\end{align*}
		for any section $\bm{f}$ of $\mc{E}_{\bullet}(w,w')$.
	\end{rem}

	\subsection{CR densities}
	We now naturally extend the notion of CR functions to densities.
	\begin{defn}
		We shall say that a density $\tau$ of weight $(w,0)$ for any $w \in \Z$ on a contact CR manifold $(\mc{M},H,J)$ is a \emph{CR density} if it satisfies the CR invariant equation
		\begin{align}\label{eq:hol_density}
			\nabla_{\bar{\alpha}} \tau = 0 \, .
		\end{align}
	\end{defn}
	
	\begin{rem}
		This definition should not be confused with that of a \emph{CR scale} introduced earlier, that is, a positive real-valued density of weight $(1,1)$.
	\end{rem}
	
	\begin{rem}
		When $w=0$, we then recover the definition of a CR function. For a strongly independent CR function, we need the additional requirement that $\tau \neq \ol{\tau}$ and $\nabla_{\alpha} \tau \neq 0$.
	\end{rem}
	
	The space of CR densities is clearly closed under multiplication, that is, if $\sigma$ and $\tau$ are two CR densities of weight $(w,0)$ and $(w',0)$ respectively, then $\sigma \tau$ is a CR density of weight $(w+w',0)$. Further, since a section $\varepsilon$ of the canonical bundle is a density of weight $-m-2$, choosing any $-\tfrac{w}{m+2}$-th power of $\varepsilon$ for any $w \in \Z$ yields a section of $\mc{E}(w,0)$. In particular, the structure equations \eqref{eq:structure_CR} together with \eqref{eq:conn1_density} gives a straightforward proof of the next result.
	\begin{lem}\label{lem:hol_sgm}
		Locally, a contact CR manifold admits a closed section of the canonical bundle if and only if it admits CR densities of weight $(w,0)$ with $0 \neq w \in \Z$.
	\end{lem}
	
	Now, since any realisable $(\mc{M},H,J)$ admits nowhere vanishing closed sections of $\mc{C}$, we have:
	\begin{lem}\label{lem:dbar_densities}
		Any realisable contact CR manifold $(\mc{M},H,J)$ admits nowhere vanishing CR densities of weight $(w,0)$ for any $w \in \Z$.
	\end{lem}
	
	\begin{rem}
		Since CR densities and CR functions are on the same footing, Jacobowitz's Theorem \ref{thm:Jacobowitz} feels very natural in that respect: in this light, a (nowhere vanishing) closed section of the canonical bundle is just a (nowhere vanishing) CR density.
	\end{rem}
	
	\begin{rem}
		Lee showed \cite{Lee1988} that locally, the existence of a nowhere vanishing closed section of $\mc{C}$ is essentially equivalent to the existence of a so-called \emph{pseudo--Einstein structure}, that is, a pseudo-Hermitian structure whose \emph{Webster--Schouten tensor} is proportional to the Levi form (but not necessarily by a constant factor).
	\end{rem}

	\subsection{Complex infinitesimal transverse symmetries}
	It is straightforward to check that for any density $\sigma$ of weight $(1,w)$, $w \in \Z$, the equation
	\begin{align}
		\nabla_{\alpha} \nabla_{\beta} \sigma + \i {\WbA}_{\alpha \beta} \sigma & = 0 \, , \label{eq:exactWW}
	\end{align}
	is CR invariant. Note that the first term of this expression is automatically symmetric in its indices since by \cite[Proposition 2.2]{Gover2005}, $\nabla_{[\alpha} \nabla_{\beta]} \sigma = 0$ for any density $\sigma$.
	
	\begin{rem}\label{rem:BGG}
		In the language of parabolic Cartan geometries, a solution $\sigma$ to \eqref{eq:exactWW} lies in the kernel of a so-called first \emph{Bernstein--Gel'fand-Gel'fand (BGG) operator} \cite{Cap2001}, which occurs as part of a sequence of vector bundles associated to irreducible representations of a parabolic subgroup of $\SU(p+1,q+1)$. In the homogeneous (flat) case and $w \geq 0$, this sequence is in fact a resolution of a vector bundle associated to an irreducible $\SU(p+1,q+1)$ representation, which is referred to, more generally, as a \emph{tractor bundle}. In the holomorphic setting and $m=1$, it appears in \cite[Section 8.2]{Baston1989}. This machinery was first conceived in \cite{Bernstein1975} in the context of simple complex Lie algebras.
	\end{rem}
	
	In the case where $w=1$, the geometric interpretation of its solutions is given in the next proposition --- see also Remark \ref{rem:Curry-Ebenfelt}.
	\begin{prop}\label{prop:complexTSym}
		Any nowhere vanishing solution $\sigma \in \Gamma(\mc{E}(1,1))$ to \eqref{eq:exactWW} is equivalent to the existence of a complex infinitesimal transverse symmetry of $H^{(1,0)}$.
	\end{prop}
	
	\begin{proof}
		In a local adapted frame $(\ell, e_{\alpha}, \ol{e}_{\bar{\alpha}})$, a vector field $v$ on $(\mc{M},H,J)$ can be expressed as
		\begin{align*}
			v & = v^{0} \ell + v^{\alpha} e_{\alpha} + v^{\bar{\alpha}} \bar{e}_{\bar{\alpha}} \, ,
		\end{align*}
		for complex-valued $v^{0} \in \Gamma(\mc{E})$, $v^{\alpha} \in \Gamma(\mc{E}^{\alpha})$ and $v^{\bar{\alpha}} \in \Gamma(\mc{E}^{\bar{\alpha}})$. Note that we are \emph{not} assuming that $v^{\bar{\alpha}}$ is the complex conjugate of $v^{\alpha}$ in general. A straightforward computation shows that $v$ preserves $H^{(1,0)}$, i.e.\ $\mathsterling_{v} e_{\alpha} \in \Gamma(H^{(1,0)})$, if and only if
		\begin{subequations}
			\begin{align}
				v_{\alpha} - \i \nabla_{\alpha} v^{0} & = 0 \, , \label{eq:infsym1} \\
				\nabla_{\alpha} v_{\beta} + \i \WbA_{\alpha \beta} & = 0 \, . \label{eq:infsym2}
			\end{align}
		\end{subequations}
		The condition that $v$ be transverse to $H$ means that $v^{0}$ is nowhere vanishing. Setting $\sigma := v^{0} s$ where $s$ is the CR scale preserved by $\nabla$, we see that $\sigma$ is nowhere vanishing and satisfies \eqref{eq:exactWW}.
		
		The converse is clearly true: let $\sigma \in \Gamma(\mc{E}(1,1))$ be a nowhere vanishing solution of \eqref{eq:exactWW}, choose a CR scale $s$ with Webster connection $\nabla$, and set $v^0 := \sigma s^{-1}$. Then $v^0$ is nowhere vanishing, and $v_{\alpha} := \i \nabla_{\alpha} v^{0} $ satisfies \eqref{eq:infsym2} by virtue of \eqref{eq:infsym1}. Defining $v := v^0 \ell + v^{\alpha} e_{\alpha} + v^{\bar{\alpha}} \ol{e}_{\bar{\alpha}}$ for \emph{any} choice of section $v^{\alpha}$ of $\mc{E}^{\alpha}$, we can now conclude that $v$ is a complex infinitesimal transverse symmetry of $H^{(1,0)}$.
	\end{proof}	
	
	\begin{rem}
		In the special case where a nowhere vanishing density $\sigma$ of weight $(1,1)$ is \emph{real-valued}, i.e.\ $\sigma$ is a CR scale, and solves \eqref{eq:exactWW}, then the Reeb vector field defined by $\sigma$ is an infinitesimal transverse symmetry of $(\mc{M},H,J)$. In this case the assumption that $(H,J)$ is contact, together with the fact that $\sigma$ is nowhere vanishing, forces the symmetry of $(H,J)$ to be transverse. We shall nevertheless retain the terminology. Such solutions have already been investigated in \cite{Cap2008a,Cap2008} in the context of BGG sequences.
		
		On the other hand, the correspondence between nowhere vanishing complex-valued densities of weight $(1,1)$ satisfying \eqref{eq:exactWW} and complex transverse symmetries is not one-to-one in general, since the $(1,0)$-part of the latter is completely arbitrary. This is to be expected since $H^{(1,0)}$ is involutive. Thus, if one solution exists, so do infinitely many.
	\end{rem}
	
	\begin{rem}
		If we drop the assumption that our real-valued density $\sigma$ is nowhere vanishing, then the associated symmetry will be transverse only off the zero set of $\sigma$ as explicit examples in \cite{Curry2021} show. This is essentially a consequence of the theory of curved orbit decompositions enunciated in \cite{Cap2014}. This also applies to complex-valued densities.
	\end{rem}
	
	In the context of contact CR manifolds and in the light of Proposition \ref{prop:complexTSym}, we can re-express Jacobowitz's Proposition \ref{prop:Jacobowitz} in the following terms: 
	\begin{prop}\label{prop:New_Jacobowitz}
		Locally, a contact CR manifold $(\mc{M},H,J)$ is realisable if and only if it admits a nowhere vanishing solution $\sigma \in \Gamma(\mc{E}(1,1))$ to  \eqref{eq:exactWW}. 
	\end{prop}
	
	\begin{rem}\label{rem:Curry-Ebenfelt}
		In dimension three, complex-valued solutions to \eqref{eq:exactWW} of weight $(1,1)$ were investigated in \cite{Curry2019a} in the context of \emph{obstruction-flat} compact strictly pseudo-convex CR three-manifolds. There, the authors proved results analogous to Proposition \ref{prop:complexTSym} and \ref{prop:New_Jacobowitz}. In addition, they showed that such solutions correspond to ambient holomorphic vector fields on the pseudo-convex side, smooth up to the CR manifold.
	\end{rem}

	\subsection{Webster--Weyl structures}
	\subsubsection{General odd dimensions}
	In this section, we introduce a CR invariant structure that may be thought of as a non-linear analogue to \eqref{eq:exactWW}, and will play a key r\^{o}le in the subsequent sections. We start with a simple lemma, whose proof follows from the transformation rule \eqref{eq:A_chg} of the Webster torsion tensor, and the fact \cite[equation (2.7)]{Gover2005} that $\wh{\nabla}_{\alpha} \lambda_{\beta} = \nabla_{\alpha} \lambda_{\beta} - 2 \Upsilon_{(\alpha} \lambda_{\beta)}$ for any $\lambda_{\alpha} \in \Gamma(\mc{E}_{\alpha})$.
	\begin{lem}
		Let ${\nabla}$ and $\wh{\nabla}$ be two Webster connections on a contact CR manifold $(\mc{M},H,J)$ related by \eqref{eq:Webstrnsf}. Then a section $\lambda_{\alpha}$ of $\mc{E}_{\alpha}$ is a solution to
		\begin{align}\label{eq:Webster--Weyl}
			{\nabla}_{(\alpha} {\lambda}_{\beta)} -  \i  {\lambda}_{\alpha} {\lambda}_{\beta}  - {\WbA}_{\alpha \beta} & = 0 \, ,
		\end{align}
		if and only if 
		\begin{align}\label{eq:lbd_trnsf}
			\wh{{\lambda}}_{\alpha} = {\lambda}_{\alpha} + \i {\Upsilon}_{\alpha} \, ,
		\end{align}
		is a solution to
		\begin{align*}
			\wh{\nabla}_{(\alpha} \wh{\lambda}_{\beta)} -  \i  \wh{\lambda}_{\alpha} \wh{\lambda}_{\beta}  - \wh{\WbA}_{\alpha \beta} & = 0 \, .
		\end{align*}
	\end{lem}

	Considering the close resemblance of \eqref{eq:Webster--Weyl} to the Einstein--Weyl equation, we make the following definition.
	\begin{defn}
		A \emph{Webster--Weyl} structure on $(\mc{M},H,J)$ is an equivalent class $[\nabla_{\alpha},\lambda_{\alpha}]$ of pairs $(\nabla_{\alpha},\lambda_{\alpha})$, where $\nabla_{\alpha}$ is a (partial) Webster connection and $\lambda_{\alpha} \in \Gamma(\mc{E}_{\alpha})$, with the property that these satisfy the \emph{Webster--Weyl equation} \eqref{eq:Webster--Weyl}, and $(\wh{\nabla}_{\alpha}, \wh{\lambda}_{\alpha}) \sim (\nabla_{\alpha},\lambda_{\alpha})$ whenever $\nabla_{\alpha}$ and $\wh{\nabla}_{\alpha}$ related by \eqref{eq:Webstrnsf}, and $\lambda_{\alpha}$ and $\wh{\lambda}_{\alpha}$ by \eqref{eq:lbd_trnsf}.
	\end{defn}
	
	\begin{rem}
		It is not difficult to check that equation \eqref{eq:exactWW} for a nowhere vanishing \emph{real} density $\sigma \in \Gamma(\mc{E}(1,1))$ is a particular case of \eqref{eq:Webster--Weyl} where, for every choice of representative $[\nabla_{\alpha},\lambda_{\alpha}]$, we have that $\lambda_{\alpha} = \i \nabla_{\alpha} \psi$ for some real-valued smooth function $\psi$ on $\mc{M}$. In particular, there is a representative for which $\lambda_{\alpha} = 0$, which tells us that the Reeb vector field corresponding to $\sigma$ is an infinitesimal symmetry of $(\mc{M},H,J)$.
	\end{rem}
	
	There is an analogue of equation \eqref{eq:exactWW}, closely related to \eqref{eq:Webster--Weyl} for gauged partial connections: Let $\accentset{\xi}{\nabla}_{\alpha}$ be a partial Webster connection gauged by some $\xi_{\alpha} \in \Gamma(\mc{E}_{\alpha})$. Then
	\begin{align}\label{eq:gexactWW}
		\accentset{\xi}{\nabla}_{(\alpha} \accentset{\xi}{\nabla}_{\beta)} \sigma + \i {\WbA}_{\alpha \beta} \sigma & = 0 \, ,
	\end{align}
	is a well-defined CR invariant on densities of weight $(1,w)$. In fact, CR invariance simply follows from the following easy observation:
	\begin{lem}\label{lem:ex-nexWW}
		Suppose that for any $w \in \Z$,  $\sigma \in \Gamma(\mc{E}(1,w))$ is a nowhere vanishing solution to \eqref{eq:gexactWW}. Then
		\begin{align}\label{eq:density_gauge}
			\lambda_{\alpha} = \i \sigma^{-1} \accentset{\xi}{\nabla}_{\alpha} \sigma = \i \sigma^{-1} \nabla_{\alpha} \sigma + \xi_{\alpha} \, ,
		\end{align}
		satisfies the Webster--Weyl equation \eqref{eq:Webster--Weyl}.
		
		Conversely, suppose that $\lambda_{\alpha}$ solves \eqref{eq:Webster--Weyl}. Fix some nowhere vanishing $\sigma \in \mc{E}(1,w)$, $w \in \Z$, and define a partial gauge $\xi_{\alpha} \in \Gamma(\mc{E}_{\alpha})$ by \eqref{eq:density_gauge}. Then $\sigma$ satisfies \eqref{eq:gexactWW}.
	\end{lem}
	Thus, we can interpret equation \eqref{eq:gexactWW} in three different ways:
	\begin{enumerate}
		\item as an equation on the density $\sigma$ for $\xi_{\alpha}$ fixed;
		\item as an equation on the partial gauge $\xi_{\alpha}$ for $\sigma$ fixed;
		\item as an equation on the pair $(\sigma, \xi_{\alpha})$ subject to the freedom \eqref{eq:freedom} since by Lemma \ref{lem:degauge}, it leaves $\lambda_{\alpha}$ as defined by \eqref{eq:density_gauge} unchanged.
	\end{enumerate}

	A consequence of Lemma \ref{lem:degauge} is the following:
	\begin{lem}
		For any nowhere vanishing complex-valued function $z$, any solution $\sigma \in \Gamma(\mc{E}(1,0))$ to \eqref{eq:exactWW} yields a solution $z \sigma$ to \eqref{eq:gexactWW}, where the partial gauge is given by $\xi_{\alpha} =\i \nabla_{\alpha} \log z$.
		
		Suppose that $\xi_{\alpha} = \i \nabla_{\alpha} \log z$ for some nowhere vanishing complex-valued function $z$. Then any solution to \eqref{eq:gexactWW} gives rise to a solution to \eqref{eq:exactWW}.
	\end{lem}
	
	Once a solution to \eqref{eq:exactWW} is known for some gauge, it is easy to produce more by using CR densities of weight $(1,0)$ as `raising' or `lowering' operators:
	\begin{lem}\label{lem:lowering}
		Suppose that $(\mc{M},H,J)$ admits a nowhere vanishing CR density $\tau$ of weight $(w,0)$ for some $w \in \Z$. Then any solution $s$ of weight $(1,w')$ to \eqref{eq:exactWW} for some $w' \in \Z$ gives rise to a solution $\sigma=s \ol{\tau}^{-1}$ of weight $(1,w'-w)$ to \eqref{eq:exactWW}.
	\end{lem}
	
	It thus follows from Lemmata \ref{lem:lowering} and \ref{lem:dbar_densities} that for realisable $(\mc{M},H,J)$, densities of weight $(1,0)$ that satisfy \eqref{eq:exactWW} exist in plentiful supply. We note that the geometric interpretation of a solution to \eqref{eq:exactWW} crucially depends on the weight of the density since in general only weight $(1,1)$ implies realisability. However, Jacobowitz's Theorem \ref{thm:Jacobowitz} already hints at a way to remedy this issue.
	\begin{thm}\label{thm:real2dens}
		Locally, a contact CR manifold $(\mc{M},H,J)$ is realisable if and only if it admits two nowhere vanishing densities $\sigma$ and $\tau$ of weight $(1,0)$ that satisfy
		\begin{subequations}
			\begin{align}
				\nabla_{\alpha} \nabla_{\beta} \sigma + \i \WbA_{\alpha \beta} \sigma & = 0 \, , \label{eq:linWW1} \\
				\nabla_{\bar{\alpha}} \tau & = 0 \, , \label{eq:linWW2}
			\end{align}
		\end{subequations}
		respectively.
	\end{thm}
	
	\begin{proof}
		Suppose that $(\mc{M},H,J)$ is realisable. Then Lemma \ref{lem:dbar_densities} provides a nowhere vanishing solution $\tau$ to \eqref{eq:linWW2}, and Proposition \ref{prop:New_Jacobowitz} provides a nowhere vanishing solution to \eqref{eq:exactWW}. Hence, by Lemma \ref{lem:lowering} we obtain a nowhere vanishing solution to \eqref{eq:linWW1}.
		
		Conversely, we apply Lemma \ref{lem:lowering} to our nowhere vanishing solutions $\sigma$ and $\tau$ to \eqref{eq:linWW1} and \eqref{eq:linWW2} respectively, to produce a nowhere vanishing solution $\sigma \ol{\tau}^{-1} \in \Gamma(\mc{E}(1,1))$ to \eqref{eq:exactWW}. Realisability then follows by Proposition \ref{prop:New_Jacobowitz}.
	\end{proof}

	\begin{rem}\label{rem:CR-Eins}
		In the special case, where $\tau=\sigma$ in Theorem \ref{thm:real2dens}, the pair of equations \eqref{eq:linWW1} and \eqref{eq:linWW2} defines what is referred to as a \emph{CR--Einstein structure} in \cite{Cap2008}, and arises from a BGG operator --- see \cite{Matsumoto2022} for (more general) almost CR manifolds. In this case, the solution $\sigma$ defines a pseudo-Hermitian structure $\theta$ for which
		\begin{align}\label{eq:CREin}
			\WbA_{\alpha \beta} & = 0 \, , & \left( \Rho_{\alpha \bar{\beta}} \right)_{\circ} & = 0 \, , & \nabla_{\alpha} \Rho & = 0 \, .
		\end{align}
		In dimension three, any solution thereof immediately tells us that $(\mc{M},H,J)$ is CR--flat. Allowing $\sigma$ to have a non-empty zero set $\mc{Z}$ also leads to a pseudo-Hermitian structure satisfying \eqref{eq:CREin}, but only off $\mc{Z}$.
	\end{rem}
	
	The assumptions of Theorem \ref{thm:real2dens} can in fact be weakened:
	\begin{thm}\label{thm:real2dens_d3}
		Locally, a contact CR manifold $(\mc{M},H,J)$ is realisable if and only if it admits two nowhere vanishing densities $\sigma$ and $\tau$ of weight $(1,0)$ that satisfy
		\begin{subequations}
			\begin{align}
				\accentset{\xi}{\nabla}_{\alpha} \accentset{\xi}{\nabla}_{\beta} \sigma + \i \WbA_{\alpha \beta} \sigma & = 0 \, , \label{eq:linWW1d3} \\
				\accentset{\xi}{\nabla}_{\bar{\alpha}} \tau & = 0 \, , \label{eq:linWW2d3}
			\end{align}
		\end{subequations}
		where $\accentset{\xi}{\nabla}_{\alpha} = \nabla_{\alpha} + \xi_{\alpha}$ for some gauge $\xi_{\alpha} \in \Gamma(\mc{E}_{\alpha})$.
	\end{thm}

	\begin{proof}
		We work locally. Suppose that $\sigma$ and $\tau$ are nowhere vanishing densities of weight $(1,0)$ that solve \eqref{eq:linWW1d3} and \eqref{eq:linWW2d3} respectively. Then
		\begin{align*}
			\nabla_{\alpha} \nabla_{\beta} (\sigma \ol{\tau}) + \i \WbA_{\alpha \beta} (\sigma \ol{\tau}) & = \accentset{\xi}{\nabla}_{\alpha} \accentset{\xi}{\nabla}_{\beta} (\sigma \ol{\tau}) + \i \WbA_{\alpha \beta} (\sigma \ol{\tau}) & \mbox{since $\sigma \ol{\tau} \in \Gamma(\mc{E}(1,1))$,} \\
			& =  \ol{\tau} \left( \accentset{\xi}{\nabla}_{\alpha} \accentset{\xi}{\nabla}_{\beta} \sigma + \i \WbA_{\alpha \beta} \sigma \right) & \mbox{by \eqref{eq:linWW2d3},} \\
			& = 0 & \mbox{by \eqref{eq:linWW1d3}.}
		\end{align*}
		For the last equality, we used the fact that $\tau$ is nowhere vanishing. Hence, by Proposition \ref{prop:New_Jacobowitz}, we conclude that  $(\mc{M},H,J)$ is locally realisable.
		
		The converse is clear by Theorem \ref{thm:real2dens}.
	\end{proof}
	
	\begin{rem}
		Note that for equation \eqref{eq:linWW2d3} to hold, the CR invariant condition $\nabla_{[\alpha} \xi_{\beta]} = 0$ must hold too, and this in turn implies that equation \eqref{eq:linWW1d3} is symmetric in its indices, since $\accentset{\xi}{\nabla}_{[\alpha} \accentset{\xi}{\nabla}_{\beta]} f = - \i (w -w') (\nabla_{[\alpha} \xi_{\beta]}) f$ for any density $f \in \Gamma(\mc{E}(w,w'))$. Evidently, this condition is trivially satisfied when $m=1$.
	\end{rem}
	
	\subsubsection{Dimension three}\label{sec:CR3}
	From now on, we assume that $(\mc{M},H,J)$ has dimension three. Then the tensor product of any number of copies of ${H}^{(1,0)}$ and ${H}^{(0,1)}$ is of complex rank $1$, so symmetrisation is a vacuous operation, and we have that ${H}^{(1,0)}\otimes{H}^{(0,1)} \cong {\mc{E}}(1,1)$, since for any ${A}^{\alpha} \in \Gamma(\mc{E}^{\alpha})$, ${B}^{\bar{\alpha}} \in \Gamma(\mc{E}^{\bar{\alpha}})$, we can write ${A}^{\alpha} {B}^{\bar{\beta}} = {A}^{\gamma} {B}_{\gamma} {\bm{h}}^{\alpha \bar{\beta}}$. While the use of indices, which concretely only take the numerical value $1$, may appear to be of rather limited utility, they turn out to be extremely useful for book keeping and ultimately make the geometric interpretation of our computations much more transparent, and the relation to the general case more direct.
	
	We now return to the problem of characterising the existence of a CR function. 	Let us choose an arbitrary local adapted coframe $(\theta, \theta^{\alpha}, \ol{\theta}{}^{\bar{\alpha}})$ on $(\mc{M},H,J)$ with dual $(\ell,e_{\alpha}, \ol{e}_{\bar{\alpha}})$. Starting in greater generality, let ${\lambda}_{\alpha}$ be an arbitrary section of ${\mc{E}}_{\alpha}$, and set
	\begin{align}\label{eq:om_lam_thet}
		{\omega} & = \theta \, , &
		{\omega}^{\alpha} & = {\theta}^{\alpha} + {\lambda}^{\alpha} \theta \, , &
		\ol{{\omega}}{}^{\bar{\alpha}} & = \ol{{\theta}}{}^{\bar{\alpha}} + {\lambda}^{\bar{\alpha}} \theta \, .
	\end{align}
	Note that by \eqref{eq:lbd_trnsf} and \eqref{eq:contact_chg}, ${\omega}^{\alpha}$ does not depend on the choice of contact form.
	Using the structure equations \eqref{eq:structure_CR}, we then find
	\begin{subequations}
		\begin{align}
			\d \omega & = \i {h}_{\alpha \bar{\beta}} {\omega}^{\alpha} \wedge \overline{{\omega}}{}^{\bar{\beta}} +  \omega \wedge \left( \i {\lambda}_{\alpha} {\omega}^{\alpha}  - \i {\lambda}_{\bar{\alpha}} {\omega}^{\bar{\alpha}} \right) \, , \label{eq:dom0} \\
			\d {\omega}^{\alpha}
			& = \left( {\Gamma}_{\bar{\beta} \gamma}{}^{\gamma} + \i {\lambda}_{\bar{\beta}} \right) {\omega}^{\alpha} \wedge \overline{{\omega}}{}^{\bar{\beta}} + \left( {\WbA}^{\alpha}{}_{\bar{\beta}} - {\nabla}_{\bar{\beta}} {\lambda}^{\alpha} -   \i {\lambda}_{\bar{\beta}} {\lambda}^{\alpha} \right)  \omega \wedge \overline{{\omega}}{}^{\bar{\beta}} \label{eq:domalf} \\
			& \qquad \qquad \qquad  + \left(  - {\Gamma}_{0 \delta}{}^{\delta} \delta_{\beta}^{\alpha} - {e}_{\beta}( {\lambda}^{\alpha} ) + \lambda^{\bar{\gamma}} \Gamma_{\bar{\gamma} \delta}{}^{\delta}  \delta_{\beta}^{\alpha} +  \i {\lambda}^{\alpha}  {\lambda}_{\beta} \right) \omega \wedge {\omega}^{\beta} \, , \nonumber \\
			\d \ol{{\omega}}^{\bar{\alpha}}
			& = \left( {\Gamma}_{\beta  \bar{\gamma}}{}^{\bar{\gamma}} - \i {\lambda}_{\beta} \right) \ol{{\omega}}^{\bar{\alpha}} \wedge {\omega}{}^{\beta} + \left( {\WbA}^{\bar{\alpha}}{}_{\beta} - {\nabla}_{\beta} {\lambda}^{\bar{\alpha}} + \i {\lambda}_{\beta} {\lambda}^{\bar{\alpha}} \right)  \omega \wedge {\omega}{}^{\beta}
			\label{eq:domalfb} \\
			& \qquad \qquad \qquad  + \left(  - \Gamma_{0 \bar{\delta}}{}^{\bar{\delta}} \delta_{\bar{\beta}}^{\bar{\alpha}} - \ol{e}_{\bar{\beta}}( {\lambda}^{\bar{\alpha}} ) + \lambda^{\gamma} \Gamma_{\gamma \bar{\delta}}{}^{\bar{\delta}} \delta_{\bar{\beta}}^{\bar{\alpha}} -  \i {\lambda}^{\bar{\alpha}}  {\lambda}_{\bar{\beta}} \right) \omega \wedge \ol{{\omega}}^{\bar{\beta}} \, . \nonumber
		\end{align}
	\end{subequations}
	In particular,
	\begin{align*}
		{\omega}^{\gamma} \wedge \d {\omega}^{\alpha} & = \left( {\nabla}_{\bar{\beta}} {\lambda}^{\alpha} +  \i {\lambda}_{\bar{\beta}} {\lambda}^{\alpha} - {\WbA}^{\alpha}{}_{\bar{\beta}}\right)  \omega \wedge {\omega}^{\gamma} \wedge \overline{{\omega}}{}^{\bar{\beta}} \, ,
	\end{align*}
	and similarly for its complex conjugate. We thus conclude:
	\begin{align*}
		{\omega}^{\alpha} \wedge \d {\omega}^{\beta} & = 0 & & \Longleftrightarrow & &    {\nabla}_{\alpha} {\lambda}_{\beta} - \i  {\lambda}_{\alpha} {\lambda}_{\beta} - {\WbA}_{\alpha \beta} = 0 \, .
	\end{align*}
	Further, we have
	\begin{align*}
		\omega^{\alpha} \wedge \ol{\omega}{}^{\bar{\beta}} & = {\theta}^{\alpha} \wedge \ol{{\theta}}{}^{\bar{\beta}} + \theta \wedge \left( {\lambda}^{\alpha} \ol{{\theta}}{}^{\bar{\beta}} - {\lambda}^{\bar{\alpha}} \ol{{\theta}}{}^{\bar{\beta}} \right) \neq 0 \, .
	\end{align*}
	Lemma \ref{lem:HLN} now concludes the proof of:
	\begin{thm}\label{thm:Webster--Weyl}
		Locally, a contact CR three-manifold admits a strongly independent CR function if and only if it admits a Webster--Weyl structure.
	\end{thm}
	
	This, together with Lemma \ref{lem:hol_sgm}, allows us to recast Jacobowitz's Theorem \ref{thm:Jacobowitz} in the following terms:
	\begin{thm}\label{thm:New_Jacobowitz}
		Locally, a smooth contact CR three-manifold is realisable if and only if it admits a Webster--Weyl structure and a nowhere vanishing CR density of weight $(1,0)$.
	\end{thm}
	
	\begin{rem}
		It is interesting to note that in dimension three, if one considers equation \eqref{eq:exactWW} alone, realisability can only be guaranteed if the weight of $\sigma$ is $(1,1)$. For any other weight $(1,w)$ with $w \neq 1$, we must content ourselves with only one strongly independent CR function by Theorem \ref{thm:Webster--Weyl} via Lemma \ref{lem:ex-nexWW} with $\xi_{\alpha}=0$.
	\end{rem}

	In the following proposition, we translate the condition of functional independence of two CR functions into a statement on their Webster--Weyl structures.
	\begin{prop}
		Let $(\mc{M},H,J)$ be a contact CR three-manifold. Then any two CR functions are independent if and only if their corresponding Webster--Weyl structures $[\nabla_{\alpha},\lambda_{\alpha}]$ and $[\nabla_{\alpha},\lambda_{\alpha}']$ differ by a nowhere vanishing $(1,0)$-form, i.e.\ $\lambda'_{\alpha} = \lambda_{\alpha} + \mu_{\alpha}$ for some nowhere vanishing $\mu_{\alpha}$.
		
		In particular, given a Webster--Weyl structure $(\nabla_{\alpha},\lambda_{\alpha})$, any other $(\nabla_{\alpha},\lambda_{\alpha}')$ arises from a solution $\mu_{\alpha}$ to the CR invariant equation
		\begin{align*}
			\nabla_{\alpha} \mu_{\beta} - 2 \i \lambda_{\alpha} \mu_{\beta} - \i \mu_{\alpha} \mu_{\beta} & = 0 \, ,
		\end{align*}
		and setting $\lambda'_{\alpha} = \lambda_{\alpha} + \mu_{\alpha}$.
	\end{prop}
	
	\begin{proof}
		Let $\omega^{1}$ and ${\omega'}{}^{1}$ be two forms annihilated by $H^{(0,1)}$ and satisfying $\omega^{1} \wedge \d \omega^{1} = 0$ and ${\omega'}{}^{1} \wedge \d {\omega'}{}^{1} = 0$, so that by Lemma \ref{lem:HLN}, $\omega^1 = z^{-1} \d w$ and ${\omega'}{}^1 = {z'}{}^{-1} \d w'$ for some complex-valued functions $z$, $w$, $z'$ and $w'$ on $\mc{M}$ with $z$ and $z'$ being nowhere vanishing. With no loss, we may choose an adapted coframe $(\theta, \theta^{1}, \ol{\theta}{}^{\bar{1}})$ and, possibly after some rescaling of $\omega^1$, write $\omega^1 = \theta^{1} + \lambda^{1} \theta$ and ${\omega'}{}^1 = \theta^{1} + {\lambda'}^{1} \theta$ for some solutions $\lambda_{\alpha}$ and $\lambda'_{\alpha}$ to the Webster--Weyl equation \eqref{eq:Webster--Weyl}. The condition that the CR functions $w$ and $w'$ be independent is then equivalent to $\omega^1 \wedge {\omega'}{}^1 = ( \lambda^{1} - {\lambda'}{}^{1}) \theta \wedge \theta^1 \neq 0$, from which the result follows immediately.
	\end{proof}
	
	As it will next emerge, in dimension three, one can always `linearise' a solution to the Webster--Weyl equation, in the sense of turning a one-form solving \eqref{eq:Webster--Weyl} into a density of weight $(1,-2)$ solving \eqref{eq:exactWW}:
	\begin{thm}\label{thm:lambda_exact}
		Let $(\mc{M},H,J)$ be a contact CR three-manifold endowed with a Webster--Weyl structure $[\nabla_{\alpha},\lambda_{\alpha}]$. Then locally, there exists a nowhere vanishing smooth density $\sigma$ of weight $(1,0)$ such that for any choice of representative $(\nabla_{\alpha}, \lambda_{\alpha})$, we have
		\begin{align}\label{eq:lambdsig}
			\lambda_{\alpha} & = \i \sigma^{-1} \ol{\sigma}^{2} \nabla_{\alpha} \left( \sigma \ol{\sigma}^{-2} \right) \, ,
		\end{align}
		or equivalently,
		\begin{align*}
			\lambda_{\alpha} & = \i \sigma^{-1} \accentset{\xi}{\nabla}_{\alpha} \sigma := \i \sigma^{-1} \nabla_{\alpha} \sigma + \xi_{\alpha} \, , & & \mbox{where} & \xi_{\alpha} & := - 2 \i \ol{\sigma}^{-1} \nabla_{\alpha} \ol{\sigma} \, ,
		\end{align*}
		with the property that if $(\theta, \theta^{1}, \ol{\theta}{}^{\bar{1}})$ is the unitary adapted coframe determined by $\sigma$, i.e.\ $\sigma^{-3} = \theta \wedge \theta^{1}$, and $(\accentset{(\theta)}{\nabla}_{\alpha},\accentset{(\theta)}{\lambda}_{\alpha})$ is the representative of $[\nabla_{\alpha},\lambda_{\alpha}]$ determined by $\theta = (\sigma \ol{\sigma})^{-1} \bm{\theta}$, then the one-form $\omega^{1} = \theta^{1} + \lambda^{1} \theta$ is given by $\omega^{1} = \d w$ for some CR function $w$, and
		\begin{align}\label{eq:divlambsig}
			\accentset{(\theta)}{\nabla}_{\alpha} \accentset{(\theta)}{\lambda}^{\alpha} - \i \accentset{(\theta)}{\lambda}_{\alpha} \accentset{(\theta)}{\lambda}^{\alpha} + 3 \sigma^{-1} \accentset{(\theta)}{\nabla}_{0} \sigma & = 0 \, .
		\end{align}
		In particular, $\sigma \ol{\sigma}^{-2} \in \Gamma ( \mc{E}(1,-2) )$ satisfies
		\begin{align}\label{eq:WWlin}
			\nabla_{\alpha} \nabla_{\beta} \left( \sigma \ol{\sigma}^{-2} \right) + \i \WbA_{\alpha \beta} \left( \sigma \ol{\sigma}^{-2} \right) & = 0 \, .
		\end{align}
		
		Any other density $\wh{\sigma}$ of weight $(1,0)$ related to $\sigma$ by
		\begin{align}\label{eq:z_change}
			\wh{\sigma} = z^{-\tfrac{3}{2}} \ol{z}^{-\tfrac{1}{2}} \sigma \, ,
		\end{align}
		for some smooth CR function $z$ satisfying $\nabla_{0} z - \lambda^\alpha \nabla_{\alpha} z = 0$, will share the same property as $\sigma$.
	\end{thm}
	
	\begin{proof}
		Since $(\mc{M},H,J)$ admits a Webster--Weyl structure, by Theorem \ref{thm:Webster--Weyl}, it locally admits a CR function $w$, and we can find a unitary adapted coframe $(\theta, \theta^{1}, \ol{\theta}{}^{\bar{1}})$ such that $(\omega, \omega^{1}, \ol{\omega}^{\bar{1}})$ as defined by \eqref{eq:om_lam_thet} is an adapted coframe of the second type, as described in Appendix \ref{app:coframe2}, with the property that $\omega^{1} = \d w$. This choice picks out a representative $(\accentset{(\theta)}{\nabla}_{\alpha},\accentset{(\theta)}{\lambda}_{\alpha})$ of $[\nabla_{\alpha},\lambda_{\alpha}]$. We now return to equation \eqref{eq:domalf}, and to emphasise the dependence on the choice of frame, we adorn the geometric quantities with $\accentset{(\theta)}{\cdot}$. Then $\d \omega^{1} = 0$ implies
		\begin{subequations}
			\begin{align}
				\accentset{(\theta)}{\nabla}_{\bar{\beta}} \accentset{(\theta)}{\lambda}^{\alpha} + \i \accentset{(\theta)}{\lambda}_{\bar{\beta}} \accentset{(\theta)}{\lambda}^{\alpha} - \accentset{(\theta)}{\WbA}^{\alpha}{}_{\bar{\beta}} & = 0 \, , \label{eq:domalf0} \\
				\accentset{(\theta)}{\Gamma}_{\bar{\beta} \gamma}{}^{\gamma} + \i \accentset{(\theta)}{\lambda}_{\bar{\beta}} & = 0 \, , \label{eq:domalf1} \\
				- \accentset{(\theta)}{\Gamma}_{0 \gamma}{}^{\gamma} \delta_{\beta}^{\alpha} - {e}_{\beta}( \accentset{(\theta)}{\lambda}^{\alpha} ) + \accentset{(\theta)}\lambda^{\bar{\gamma}} \accentset{(\theta)}\Gamma_{\bar{\gamma} \delta}{}^{\delta} \delta_{\beta}^{\alpha} +  \i \accentset{(\theta)}{\lambda}^{\alpha}  \accentset{(\theta)}{\lambda}_{\beta}  & = 0 \, . \label{eq:domalf2}
			\end{align}
		\end{subequations}
		Clearly, equation \eqref{eq:domalf0} is vacuous. By \eqref{eq:conn1_density}, equation \eqref{eq:domalf1} tells us that $\accentset{(\theta)}{\lambda}_{\bar{\alpha}} = 3 \i \sigma^{-1} \accentset{(\theta)}{\nabla}_{\bar{\alpha}} \sigma$ for some density $\sigma$ of weight $(1,0)$ such that $\sigma^{-3} = \theta \wedge \theta^1$ and $\theta = (\sigma \ol{\sigma})^{-1} \bm{\theta}$. However, this expression depends on the choice of coframe, and thus on the choice of connection. To resolve this issue, we note that since $\accentset{(\theta)}{\nabla}(\sigma \ol{\sigma}) = 0$, we can write
		\begin{align*}
			\accentset{(\theta)}{\lambda}_{\alpha} & =  \i \sigma^{-1} \accentset{(\theta)}{\nabla}_{\alpha} \sigma - 2 \i \ol{\sigma}^{-1} \accentset{(\theta)}{\nabla}_{\alpha} \ol{\sigma} = \i \sigma^{-1} \ol{\sigma}^2 \accentset{(\theta)}{\nabla}_{\alpha} (\sigma \ol{\sigma}^{-2} )\, .
		\end{align*}
		This new expression gives the appropriate transformation for $\accentset{(\theta)}{\lambda}_{\alpha}$ under a change of contact form. In particular, it does not depend on the particular choice of representative of $[\nabla_{\alpha}, \lambda_{\alpha}]$, and we can safely drop the adornment $\accentset{(\theta)}{\cdot}$, thereby establishing \eqref{eq:lambdsig}. We can now recast \eqref{eq:Webster--Weyl} as \eqref{eq:WWlin}. Plugging \eqref{eq:domalf1} into \eqref{eq:domalf2} then yields \eqref{eq:divlambsig}.

		Finally, we note that if $z$ is some nowhere vanishing complex-valued function that satisfies $\nabla_{\bar{\alpha}} z = 0$ (so $z$ is a CR function) and $\nabla_{0} z - \lambda^\alpha \nabla_{\alpha} z = 0$, the one-form $\wh{\omega}^{1} = z \omega^{1}$ is also exact with a CR potential. Rescaling $\sigma$ as in \eqref{eq:z_change} yields a unitary adapted coframe of the first kind $(\wh{\theta}, \wh{\theta}^{1}, \ol{\wh{\theta}}{}^{\bar{1}})$ where $\wh{\theta} = z \ol{z} \theta$ and $\wh{\theta}^{1} = z \theta^{1} + \Upsilon^{1} \theta$. Carrying out the analysis as before, we find that
		\begin{align*}
			\wh{\lambda}_{\alpha} & = \i \wh{\sigma}^{-1} \wh{\nabla}_{\alpha} \wh{\sigma} - 2 \i \ol{\wh{\sigma}}^{-1} \wh{\nabla}_{\alpha} \ol{\wh{\sigma}} \, ,
		\end{align*}
		and we may verify for consistency that $\wh{\lambda}_{\alpha} = \lambda_{\alpha} + \i \Upsilon_{\alpha}$.
	\end{proof}
	
	\begin{rem}
		Theorem \ref{thm:lambda_exact} tells us that we obtain an equivalence class of densities related by \eqref{eq:z_change}, for which the solution $\lambda_{\alpha}$ to \eqref{eq:Webster--Weyl} takes the form \eqref{eq:lambdsig}. Each of these densities singles out a connection $\nabla$ in which the connection one-form component $\Gamma_{\alpha \bar{\beta}}{}^{\bar{\beta}}$ in the associated adapted coframe can be identified with $\i \lambda_{\alpha}$. This falls short of a distinguished connection, but nevertheless partially mimics the behaviour of Einstein--Weyl structure in the conformal setting.
	\end{rem}

	\begin{rem}
		While Theorem \ref{thm:lambda_exact} singles out a family of unitary adapted coframes, and in particular, of contact forms, in the context of CR functions, these might not be the most convenient ones when it comes to geometric structures. Indeed, a local transverse symmetry determines a unique contact form $\theta$ for which the Reeb vector field is its generator. But in general, there will be no admissible unitary coframe $(\theta^{1})$ for this $\theta$ such that $\theta^1 = \d w$ where $w$ is a CR function.
	\end{rem}
	
	\section{Conformal geometry}\label{sec:conf_geom}
	In this section, we present some background material on conformal Lorentzian four-manifolds, which will include a discussion of the main object of this article, namely, non-shearing congruences of null geodesics. Besides providing a link between conformal and CR geometries, these are fundamental in our understanding of exact solutions to the Einstein equations, or a subsystem thereof, as will be addressed in Section \ref{sec:confsc} in the language of almost Lorentzian scales. We will also review the Petrov classification of the Weyl tensor and the Goldberg--Sachs theorem, two essential ingredients of mathematical relativity in Section \ref{sec:Petrov}.
	
	\subsection{Preliminaries}
	Let $(\wt{\mc{M}}, \wt{\mbf{c}})$ be an oriented and time-oriented conformal smooth manifold of dimension four and of Lorentzian signature $(+,+,+,-)$. Two metrics $\wt{g}$ and $\wh{\wt{g}}$ belong to the conformal class $\wt{\mbf{c}}$ if and only if
	\begin{align}\label{eq-conf-res}
		\wh{\wt{g}} & = \e^{2 \, \wt{\varphi}} \wt{g} \, , & \mbox{for some smooth function $\wt{\varphi}$ on $\wt{\mc{M}}$,}
	\end{align}
	Following \cite{Bailey1994}, for each $w \in \R$, there are naturally associated density bundles denoted $\wt{\mc{E}}[w]$, and the Levi-Civita connection of any metric $\wt{g}$ in $\wt{\mbf{c}}$ extends to a linear connection on $\wt{\mc{E}}[w]$. In particular, there is a one-to-one correspondence between metrics in $\wt{\mbf{c}}$ and sections of the bundle of \emph{conformal scales}, denoted $\wt{\mc{E}}_+[1]$, which is a choice of positive square root of $\wt{\mc{E}}[2]$. The correspondence is achieved by means of the \emph{conformal metric} $\wt{\bm{g}}$, which is a distinguished non-degenerate section of $\bigodot^2 T \wt{\mc{M}} \otimes \wt{\mc{E}} [2]$: if $\wt{\sigma} \in \Gamma(\wt{\mc{E}}_+[1])$, then $\wt{g} = \wt{\sigma}^{-2} \wt{\bm{g}}$ is the corresponding metric in $\wt{\mbf{c}}$.  We thus have a canonical identification of $T \wt{\mc{M}}$ with $T^* \wt{\mc{M}}\otimes \wt{\mc{E}} [2]$ via $\wt{\bm{g}}$. The Levi-Civita connection $\wt{\nabla}$ of $\wt{g}$ also preserves $\wt{\sigma}$, and it follows that $\wt{\bm{g}}$ is preserved by the Levi-Civita connection of any metric $\wt{g}$ in $\wt{\mbf{c}}$.

	We shall often use the abstract index notation, whereby sections of $T \wt{\mc{M}}$, respectively, $T^* \wt{\mc{M}}$ will be adorned with upper, respectively, lower minuscule Roman indices starting from the beginning of the alphabet, i.e.\ $\wt{\mc{E}}^a := T \wt{\mc{M}}$ and $\wt{\mc{E}}_{a} := T^* \wt{\mc{M}}$. Symmetrisation will be denoted by round brackets, and skew-symmetrisation by square brackets, e.g.\ $\wt{\lambda}_{(a b)} = \tfrac{1}{2} \left( \wt{\lambda}_{a b} + \wt{\lambda}_{b a} \right)$ and $\wt{\lambda}_{[a b]} = \tfrac{1}{2} \left( \wt{\lambda}_{a b} - \wt{\lambda}_{b a} \right)$. In particular, we will write $\mc{E}_{[a_1 \ldots a_k]}$ for the $k$-th exterior power of the cotangent bundle $\bigwedge^k T^* \wt{\mc{M}}$, and $\wt{\mc{E}}_{(a_1 \ldots a_k)}$ for its $k$-th symmetric power $\bigodot^k T^* \wt{\mc{M}}$.  Its trace-free part $\bigodot^k_\circ T^* \wt{\mc{M}}$ will be denoted  $\wt{\mc{E}}_{(a_1 \ldots a_k)_\circ}$.
	
	By convention, we take the Riemann tensor of a given metric $\wt{g}_{a b}$ in $\wt{\mbf{c}}$ to be defined by
	\begin{align}
		2 \wt{\nabla}_{[a} \wt{\nabla}_{b]} \wt{\alpha}_c & = - \wt{\Riem}_{a b}{}^{d}{}_{c} \wt{\alpha}_d \, , &  \wt{\alpha}_a \in \Gamma(\wt{\mc{E}}_a[w]) \, . \label{eq:Riem}
	\end{align}
	It decomposes as
	\begin{align}\label{eq-Riem_decomp}
		\wt{\Riem}_{a b c d} & = \wt{\Weyl}_{a b c d} + 4 \wt{\bm{g}}_{[a|[c} \wt{\Rho}_{d]|b]} \, ,
	\end{align}
	where $\wt{\Weyl}_{a b c d}$ is the conformally invariant \emph{Weyl tensor} and $\wt{\Rho}_{a b}$ the \emph{Schouten tensor}, which is related to the \emph{Ricci tensor} $\wt{\Ric}_{a b} = \wt{\Riem}_{c a}{}^{c}{}_{b}$ and the \emph{Ricci scalar} $\wt{\Sc} = \wt{\Ric}_{a}{}^{a}$ by
	\begin{align*}
		\wt{\Rho}_{a b} & = \tfrac{1}{2} \left( \wt{\Ric}_{a b} - \tfrac{1}{6} \wt{\Sc} \wt{\bm{g}}_{a b} \right) \, .
	\end{align*}
	The \emph{Schouten scalar} is defined to be $\wt{\Rho} := \wt{\Rho}_{a b} \wt{\bm{g}}^{a b} = \tfrac{1}{6} \wt{\Sc} $. The \emph{Cotton tensor} and the \emph{Bach tensor} are given by
	\begin{align}
		\wt{\Cot}_{a b c} & = 2 \wt{\nabla}_{[b} \wt{\Rho}_{c] a} \, , \label{eq:Cotton_Rho} \\
		\wt{\Bach}_{a b} & = - \wt{\nabla}^{c} \wt{\Cot}_{a b c} + \wt{\Rho}^{c d} \wt{\Weyl}_{a c b d} \, , \label{eq:Bach}
	\end{align}
	respectively. The former vanishes whenever $\wt{g}$ is an Einstein metric, while the latter is a conformally invariant obstruction to the existence of an Einstein metric. By the Bianchi identities, we have
	\begin{align}\label{eq:Y=divW}
		\wt{\Cot}_{a b c} = \wt{\nabla}^d \wt{\Weyl}_{d a b c} \, .
	\end{align}
	
	For any two metrics $\wt{g}$ and $\wh{\wt{g}}$ in $\wt{\mbf{c}}$ related by \eqref{eq-conf-res}, their respective Levi-Civita connections $\wt{\nabla}$ and $\wh{\wt{\nabla}}$ are related by
	\begin{align}
		\widehat{\wt{\nabla}}_a \wt{\alpha}_b & = \wt{\nabla}_a \wt{\alpha}_b + (w - 1) \wt{\Upsilon}_a \wt{\alpha}_b - \wt{\Upsilon}_b \wt{\alpha}_a + \wt{\Upsilon}_c \wt{\alpha}^c \wt{\bm{g}}_{a b} \, , & \wt{\alpha}_a \in \Gamma(\wt{\mc{E}}_a[w])  \label{eq-conf-tr-form}
	\end{align}
	where $\wt{\Upsilon}_a := \wt{\nabla}_a \wt{\varphi}$. The Schouten tensor, Schouten scalar and Cotton tensor transform as
	\begin{align}
		\widehat{\wt{\Rho}}_{a b} &  = \wt{\Rho}_{a b} - \wt{\nabla}_{a} \wt{\Upsilon}_{b} + \wt{\Upsilon}_{a} \wt{\Upsilon}_{b} - \tfrac{1}{2} \wt{\Upsilon}^{c} \wt{\Upsilon}_{c} \wt{\bm{g}}_{a b} \, , &
		\widehat{\wt{\Rho}} &  = \wt{\Rho} - \wt{\nabla}^{a} \wt{\Upsilon}_{a} - \wt{\Upsilon}^{c} \wt{\Upsilon}_{c} \, , \label{eq:Rho_transf} \\
		\wh{\wt{\Cot}}_{a b c} & = \wt{\Cot}_{a b c} + \wt{\Upsilon}^d \wt{\Weyl}_{d a b c} \, ,
	\end{align}
	respectively. 	
	
	\subsection{Twisting non-shearing congruences of null geodesics}\label{sec-optical}
	We briefly review some notions that have been introduced and studied variously in \cite{Trautman1984,Trautman1985,Robinson1985,Penrose1986,Robinson1986,Robinson1989,Trautman1999,Fino2023a,Fino2023}. We shall refer to a null line distribution $\wt{K}$ on $(\wt{\mc{M}},\wt{\mbf{c}})$ as an \emph{optical structure}, and the triple $(\wt{\mc{M}},\wt{\mbf{c}},\wt{K})$ as an \emph{optical geometry}. The orientation and time-orientation on $(\wt{\mc{M}},\wt{\mbf{c}})$ induces an orientation on $\wt{K}$. It is a subbundle of its orthogonal complement  $\wt{K}^\perp$ with respect to $\wt{\mbf{c}}$, i.e.\ $\wt{K} \subset \wt{K}^\perp \subset T \wt{\mc{M}}$, and the \emph{screen bundle} $\wt{H} := \wt{K}^\perp / \wt{K}$ of $\wt{K}$ inherits a conformal structure of Riemannian signature from $\wt{\mbf{c}}$.
	
	Since $\wt{H}$ has rank two it admits a bundle complex structure $\wt{J}$, and thus a splitting ${}^\C \wt{H} = \wt{H}^{(1,0)} \oplus \wt{H}^{(0,1)}$ into the eigenbundles of $\wt{J}$. One can show that there is a unique complex rank-two distribution $\wt{N}$, totally null with respect to $\wt{\mbf{c}}$, such that ${}^\C \wt{K} = \wt{N} \cap \ol{\wt{N}}$ and $\wt{N} / {}^\C \wt{K} = \wt{H}^{(0,1)}$. Conversely, any totally null complex rank-two distribution arises in this way. We refer to the pair $(\wt{N},\wt{K})$ as the \emph{almost Robinson structure} associated to $\wt{K}$.
	
	Any nowhere vanishing section $\wt{k}$ of $\wt{K}$ generates a \emph{congruence of null curves} $\wt{\mc{K}}$ tangent to $\wt{K}$. These curves are oriented since $\wt{K}$ is oriented. The local leaf space $\mc{M}$ of $\wt{\mc{K}}$ inherits geometric structures from $(\wt{\mc{M}},\wt{\mbf{c}},\wt{K})$ under certain conditions. The weakest of all is when the curves of $\wt{\mc{K}}$ are geodesics, that is,
	\begin{align}\label{eq:Lie_kap}
		\mathsterling_{\wt{k}} \wt{\kappa}  (\wt{v})  & = 0 \, , & \wt{v} & \in \Gamma(\wt{K}^\perp) \, .
	\end{align}
	where $\wt{\kappa} = \wt{g} (\wt{k},\cdot)$ for some metric $\wt{g}$ in $\wt{\mbf{c}}$. This means that $\wt{H}$ descends to a rank-two distribution $H$ on $\mc{M}$.
	
	We shall suppose that in addition,  $\wt{\mc{K}}$ is \emph{non-shearing},
	that is, the conformal structure on $\wt{H}$ is preserved along the flow of $\wt{k}$, i.e.\
	\begin{align}
		\mathsterling_{\wt{k}} \wt{g} (\wt{v} , \wt{w} ) & = \wt{\epsilon} \wt{g} (\wt{v} , \wt{w} )  \, , & \wt{v}, \wt{w} \in \Gamma(\wt{K}^\perp) \, , \label{eq-nonshear} 
	\end{align}
	for some smooth function $\wt{\epsilon}$. Then $H$ inherits a bundle conformal structure from $\wt{H}$, and thus a bundle complex structure $J$. In fact, the property that $\wt{K}$ is non-shearing is equivalent to the associated almost Robinson structure $(\wt{N},\wt{K})$ being involutive, in which case we refer to $(\wt{N},\wt{K})$ as a \emph{Robinson structure}. The subbundle $\wt{H}^{(1,0)}$ descends to the $\i$-eigenbundle $H^{(1,0)}$ of $J$. This makes $(\mc{M},H,J)$ a CR three-manifold.
	
	Our last requirement will be that $\wt{\mc{K}}$ is \emph{twisting}, i.e.\ $\wt{K}^\perp$ is non-integrable: for any generator $\wt{k}$ and any metric $\wt{g}$ in $\wt{\mbf{c}}$, the one-form $\wt{\kappa} = \wt{g} (\wt{k},\cdot)$ satisfies
	\begin{align}
		\d \wt{\kappa} ( \wt{v} , \wt{w}) & \neq 0 \, , & \wt{v}, \wt{w} & \in \Gamma(\wt{K}^\perp) \, .
	\end{align}
	By virtue of \eqref{eq:Lie_kap}, the leaf space $(\mc{M},H,J)$ of $\wt{\mc{K}}$ is then a contact CR three-manifold.
	
	\begin{rem}
		There are two further notions that are useful in relation to the higher-dimensional story \cite{Fino2023}:
		\begin{itemize}
			\item  We refer to $(\wt{N},\wt{K})$ as \emph{nearly Robinson} if $[\wt{K},\wt{N}] \subset \wt{N}$ --- this implies that $\wt{K}$ is tangent to a congruence of null geodesics. But in dimension four, a nearly Robinson structure is necessarily involutive, and thus equivalent to a non-shearing congruence of null geodesics.
			\item We say that the twist \emph{induces} an almost Robinson structure if for any $\wt{\kappa} \in \Gamma(\Ann(\wt{K}^\perp))$,  the two-form $\d \wt{\kappa}$ restricted to $\wt{H}$ defines a bundle complex structure $\wt{J}$. In dimension four, any twisting congruence of null geodesics defines a twist-induced almost Robinson structure.
		\end{itemize}
	\end{rem}
	
	We summarise the previous remark as a lemma, which will allow us to refer freely to the results of \cite{TaghaviChabert2023b}.
	\begin{lem}
		A four-dimensional optical geometry with twisting non-shearing congruence of null geodesics is the same as a twist-induced (nearly) Robinson geometry with non-shearing congruence.
	\end{lem}
	
	The tensorial quantities associated to a null geodesic congruence thus far considered, namely, the shear and twist, are conformally invariant, and, up to rescalings, do not depend on the choice of generator of $\wt{\mc{K}}$. We now summarise the discussion contained in \cite{TaghaviChabert2022,Fino2023a,Fino2023} concerning the relation between pseudo-Hermitian structures in $(\mc{M},H,J)$ and metrics in $\wt{\mbf{c}}$. To each pseudo-Hermitian structure $\theta$ with Levi form $h$, there is a unique associated metric $\wt{g}_{\theta}$ in $\wt{\mbf{c}}$ with the property that $\wt{\mc{K}}$ is \emph{non-expanding} with respect to it, i.e.\ $\mathsterling_{\wt{k}} \wt{g}_{\theta} (\wt{v},\wt{v}) = 0$ for any $\wt{k} \in \Gamma(\wt{K})$, $\wt{v} \in \Gamma(\wt{K}^\perp)$, and any such metric arises in this way. This defines a conformal subclass of metrics, which we shall denote $\accentset{n.e.}{\wt{\mbf{c}}}$. In addition, there is a unique nowhere vanishing vector field $\wt{k}$ tangent to $\wt{\mc{K}}$ such that, for each metric $\wt{g}_{\theta} \in \accentset{n.e.}{\wt{\mbf{c}}}$, we have $\wt{\kappa} := \wt{g}_{\theta} (\wt{k},\cdot) = 2 \theta$, and for each such choice, there is a unique nowhere vanishing vector field $\wt{\ell}$ satisfying $\wt{\ell} \hook \wt{\kappa} = 1$ and $\wt{\ell} \hook \d \wt{\kappa} = 0$ --- see \cite[Proposition~4.30]{Fino2023a}. Each $\wt{g}_{\theta}$ can then be uniquely expressed as
	\begin{subequations}\label{eq:can_met}
		\begin{align}\label{eq:can_met1}
			\wt{g}_{\theta} & = 4 \theta \odot \wt{\lambda} + h \, ,
		\end{align}
		where, with a slight abuse of notation, $h$ is the metric induced from the Levi form of $\theta$, and $\wt{\lambda} = \wt{g}(\wt{\ell},\cdot)$. Choosing an admissible coframe $(\theta^{\alpha})$ for $\theta$, and an affine parameter $\phi$ along the geodesics of $\wt{k}$ so that $\wt{k}=\parderv{}{\phi}$, we can write
		\begin{align}\label{eq:can_met2}
			h & = 2 h_{\alpha \bar{\beta}} \theta^{\alpha} \odot \ol{\theta}{}^{\bar{\beta}} \, , & 
			\wt{\lambda} & = \d \phi + \wt{\lambda}_{\alpha} \theta^{\alpha} + \wt{\lambda}_{\bar{\alpha}} \ol{\theta}{}^{\bar{\alpha}} + \wt{\lambda}_0 \theta \, ,
		\end{align}
	\end{subequations}
	for some complex-valued functions $\wt{\lambda}_{\alpha}$ and $\wt{\lambda}_{\bar{\alpha}}=\ol{\wt{\lambda}_{\alpha}}$, and real-valued function $\wt{\lambda}_0$ on $\wt{\mc{M}}$. Differentiation with respect to $\phi$ will be denoted by a dot, i.e.\ $\dot{\wt{f}} := \mathsterling_{\wt{k}} \wt{f}$ for any smooth tensor-valued function $\wt{f}$ on $\wt{\mc{M}}$, and this notation will be extended to tensor components. Choosing $(\theta^{\alpha},\ol{\theta}{}^{\bar{\alpha}})$ to be unitary so that $h_{\alpha \bar{\beta}} = \delta_{\alpha \bar{\beta}}$, and bearing in mind that indices only take the value $1$, we then see that to each choice of metric $\wt{g}_{\theta}$ in $\accentset{n.e.}{\wt{\mbf{c}}}$ corresponds a coframe $(\wt{\kappa},\theta^{\alpha},\ol{\theta}{}^{\bar{\alpha}},\wt{\lambda})$ adapted to $(\wt{N},\wt{k})$, unique up to a phase transformation of $\theta^{\alpha}$.
	
	Let $\wt{g}{}_{\theta}$ and $\wt{g}{}_{\wh{\theta}}$ be two metrics in $\accentset{n.e.}{\wt{\mbf{c}}}$ corresponding to two pseudo-Hermitian structures $\theta$ and $\wh{\theta}$, so that $\wt{g}{}_{\wh{\theta}} = \e^\varphi \wt{g}{}_{\theta}$ and $\wh{\theta} = \e^{\varphi} \theta$  for some smooth function $\varphi$ on $\mc{M}$. Then the corresponding coframes $(\wt{\kappa},\theta^{\alpha},\ol{\theta}{}^{\bar{\alpha}},\wt{\lambda})$ and $(\wh{\wt{\kappa}},\wh{\theta}^{\alpha},\ol{\wh{\theta}}{}^{\bar{\alpha}},\wh{\wt{\lambda}})$ are related by
	\begin{gather}\label{eq:adFrch}
		\begin{aligned}
			\wh{\wt{\kappa}} = \e^{\varphi} \wt{\kappa} \, , \qquad 
			\wh{\theta}{}^{\alpha} = \theta^{\alpha} + \i \Upsilon^{\alpha} \theta \, , \qquad 		\ol{\wh{\theta}}{}^{\bar{\alpha}} = \ol{\theta}{}^{\bar{\alpha}} - \i \Upsilon{}^{\bar{\alpha}} \theta \, , \\
			\wh{\wt{\lambda}} = \wt{\lambda} + \tfrac{1}{2} \i \Upsilon_{\alpha} \theta^{\alpha} - \tfrac{1}{2} \i \Upsilon_{\bar{\alpha}} \ol{\theta}{}^{\bar{\alpha}} - \tfrac{1}{2} \Upsilon_{\alpha} \Upsilon^{\alpha} \theta \, ,
		\end{aligned}
	\end{gather}
	where $\Upsilon_{\alpha} = \nabla_{\alpha} \varphi$, in agreement with equation \eqref{eq:contact_chg}. More details can be found in \cite{TaghaviChabert2023b}.
	
	We shall henceforth denote an optical geometry with twisting non-shearing congruence of null geodesics by the triple $(\wt{\mc{M}},\wt{\mbf{c}},\wt{k})$. The perturbed Fefferman spaces to be introduced in Section \ref{sec:Fefferman} are a special case of these.

	\subsection{Almost Einstein scales and generalisations}\label{sec:confsc}
	Recall that a metric $\wt{g}$ in $\wt{\mbf{c}}$ is said to be \emph{Einstein} if its Ricci tensor satisfies
	\begin{align*}
		\wt{\Ric} & = \wt{\Lambda} \wt{g} \, , & \mbox{for some constant $\wt{\Lambda}$.}
	\end{align*}
	In the context of optical geometries, one can make further definitions whereby the Einstein condition is weakened. These were already introduced in \cite{TaghaviChabert2023b}, but on account of the features specific to dimension four, we restate them here in a slightly different way.
	\begin{defn}\label{defn:redEins}
		Let $(\wt{\mc{M}},\wt{\mbf{c}},\wt{K})$ be an optical structure with associated almost Robinson structure $(\wt{N},\wt{K})$. We say that a metric $\wt{g}$ in $\wt{\mbf{c}}$ is:
		\begin{itemize}
			\item a \emph{weakly half-Einstein} metric if its Ricci tensor satisfies
			\begin{align}\label{eq:wk_hlf_Einstein}
				\wt{\Ric}_{a b} \wt{v}^a \wt{v}^b & = 0 \, , & \mbox{for any $\wt{v} \in \Gamma(\wt{N})$;}
			\end{align}
			\item a \emph{half-Einstein} metric if it is weakly half-Einstein and has constant Ricci scalar curvature;
			\item a \emph{pure radiation} metric if it is half-Einstein and the trace-free part of the Ricci tensor satisfies
			\begin{subequations}\label{eq:puradvac}
				\begin{align}\label{eq:puradvac1}
					\wt{\kappa}_{[a} \left( \wt{\Ric}_{b] c} \right)_\circ & = 0 \, ,  & \mbox{where $\wt{\kappa}_{a} = \wt{g}_{a b} \wt{k}^b$ for some $\wt{k} \in \Gamma(\wt{K})$,}
				\end{align}
				i.e.\
				\begin{align}\label{eq:puradvac2}
					\wt{\Ric}_{a b} & = \wt{\Lambda} \wt{g}_{a b} + \wt{\Phi} \wt{\kappa}_{a} \wt{\kappa}_{b} \, ,  & \mbox{where $\wt{\kappa}_{a} = \wt{g}_{a b} \wt{k}^b$ for some $\wt{k} \in \Gamma(\wt{K})$,}
				\end{align}
				for some constant $\wt{\Lambda}$ and smooth function $\wt{\Phi}$ on $\wt{\mc{M}}$.
			\end{subequations}
		\end{itemize}
	\end{defn}
	
	For our considerations, these conditions are somewhat too strong. For this reason, we shall use the notion of \emph{almost Lorentzian scale}, referred to as \emph{almost pseudo-Riemannian structure} in \cite{Curry2018}: This is a section $\wt{\sigma}$ of $\wt{\mc{E}}[1]$ with zero set $\wt{\mc{Z}}$ such that $\wt{\nabla}_{a} \wt{\sigma} \neq 0$ on $\wt{\mc{Z}}$ where $\wt{\nabla}$ is the Levi-Civita connection compatible with some metric in $\wt{\mbf{c}}$. That this definition does not depend on the choice of metric follows from the transformation rule $\wh{\wt{\nabla}}_{a} \wt{\sigma} = \wt{\nabla}_{a} \wt{\sigma} + \wt{\Upsilon}_{a} \wt{\sigma}$, and the fact that the second term vanishes on $\wt{\mc{Z}}$. The density $\wt{\sigma}$ defines a metric $\wt{g} = \wt{\sigma}^{-2} \wt{\bm{g}}$ in $\wt{\mbf{c}}$, but regular only off $\wt{\mc{Z}}$.
	
	First introduced by Gover in \cite{Gover2005a} --- see also \cite{LeBrun1985} --- an \emph{almost Einstein scale} is an almost Lorentzian scale $\wt{\sigma}$ that satisfies the conformally invariant equation
	\begin{align}
		\left( \wt{\nabla}_{a} \wt{\nabla}_{b} \wt{\sigma} + \wt{\Rho}_{a b} \wt{\sigma} \right)_\circ = 0 \, . \label{eq:alEinstein}		
	\end{align}
	If $\wt{\sigma}$ has empty zero set $\wt{\mc{Z}}$, then it is referred to as an \emph{Einstein scale}, and it defines a (global) Einstein metric. Otherwise, if $\wt{\mc{Z}}$ is non-empty, then off $\wt{\mc{Z}}$, the metric $\wt{g} = \wt{\sigma}^{-2} \wt{\bm{g}}$ is Einstein. A useful interpretation of $\wt{\mc{Z}}$ is as the \emph{conformal infinity} of some Lorentzian Einstein manifold --- see \cite{Curry2018} for more details.

	We can readily generalise this idea to the metrics introduced in Definition \ref{defn:redEins}.
	\begin{defn}\label{defn:aESc}
		Let $(\wt{\mc{M}},\wt{\mbf{c}},\wt{K})$ be an optical geometry with associated almost Robinson structure $(\wt{N},\wt{K})$. We say that a density $\wt{\sigma} \in \Gamma(\wt{\mc{E}}[1])$ is 
		\begin{itemize}
			\item an \emph{almost weakly half-Einstein scale} if it satisfies
			\begin{align}\label{eq:alwkEins}
				&	\left( \wt{\nabla}_{a} \wt{\nabla}_{b} \wt{\sigma} + \wt{\Rho}_{a b} \wt{\sigma} \right)_\circ = \tfrac{1}{2} \wt{\Phi}_{a b} \wt{\sigma} \, , 
			\end{align}
			for some trace-free symmetric tensor $\wt{\Phi}_{a b}$ satisfying
			\begin{align}\label{eq:RiccwkEins}
				\wt{\Phi}_{a b} \wt{v}^a \wt{v}^b & = 0 \, , & \mbox{for any $\wt{v}^a \in \Gamma (\wt{N})$;}
			\end{align}
			\item an \emph{almost half-Einstein scale} if it is an almost weakly half-Einstein scale, and
			\begin{align}\label{eq:alcstRicSc}
				& \wt{\Phi}_{a}{}^{b} \wt{\nabla}_{b} \wt{\sigma} - \tfrac{1}{2} \wt{\sigma} \wt{\nabla}_{b} \wt{\Phi}_{a}{}^{b} = 0 \, ;
			\end{align}
			\item an \emph{almost pure radiation scale} if it is an almost half-Einstein scale, and
			\begin{align}\label{eq:purad}
				& \wt{\Phi}_{a [b} \wt{\bm{\kappa}}_{c]}  = 0 \, , & \mbox{where $\wt{\bm{\kappa}}_{a} = \wt{\bm{g}}_{a b} \wt{k}^{b}$ for some $\wt{k} \in \Gamma(\wt{K})$.}
			\end{align}
		\end{itemize} 
	\end{defn}
	All the conditions given in the above definition are conformally invariant.
	\begin{rem}
		Note that condition \eqref{eq:purad} is equivalent to
		\begin{align}\label{eq:purad_alt}
			\wt{\Phi}_{a b} & = \wt{\bm{\Phi}} \wt{\bm{\kappa}}_{a} \wt{\bm{\kappa}}_{b} \, , & \mbox{for some $\wt{\bm{\Phi}} \in \Gamma(\wt{\mc{E}}(-4))$.}
		\end{align}
	\end{rem}
	The relation between Definition \ref{defn:redEins} and Definition \ref{defn:aESc} is given below. The proof is already given in \cite{TaghaviChabert2023b}.
	\begin{prop}\label{prop:scale2metric}
		Let $(\wt{\mc{M}},\wt{\mbf{c}},\wt{K})$ be a four-dimensional optical geometry, and let $\wt{\sigma} \in \Gamma(\mc{E}[1])$ so that $\wt{g} = \wt{\sigma}^{-2} \wt{\bm{g}} \in \wt{\mbf{c}}$ is a smooth metric off the zero set of $\wt{\sigma}$. Then
		\begin{enumerate}
			\item $\wt{\sigma}$ is an almost weakly half-Einstein scale if and only if $\wt{g}$ is a weakly half-Einstein metric;
			\item $\wt{\sigma}$ is an almost half-Einstein scale if and only if $\wt{g}$ is a half-Einstein metric;
			\item $\wt{\sigma}$ is an almost pure radiation scale if and only if $\wt{g}$ is a pure radiation metric.
		\end{enumerate}
	\end{prop}

	\subsection{Principal null directions and the Petrov types of the Weyl tensor}\label{sec:Petrov}
	Let $(\wt{\mc{M}}, \wt{\mbf{c}},\wt{K})$  be a four-dimensional optical geometry.  We shall adapt a number of definitions from mathematical relativity to the present setting --- see \cite{Penrose1986} for details. We first work at a point $\wt{p}$ of $\wt{\mc{M}}$ and assume that the Weyl tensor $\wt{\Weyl}$ is non-zero at $\wt{p}$. We say that $\wt{K}_{\wt{p}}$ is a
	\emph{principal null direction (PND)} of the Weyl tensor $\wt{\Weyl}$ at $\wt{p}$ if
	\begin{align}
		\wt{\Weyl}(\wt{k},\wt{v},\wt{k},\wt{v}) & = 0 \, , & \mbox{for any $\wt{k} \in \Gamma(\wt{K}_{\wt{p}})$, $\wt{v} \in \Gamma(\wt{K}_{\wt{p}}^\perp)$,} \label{eq:1PND}
	\end{align}
	and it is said to be \emph{repeated} if it satisfies
	\begin{align}
		\wt{\Weyl}(\wt{k},\wt{v},\wt{k},\cdot) & = 0 \, , & \mbox{for any $\wt{k} \in \Gamma(\wt{K}_{\wt{p}})$, $\wt{v} \in \Gamma(\wt{K}_{\wt{p}}^\perp)$.} \label{eq:2PND}
	\end{align}
	Further degeneracy conditions are possible, namely,
	\begin{align}
		\wt{\Weyl}(\wt{k},\wt{v},\wt{w},\cdot) & = 0 \, , & \mbox{for any $\wt{k} \in \Gamma(\wt{K}_{\wt{p}})$, $\wt{v}, \wt{w} \in \Gamma(\wt{K}_{\wt{p}}^\perp)$,} \label{eq:3PND} \\
		\wt{\Weyl}(\wt{k},\cdot,\cdot,\cdot) & = 0 \, , & \mbox{for any $\wt{k} \in \Gamma(\wt{K}_{\wt{p}})$.} \label{eq:4PND}
	\end{align}
	In this case, we refer to a PND $\wt{K}_{\wt{p}}$ at $\wt{p}$ as
	\begin{itemize}
		\item  \emph{simple} if it satisfies \eqref{eq:1PND} but not \eqref{eq:2PND},
		\item  \emph{double} if it satisfies \eqref{eq:2PND} but not \eqref{eq:3PND},
		\item \emph{triple} if it satisfies \eqref{eq:3PND} but not \eqref{eq:4PND},
		\item \emph{quadruple} if it satisfies \eqref{eq:4PND}.
	\end{itemize}
	The reader may check that this is in agreement with the established terminology by comparing the above defining equations with, for instance, \cite[Table~8.1.4]{Penrose1986}. We can naturally extend these definitions to the neighbourhood $\wt{\mc{U}}$ of any given point in $\wt{\mc{M}}$ (or $\wt{\mc{M}}$ itself): thus, $\wt{K}$ will be said to be a \emph{simple/double/triple/quadruple} on $\wt{\mc{U}}$ if it is so at every point of $\wt{\mc{U}}$.
	
	On the other hand, at any point of a four-dimensional Lorentzian conformal manifold $(\wt{\mc{M}},\wt{\mbf{c}})$, where it is is non-zero, the Weyl tensor $\wt{\Weyl}$ always determines at most four distinct PNDs, and thus locally, at most four distinct optical structures. The multiplicity of each of these PNDs form the basis of the \emph{Petrov classification} of the Weyl tensor \cite{Petrov1954}. We say that $\wt{\Weyl}$ is \emph{algebraically special} at $\wt{p} \in \wt{\mc{M}}$ if it admits a repeated PND there, and \emph{algebraically general} otherwise. In more details, now allowing for the Weyl tensor to vanish, $\wt{\Weyl}$ is said to be
	\begin{itemize}
		\item of \emph{Petrov type I} at $\wt{p}$ if it does not admit any repeated PNDs at $\wt{p}$,
		\item of \emph{Petrov type II} at $\wt{p}$ if it admits a single double PND at $\wt{p}$,
		\item of \emph{Petrov type D} at $\wt{p}$ if it admits a pair of distinct double PNDs at $\wt{p}$,
		\item of \emph{Petrov type III} at $\wt{p}$ if it admits a triple PND at $\wt{p}$,
		\item of \emph{Petrov type N} at $\wt{p}$ if it admits a quadruple PND at $\wt{p}$,
		\item of \emph{Petrov type O} at $\wt{p}$ if it vanishes at $\wt{p}$.
	\end{itemize}
	Again, these definitions can be applied to subsets of $\wt{\mc{M}}$ in the obvious way. In particular, $(\wt{\mc{M}},\wt{\mbf{c}})$ is conformally flat if it is of Petrov type O at every point.
	
	It is a standard result that if $\wt{K}$ is a non-shearing congruence of null geodesics, then $\wt{K}$ is a PND of the Weyl tensor. Another important result is the so-called Goldberg--Sachs theorem \cite{Goldberg1962}, which in the present context, can be formulated in these terms:
	\begin{thm}\label{thm:GS}
		Let $(\wt{\mc{M}},\wt{\mbf{c}})$ be a four-dimensional Lorentzian manifold that admits an almost Einstein scale. Then the Weyl tensor is algebraically special if and only if it admits a non-shearing congruence of null geodesics.
	\end{thm}
	We omit the proof, which is similar to that of Theorem \ref{thm:GSalpha} below. There are variations of this theorem, and for our purpose, we shall need the following reformulation of a result due to \cite{GoverHillNurowski11}.
	\begin{thm}\label{thm:GSalpha}
		Let $(\wt{\mc{M}},\wt{\mbf{c}},\wt{K})$ be an optical structure with non-shearing congruence of null geodesics that admits an almost weakly half-Einstein scale. Then $\wt{K}$ is a repeated PND of the Weyl tensor.
	\end{thm}
	
	\begin{proof}
		By hypothesis, the almost weakly half-Einstein scale $\wt{\sigma}$ defines a weakly half-Einstein metric $\wt{g}$ off the zero set $\wt{\mc{Z}}$ of $\wt{\sigma}$. By \cite[Theorem 5.9]{GoverHillNurowski11}, we already know that off $\wt{\mc{Z}}$, $\wt{K}$ is a repeated PND of the Weyl tensor. This property extends smoothly to $\wt{\mc{Z}}$.
	\end{proof}
	Of course, the Weyl tensor may degenerate further on the zero set, as will be investigated in Section \ref{sec:Weyl}.
	
	We end this section by introducing some notation. Let $(\wt{\mc{M}},\wt{\mbf{c}},\wt{K})$ be an optical geometry with congruence of null curves $\wt{\mc{K}}$ and associated almost Robinson structure $(\wt{N},\wt{K})$. For a given choice of metric $\wt{g}$  in $\wt{\mbf{c}}$, it is convenient to choose a frame $(\wt{\ell}, \wt{e}_{1},\ol{\wt{e}}_{\bar{1}},\wt{k}) = (\wt{e}_{0}, \wt{e}_{1},\ol{\wt{e}}_{\bar{1}},\wt{e}^{0})$ adapted to $(\wt{N},\wt{K})$, with dual $(\wt{\kappa},\wt{\theta}{}^{1}, \ol{\wt{\theta}}{}^{\bar{1}},\wt{\lambda}) = (\wt{\theta}{}^0,\wt{\theta}{}^{1}, \ol{\wt{\theta}}{}^{\bar{1}},\wt{\theta}{}_0)$, so that $\wt{K} = \langle \wt{k} \rangle$, $\wt{N} = {}^\C \langle \ol{\wt{e}}_{\bar{1}} , \wt{k}  \rangle$, and
	\begin{align*}
		\wt{g}(\wt{e}_{0} , \wt{e}^0) & = 1 \, , & 	\wt{g}(\wt{e}_{1} , \ol{\wt{e}}_{\bar{1}}) & = 1 \, ,
	\end{align*}
	and all other pairings zero. For compatibility with our earlier notation, we shall also write $\wt{e}_{\alpha}$ and $\ol{\wt{e}}_{\bar{\alpha}}$ for $\wt{e}_{1}$ and $\ol{\wt{e}}_{\bar{1}}$ respectively, and similarly for their duals. This will also allow us to make contact with the computations of \cite{TaghaviChabert2022}.
	
	Recall that in dimension four, the Weyl tensor splits into a self-dual part $\wt{\Psi}$ and an anti-self-dual part $\ol{\wt{\Psi}}$, complex-valued tensors that are conjugate of each other. Assuming that our orientation is given by positive multiples of $\i \wt{\lambda} \wedge \wt{\kappa}\wedge \wt{\theta}{}^{1} \wedge  \ol{\wt{\theta}}{}^{\bar{1}}$, we set
	\begin{align}\label{eq:NP-Psi}
		\wt{\Psi}_0 & := \wt{\Weyl}_{\alpha}{}^{0} {}_{\beta}{}^{0} \, , & \wt{\Psi}_1 & := \wt{\Weyl}^{0}{}_{\alpha} {}^{0}{}_{0} \, , &	\wt{\Psi}_2 & := \wt{\Weyl}_{\alpha 0}{}^{\alpha 0} \, , &	\wt{\Psi}_3 & := \wt{\Weyl}_{0 \alpha} {}_{0}{}^{0} \, , &	\wt{\Psi}_4 := \wt{\Weyl}_{\alpha 0} {}_{\beta 0} \, ,
	\end{align}
	with respective complex conjugates $\ol{\wt{\Psi}}_{0}$, $\ol{\wt{\Psi}}_{1}$, $\ol{\wt{\Psi}}_{2}$, $\ol{\wt{\Psi}}_{3}$ and $\ol{\wt{\Psi}}_{4}$.
	We shall now assume that $\wt{\mc{K}}$ is geodesic and non-shearing (so that $\wt{\Psi}_0 = 0$), and in addition that $\wt{K}$ is a repeated PND, so that
	\begin{align*}
		\wt{\Psi}_{0} = \wt{\Psi}_{1} = 0 \, .
	\end{align*}
	We suppose further that $\wt{\mc{K}}$ is twisting. Referring to the setup of Section \ref{sec-optical}, we choose two metrics $\wt{g}_{\theta}$ and $\wt{g}_{\wh{\theta}}$ in $\accentset{n.e.}{\wt{\mbf{c}}}$, so that the corresponding change of adapted frame given by \eqref{eq:adFrch} induces the change
	\begin{subequations}\label{eq:Psichange}
		\begin{align}
			\wh{\wt{\Psi}}_2 & = \wt{\Psi}_2 \, , \\ 	(\wh{\wt{\Psi}}_3)_{\alpha} & = (\wt{\Psi}_3)_{\alpha} - \tfrac{3}{2}\i \Upsilon_{\alpha} \wt{\Psi}_2 \, , \\ 	(\wh{\wt{\Psi}}_4)_{\alpha \beta} & = (\wt{\Psi}_4)_{\alpha \beta} - 2 \i \Upsilon_{\alpha} (\wt{\Psi}_3)_{\beta} - \tfrac{3}{2} \Upsilon_{\alpha} \Upsilon_{\beta} \wt{\Psi}_2 \, .
		\end{align}
	\end{subequations}
	Here, we have adjoined abstract indices to $\wt{\Psi}_3$ and $\wt{\Psi}_4$ to emphasize the transformation laws.
	
	Table \ref{tab:Petrov} below characterises the remaining possible Petrov types in terms of the components $\wt{\Psi}_{i}$ defined in \eqref{eq:NP-Psi}:
	\begin{table}[ht]
		\begin{tabular}{||c | c||} 
			\hline
			Petrov type  & Conditions on $\wt{\Psi}_i$ \\ [0.5ex] 
			\hline\hline
			D & $4 (\wt{\Psi}_3)^2 - 6 \wt{\Psi}_2 \wt{\Psi}_4 = 0$ \\
			\hline
			III & $\wt{\Psi}_2 = 0$ \\
			\hline
			N & $\wt{\Psi}_2 = \wt{\Psi}_3 = 0$ \\
			\hline
			O & $\wt{\Psi}_2 = \wt{\Psi}_3 = \wt{\Psi}_4 = 0$ \\
			\hline
		\end{tabular}
		\caption{\label{tab:Petrov}}
	\end{table}
	
	\section{The Fefferman construction and its variants}\label{sec:Fefferman}
	Among the twisting non-shearing congruences of null geodesics are those generated by a conformal Killing symmetry. As shown by Lewandowski and Nurowski in \cite{Lewandowski1990b}, under appropriate curvature conditions, namely that the Weyl tensor is of Petrov type N, one recovers Fefferman's canonical conformal structure fibred over a CR three-manifold, which we shall first review. In preparation for the subsequent sections, notably \ref{sec:algsp} and \ref{sec:al_Lor_sc}, we will then perturb this conformal structure in Section \ref{sec:perb_Feff}.
	
	\subsection{Fefferman spaces}
	We  recall the Fefferman construction as introduced in \cite{Cap2008}. A contact CR three-manifold $(\mc{M},H,J)$ has a canonical circle bundle $\wt{\mc{M}}$ defined as the quotient of the $\C^*$-principal bundle $\mc{E}(-1,0) \setminus \{ 0 \}$ by the natural $\R_{>0}$-action. We can project a nowhere vanishing section $\tau$ of $\mc{E}(-1,0)$ to a section of $\wt{\mc{M}} \rightarrow \mc{M}$ as an equivalence class $[\tau]$, where $\tau, \tau' \in [\tau]$ if and only if $\tau' = \varrho \tau$ for some positive real-valued function $\varrho$ on $\mc{M}$. It then follows that a given choice of contact form $\theta$ allows us to identify sections of $\wt{\mc{M}}$ with nowhere vanishing densities $\tau$ of weight $(-1,0)$ such that $\theta = \tau \ol{\tau} \bm{\theta}$. At the same time, $\tau$ defines a fibre coordinate $\phi \in [ - \pi, \pi)$ on $\wt{\mc{M}}$ such that $\e^{\i \phi} \tau$ is a section of $\wt{\mc{M}}$. The \emph{Fefferman metric} associated to $\theta$ is the Lorentzian metric on $\wt{\mc{M}}$ given by
	\begin{align}\label{eq:Feff_metric}
		\wt{g}_{\theta} & = 4 \theta \odot \left( \d \phi + \tfrac{\i}{2} \left( \sigma^{-1} \nabla \sigma - \ol{\sigma}^{-1} \nabla \ol{\sigma} \right) - \tfrac{1}{3}  \Rho \theta \right) + h \, ,
	\end{align}
	where $\sigma = \tau^{-1} \in \Gamma(\mc{E}(1,0))$, and $h$ is the degenerate metric induced by the Levi form of $\theta$. One can check that $\wt{g}_{\theta}$ is independent of the choice of $\sigma$ provided $\theta = (\sigma \ol{\sigma})^{-1} \bm{\theta}$, and transforms conformally under a change of pseudo-Hermitian structure. Thus, we obtain a conformal class $\wt{\mbf{c}}$ associated to $(\mc{M},H,J)$. Further, the vector field $\wt{k} = \frac{\partial}{\partial \phi}$ generating the fibres of $\wt{\mc{M}} \rightarrow \mc{M}$ is conformal Killing for $\wt{\mbf{c}}$, and in fact, Killing for each of the Fefferman metrics in $\wt{\mbf{c}}$. We shall refer to $(\wt{\mc{M}},\wt{\mbf{c}},\wt{k}) \longrightarrow (\mc{M},H,J)$ as the \emph{Fefferman space} of $(\mc{M},H,J)$.
	
	\subsection{Perturbed Fefferman spaces}\label{sec:perb_Feff}
	Following \cite{TaghaviChabert2023b}, we consider a Fefferman space\linebreak $(\wt{\mc{M}},\wt{\mbf{c}},\wt{k}) \longrightarrow (\mc{M},H,J)$ and choose a semi-basic one-form $\wt{\xi}$ on $\wt{\mc{M}}$, that is, $\wt{k} \hook \wt{\xi} = 0$. To each Fefferman metric $\wt{g}_{\theta}$ in $\wt{\mbf{c}}$, we define the \emph{Fefferman metric perturbed by $\wt{\xi}$} as
	\begin{align}\label{eq:pertFeff}
		\wt{g}_{\theta,\wt{\xi}} & = \wt{g}_{\theta} + 4 \theta \odot \wt{\xi} \, ,
	\end{align}
	from which we clearly obtain a corresponding \emph{perturbed Fefferman conformal structure} $\wt{\mbf{c}}_{\wt{\xi}}$. We call $\wt{\xi}$ the \emph{perturbation one-form} of $(\wt{\mc{M}},\wt{\mbf{c}},\wt{k})$, and the triple $(\wt{\mc{M}},\wt{\mbf{c}}_{\wt{\xi}},\wt{k})$ as a \emph{perturbed Fefferman space}.
	
	We now pick a nowhere vanishing density $\sigma \in \Gamma(\mc{E}(1,0))$ such that $\theta = (\sigma \ol{\sigma})^{-1} \bm{\theta}$ is the contact form corresponding to the metric \eqref{eq:pertFeff}, and let $\phi$ be the associated fibre coordinate on $\wt{\mc{M}} \rightarrow \mc{M}$. With respect to an admissible coframe $(\theta^{\alpha})$, we express the semi-basic one-form as
	\begin{align}\label{eq:xi_cmpnt}
		\wt{\xi} & = \wt{\xi}_{\alpha} \theta^{\alpha} + \wt{\xi}_{\bar{\alpha}} \ol{\theta}{}^{\bar{\alpha}} + \wt{\xi}_{0} \theta \, .
	\end{align}
	Since $\wt{\mc{M}}$ is a circle bundle, we can Fourier expand the components $\wt{\xi}_{\alpha}$, $\wt{\xi}_{\bar{\alpha}}$ and  $\wt{\xi}_{0}$, i.e.\
	\begin{align}\label{eq:xi_Fourier}
		\wt{\xi}_{\alpha} & = \sum_{k \in \mc{I}} \xi_{\alpha}^{(k)} \e^{k \i \phi} \, , &
		\wt{\xi}_{\bar{\alpha}} & = \sum_{k \in \mc{I}} \xi_{\bar{\alpha}}^{(-k)} \e^{-k \i \phi} \, , &
		\wt{\xi}_{0} & = \sum_{k \in - \mc{J} \cup \mc{J}} \xi_{0}^{(k)} \e^{k \i \phi} \, ,
	\end{align}
	for some subsets $\mc{I} \subset \Z$, $\mc{J} \subset \Z_{\geq 0}$. For each $k \in \mc{I}$ and for each $k \in \mc{J}$, we can then define densities $\bm{\xi}_{\alpha}^{(k)} \in \Gamma(\mc{E}_{\alpha}\left(\tfrac{k}{2},-\tfrac{k}{2}\right))$ and $\bm{\xi}_{0}^{(k)} \in \Gamma(\mc{E} \left(\tfrac{k}{2}-1,-\tfrac{k}{2}-1\right))$ by
	\begin{align}\label{eq:Fourier_coef}
		\bm{\xi}_{\alpha}^{(k)} & = \xi_{\alpha}^{(k)} \sigma^{\tfrac{k}{2}} \ol{\sigma}^{-\tfrac{k}{2}} \, , & \bm{\xi}_{\bar{\alpha}}^{(k)} & = \xi_{\bar{\alpha}}^{(k)} \ol{\sigma}^{\tfrac{k}{2}} \sigma^{-\tfrac{k}{2}} \, , & \bm{\xi}_{0}^{(k)} = \xi_{0}^{(k)}  \sigma^{-1+\tfrac{k}{2}} \ol{\sigma}^{-1-\tfrac{k}{2}} \, .
	\end{align}
	Note that, for $i \in \mc{I}, j \in \mc{J}$, $\bm{\xi}_{\bar{\alpha}}^{(i)} = \ol{\bm{\xi}_{\alpha}^{(-i)}}$, $\bm{\xi}_{0}^{(j)} = \ol{\bm{\xi}_{0}^{(-j)}}$. In particular $\bm{\xi}_{0}^{(0)}$ is real.
	
	The $\bm{\xi}_{\alpha}^{(k)}$ do not depend on the choice of contact form, while the $\bm{\xi}_{0}^{(k)}$ transform as
	\begin{align}\label{eq:wtd_Fourier}
		\wh{\bm{\xi}}{}^{(k)}_0 & = \bm{\xi}^{(k)}_0 - \i \bm{\xi}^{(k)}_{\alpha} \Upsilon^\alpha + \i \bm{\xi}^{(k)}_{\bar{\alpha}} \Upsilon^{\bar{\alpha}} \, .
	\end{align}
	We shall denote the equivalence class of such densities related by this transformation by $[\nabla, \bm{\xi}^{(k)}_0]$.
	
	Summarising, the perturbation $\wt{\xi}$ determines a tuple $\left( \bm{\xi}^{(i)}_{\alpha} , [\nabla, \bm{\xi}^{(j)}_{0}] \right)_{i \in \mc{I}, j \in \mc{J}}$, and any such tuple determines a perturbation one-form. We shall refer to this tuple as the \emph{CR data associated to $\wt{\xi}$} --- see \cite{TaghaviChabert2023b} for details.
	
	We give an example to illustrate the definition.
	\begin{exa}\label{exa:CR_data}
		A perturbation one-form $\wt{\xi}$ with CR data
		\begin{equation*}
			\left( \bm{\xi}_{\alpha}^{(-2)}, \bm{\xi}_{\alpha}^{(0)}, [\nabla, \bm{\xi}_{0}^{(0)}, \bm{\xi}_{0}^{(2)}], \bm{\xi}_{0}^{(4)} \right)
		\end{equation*}
		means that with a choice of trivialisation $\sigma$ of $\wt{\mc{M}}$ with fibre coordinate $\phi$ and adapted coframe $(\theta,\theta^{\alpha}, \ol{\theta}{}^{\bar{\alpha}})$ with $\theta = (\sigma \ol{\sigma})^{-1} \bm{\theta}$, the one-form $\wt{\xi}$ is given by
		\begin{multline*}
			\wt{\xi} = \left( \xi_{\alpha}^{(0)} + \xi_{\alpha}^{(-2)} \e^{-2\i \phi} \right) \theta^{\alpha} + \left( \xi_{\bar{\alpha}}^{(0)} + \xi_{\bar{\alpha}}^{(2)} \e^{2\i \phi} \right) \ol{\theta}{}^{\bar{\alpha}} \\
			+ \left( \xi_{0}^{(4)} \e^{4\i \phi} + \xi_{0}^{(2)} \e^{2\i \phi} + \xi_{0}^{(0)} + \xi_{0}^{(-2)} \e^{-2\i \phi} + \xi_{0}^{(-4)} \e^{-4\i \phi} \right) \theta \, ,
		\end{multline*}
		where $\xi_{\alpha}^{(0)} = \bm{\xi}_{\alpha}^{(0)}$, $\xi_{\alpha}^{(-2)} = \sigma \ol{\sigma}^{-1} \bm{\xi}_{\alpha}^{(-2)}$, $\xi_{0}^{(0)} = \sigma \ol{\sigma} \bm{\xi}_{0}^{(0)}$, $\xi_{0}^{(-2)} = \sigma^2 \bm{\xi}_{0}^{(-2)}$, $\xi_{0}^{(-4)} = \sigma^3 \ol{\sigma}^{-1} \bm{\xi}_{0}^{(-4)}$, and similar for their complex conjugates. Note that ${\bm{\xi}}_{0}^{(4)} \mapsto \wh{\bm{\xi}}_{0}^{(4)} = {\bm{\xi}}_{0}^{(4)}$ under a change of contact form, which justifies our notation for the tuple.
	\end{exa}
	
	We end the section with the following straightforward, yet fundamental, lemma.
	\begin{lem}
		A four-dimensional perturbed Fefferman space is an optical geometry with twisting non-shearing congruence of null geodesics.
	\end{lem}
	We shall soon characterise those optical structures that arise from perturbed Fefferman spaces.

	\section{Algebraically special perturbed Fefferman spaces}\label{sec:algsp}
	Not every twisting non-shearing congruence of null geodesics arises as a perturbed Fefferman space as introduced in the previous section, and the aim of the present one is to characterise those that do, at least for a certain class of CR data --- see Theorem \ref{thm:PetrovII+Bach}.
	
	Throughout, we consider a four-dimensional optical geometry $(\wt{\mc{M}},\wt{\mbf{c}},\wt{k}) \longrightarrow (\mc{M},H,J)$ with twisting non-shearing congruence of null geodesics $\wt{\mc{K}}$. We refer the reader to Section \ref{sec-optical} for the general setup and notation concerning these geometries. In addition, in the case of a perturbed Fefferman space $(\wt{\mc{M}},\wt{\mbf{c}}_{\wt{\xi}},\wt{k})$, since a perturbed Fefferman metric is a special case of an optical geometry with twisting non-shearing congruence of null geodesics, it will be convenient to write
	\begin{align*}
		\wt{g}_{\theta,\wt{\xi}} & = 4 \theta \odot \wt{\lambda} + h \, , &
		\wt{\lambda} & := \wt{\omega}_{\theta} - \tfrac{1}{3}\Rho \theta + \wt{\xi} \, ,
	\end{align*}
	where $\wt{\omega}_{\theta}$ is the induced Webster connection one-form on $\wt{\mc{M}}$ and $\Rho$ is the Schouten scalar defined by \eqref{eq:SchoutenSc}. Just as we did for the components of the perturbation one-form $\wt{\xi}$, we can Fourier expand the components of $\wt{\lambda}$ with respect to an adapted coframe. The relation between the respective Fourier coefficients is given by
	\begin{subequations}\label{eq:lmb2xi}
		\begin{align}
			\lambda_{\alpha}^{(0)} & = \i \sigma^{-1} \nabla_{\alpha} \sigma + \xi_{\alpha}^{(0)}   \, , &
			\lambda_{0}^{(0)} & = \i \sigma^{-1} \nabla_{0} \sigma + \xi_{0}^{(0)} - \tfrac{1}{3} \Rho  \, , &
			\\
			\lambda_{\alpha}^{(k)} & =  \xi_{\alpha}^{(k)} \, , &
			\lambda_{0}^{(k)} & =  \xi_{0}^{(k)} \, , & k \neq 0 \, ,
		\end{align}
	\end{subequations}
	with $\lambda_{\bar{\alpha}}^{(k)} = \ol{\lambda_{\alpha}^{(-k)}}$ and $\lambda_{0}^{(k)} = \ol{\lambda_{0}^{(-k)}}$ for any $k$. This means that $\lambda_{0}^{(0)}$ is real-valued --- here, with a slight abuse of notation, we view $\nabla_{0}$ and $\Rho$ as unweighted. We shall also treat the complex-valued coefficients $\lambda^{(k)}_{\alpha}$ abstractly as sections of $\mc{E}_{\alpha}$.
	
	We start with the following technical lemma.
	\begin{lem}[\cite{Lewandowski1990}]\label{lem:PetrovII}
		Let $(\wt{\mc{M}},\wt{\mbf{c}},\wt{k}) \longrightarrow (\mc{M},H,J)$ be a four-dimensional optical geometry with twisting non-shearing congruence of null geodesics. Let $\wt{g}_{\theta}$ be any metric in $\accentset{n.e}{\wt{\mbf{c}}}$ so that it takes the form \eqref{eq:can_met}. The following two conditions are equivalent:
		\begin{enumerate}
			\item $\langle\wt{k}\rangle$ is a repeated principal direction of the Weyl tensor $\wt{\Weyl}$, i.e.\ it satisfies
			\begin{align}
				\wt{\Weyl}(\wt{k}, \wt{v}, \wt{k}, \cdot) & = 0 \, , & \mbox{for any $\wt{v} \in \Gamma(\langle\wt{k}\rangle^\perp)$,} \label{eq:Petrov_typeII}
			\end{align}
			so that $\wt{\Weyl}$ is of Petrov type II or more degenerate.
			\item The component $\wt{\lambda}_{\alpha}$ takes the form
			\begin{align}
				\wt{\lambda}_{\alpha} & = \lambda_{\alpha}^{(0)} + \lambda_{\alpha}^{(-2)} \e^{-2 \i \phi} \, ,  \label{eq:solambdalf} 
			\end{align}
			for some complex-valued functions $\lambda_\alpha^{(0)}$ and $\lambda_\alpha^{(-2)}$ on $(\mc{M},H,J)$.
		\end{enumerate}
	\end{lem}
	
	\begin{proof}
		This is already computed in \cite{Lewandowski1990,TaghaviChabert2022}. Condition \eqref{eq:Petrov_typeII} is equivalent to $\wt{\Psi}_{1} = 0$ in the notation of Section \ref{sec:Petrov}. This can be seen to be equivalent to the second-order linear ordinary differential equation $\ddot{\wt{\lambda}}_{\alpha} + 2 \i \dot{\wt{\lambda}}_{\alpha} = 0$, which has general solution \eqref{eq:solambdalf}.
	\end{proof}
	
	Applying Lemma \ref{lem:PetrovII} to a perturbed Fefferman space gives:
	\begin{cor}\label{cor:PetrovII}
		Let $(\wt{\mc{M}},\wt{\mbf{c}}_{\wt{\xi}},\wt{k}) \longrightarrow (\mc{M},H,J)$ be a four-dimensional perturbed Fefferman space. Then $\langle \wt{k} \rangle$ is a repeated PND of the Weyl tensor of $\wt{\mbf{c}}_{\xi}$ if and only if the perturbation one-form $\wt{\xi}$ is determined by some CR data $\left( \bm{\xi}_{\alpha}^{(-2)},  \bm{\xi}_{\alpha}^{(0)}, [\nabla,\bm{\xi}_{0}^{(k)}] \right)_{k \in \mc{I}}$ for some subset $\mc{I} \subset \Z_{\geq0}$.
	\end{cor}
	
	Let us now derive further sufficient conditions for an optical geometry to be conformally isometric to a perturbed Fefferman space.
	\begin{prop}\label{prop:PetrovIII}
		Let $(\wt{\mc{M}},\wt{\mbf{c}},\wt{k}) \longrightarrow (\mc{M},H,J)$ be a four-dimensional optical geometry with twisting non-shearing congruence of null geodesics. Suppose that $\langle \wt{k} \rangle$ is a triple or quadruple PND of the Weyl tensor $\wt{\Weyl}$, i.e.\
		\begin{align}
			\wt{\Weyl}(\wt{k}, \wt{v}, \wt{w}, \cdot) & = 0 \, , & \mbox{for any $\wt{v}, \wt{w} \in \Gamma(\langle\wt{k}\rangle^\perp)$,} \label{eq:Petrov_typeIII}
		\end{align}
		so that $\wt{\Weyl}$ is of Petrov type III or more degenerate. Then $(\wt{\mc{M}},\wt{\mbf{c}},\wt{k})$ is locally conformally isometric to a perturbed Fefferman space $(\wt{\mc{M}}',\wt{\mbf{c}}'_{\wt{\xi}},\wt{k}') \longrightarrow (\mc{M},H,J)$, and the perturbation one-form $\wt{\xi}$ is determined by the CR data $\left(\bm{\xi}_{\alpha}^{(-2)} , \bm{\xi}_{\alpha}^{(0)}, [\nabla, \bm{\xi}_{0}^{(0)}, \bm{\xi}_{0}^{(2)}] \right)$ where
		\begin{align}
			\bm{\xi}_{0}^{(0)} & = \i  \nabla_{\alpha} \bm{\xi}^{\alpha}_{(0)} -   \i \nabla^{\alpha} \bm{\xi}_{\alpha}^{(0)} 
			+ 3 \bm{\xi}^{\alpha}_{(2)} \bm{\xi}_{\alpha}^{(-2)}  \, , \label{eq:coeff00} \\
			\bm{\xi}_{0}^{(2)} & = \i \accentset{\xi}{\nabla}_{\alpha}\bm{\xi}^{\alpha}_{(2)} \, , \label{eq:coeff02} 
		\end{align}
		where $\accentset{\xi}{\nabla}_{\alpha}$ is the partial Webster connection gauged by $\bm{\xi}_{\alpha}^{(0)}$.
	\end{prop}
	
	\begin{proof}
		As before, we work with a metric $\wt{g}_{\theta}$ in $\accentset{n.e.}{\wt{\mbf{c}}}$ with the form \eqref{eq:can_met}. Condition \eqref{eq:Petrov_typeIII} clearly implies \eqref{eq:Petrov_typeII}, so we know by Lemma \ref{lem:PetrovII} that the component $\wt{\lambda}_{\alpha}$ is given by \eqref{eq:solambdalf}. 
		
		We proceed to show that under the assumption \eqref{eq:Petrov_typeIII}, the component $\wt{\lambda}_0$ takes the form
		\begin{align}
			\wt{\lambda}_{0} & = \lambda_{0}^{(-2)} \e^{-2 \i \phi} + \lambda_{0}^{(0)} + \lambda_{0}^{(2)} \e^{2 \i \phi} \, , \label{eq:solambd0III}
		\end{align}
		where $\lambda_{0}^{(0)}$ and $\lambda_{0}^{(\pm 2)}$ are smooth functions on $\mc{M}$, with $\ol{\lambda_{0}^{(0)}} = \lambda_{0}^{(0)}$ and $\ol{\lambda_{0}^{(2)}} = \lambda_{0}^{(-2)}$. Using \cite[Appendix~A]{TaghaviChabert2022}, we compute the component $\wt{\Psi}_2$ of the Weyl tensor, as defined in Section \ref{sec:Petrov}, under the assumption \eqref{eq:solambdalf}. For convenience, we separate the real and imaginary parts of $\wt{\Psi}_2$. We find
		\begin{multline}\label{eq:RePsi2}
			\wt{\Psi}_{2} + \ol{\wt{\Psi}}_2 = \tfrac{1}{6} \ddot{\wt{\lambda}}\,_{0} - \tfrac{4}{3} \wt{\lambda}_0 + \tfrac{4}{3}  \i \left( \nabla_{\alpha} \lambda^{\alpha}_{(0)} -  \i \nabla^{\alpha} \lambda_{\alpha}^{(0)} + 3 \lambda^{\alpha}_{(2)} \lambda_{\alpha}^{(-2)}   
			+ \Rho \right)  \\
			+  \i  \left(   \nabla_{\alpha}  \lambda^{\alpha}_{(2)} - 2 \i  \lambda_{\alpha}^{(0)} \lambda^{\alpha}_{(2)} \right) \e^{2 \i \phi}  - \i  \left( \nabla^{\alpha} \lambda_{\alpha}^{(-2)} + 2 \i   \lambda^{\alpha}_{(0)} \lambda_{\alpha}^{(-2)} \right) \e^{-2 \i \phi} \, ,
		\end{multline}
		and
		\begin{equation}\label{eq:ImPsi2}
			\i ( \ol{\wt{\Psi}}_2 - \wt{\Psi}_2 ) = \dot{\wt{\lambda}}_{0} + \left( \nabla_{\alpha}{\lambda^{\alpha}_{(2)} } - 2 \i \lambda_{\alpha}^{(0)} \lambda^{\alpha}_{(2)} \right) \e^{2 \i \phi} 
			+ \left( \nabla^{\alpha}{ \lambda_{\alpha}^{(-2)} }  + 2 \i \lambda^{\alpha}_{(0)} \lambda_{\alpha}^{(-2)} \right) \e^{-2 \i \phi}  \, .
		\end{equation}
		Each of these parts must vanish separately. Differentiating $\eqref{eq:ImPsi2} = 0$ with respect to $\phi$, and substituting the resulting expression for $\ddot{\wt{\lambda}}_{0}$ into $\eqref{eq:RePsi2} = 0$, we can immediately solve for $\lambda_0$ algebraically, and find that it takes the form \eqref{eq:solambd0III}, where
		\begin{align*}
			\lambda_0^{(0)} & = \i  \nabla_{\alpha} \lambda^{\alpha}_{(0)} -   \i \nabla^{\alpha} \lambda_{\alpha}^{(0)} 
			+ 3 \lambda^{\alpha}_{(2)} \lambda_{\alpha}^{(-2)}   
			+ \Rho  \, , \\
			\lambda_0^{(2)} & = \i  \nabla_{\alpha}  \lambda^{\alpha}_{(2)} + 2 \lambda_{\alpha}^{(0)} \lambda^{\alpha}_{(2)} \, .
		\end{align*}
		Next, we use the relations \eqref{eq:lmb2xi} and the commutation relation \eqref{eq:CR_com1}, multiply the resulting expressions by $(\sigma \ol{\sigma})^{-1}$ and $\ol{\sigma}^{-2}$ respectively, and using \eqref{eq:Fourier_coef}, we arrive at \eqref{eq:coeff00} and \eqref{eq:coeff02} respectively. To complete the proof, we simply apply \cite[Lemma~5.7]{TaghaviChabert2023b}.
	\end{proof}
	
	\begin{rem}
		We note that both $\bm{\xi}_{0}^{(0)}$ and $\bm{\xi}_{0}^{(2)}$ transform as \eqref{eq:wtd_Fourier} as required.
	\end{rem}

	In the present context, the condition that the Weyl tensor is of Petrov type III or more degenerate has the following conformally invariant consequence:
	
	\begin{prop}\label{prop:PetrovIII2Bach}
		Let $(\wt{\mc{M}},\wt{\mbf{c}},\wt{k})$ be a four-dimensional optical geometry with twisting non-shearing congruence of null geodesics. Suppose that $\langle\wt{k}\rangle$ is a triple or quadruple PND of the Weyl tensor, i.e.\ \eqref{eq:Petrov_typeIII} holds.	Then the Bach tensor satisfies $\wt{\Bach}(\wt{k},\wt{k}) = 0$.
	\end{prop}
	
	\begin{proof}
		This is a long and tedious computation. Details are given in Appendix \ref{app:Bach_prop}. Formula \eqref{eq:BachPsi2} gives $\wt{\Bach}(\wt{k},\wt{k})$ as a linear combination of the real and imaginary parts of the component $\wt{\Psi}_2$ of the Weyl tensor and its derivatives with respect to $\wt{k}$. The result follows immediately.
	\end{proof}
	
	The converse of Proposition \ref{prop:PetrovIII2Bach} is not true in general, and this leads us to weaken the hypothesis of Proposition \ref{prop:PetrovIII}:
	\begin{thm}\label{thm:PetrovII+Bach}
		Let $(\wt{\mc{M}},\wt{\mbf{c}},\wt{k})$ be a four-dimensional optical geometry with twisting non-shearing congruence of null geodesics. Suppose that the Weyl tensor and the Bach tensor satisfy
		\begin{align}
			\wt{\Weyl}(\wt{k}, \wt{v}, \wt{k}, \cdot) & = 0 \, , & \mbox{for any $v \in \Gamma(\langle\wt{k}\rangle^\perp)$,} \label{eq:Weyl_deg} \\
			\wt{\Bach}(\wt{k},\wt{k}) & = 0 \, , \label{eq:Bach_deg} 
		\end{align}
		respectively. Then $(\wt{\mc{M}},\wt{\mbf{c}},\wt{k})$ is locally conformally isometric to a perturbed Fefferman space $(\wt{\mc{M}}',\wt{\mbf{c}}'_{\wt{\xi}},\wt{k}') \longrightarrow (\mc{M},H,J)$. More precisely, the perturbation one-form $\wt{\xi}$ is determined by the CR data $\left(\bm{\xi}_{\alpha}^{(0)}, \bm{\xi}_{\alpha}^{(-2)} , [\nabla,\bm{\xi}_{0}^{(0)}, \bm{\xi}_{0}^{(2)}], \bm{\xi}_{0}^{(4)} \right)$ where
		\begin{align}\label{eq:xi00PetrovII+Bach}
			\bm{\xi}_{0}^{(0)} & = \i  \nabla_{\alpha} \bm{\xi}^{\alpha}_{(0)} -   \i \nabla^{\alpha} \bm{\xi}_{\alpha}^{(0)} 
			+ 3 \bm{\xi}^{\alpha}_{(2)} \bm{\xi}_{\alpha}^{(-2)}  \, .
		\end{align}
	\end{thm}
	
	\begin{proof}
		Again we refer to the setup given at the beginning of this section. The line of arguments follows that of the proof of Proposition \ref{prop:PetrovIII}. In particular, condition \eqref{eq:Weyl_deg} already tells us that 
		$\wt{\lambda}_{\alpha}$ is given by \eqref{eq:solambdalf}.
		
		Next, we show that condition \eqref{eq:Bach_deg} forces the coefficient $\wt{\lambda}_0$ to take the form
		\begin{align}
			\wt{\lambda}_{0} & = \lambda_{0}^{(-4)} \e^{-4 \i \phi} + \lambda_{0}^{(-2)} \e^{-2 \i \phi} + \lambda_{0}^{(0)} + \lambda_{0}^{(2)} \e^{2 \i \phi} + \lambda_{0}^{(4)} \e^{4 \i \phi} \, , \label{eq:solambd0}
		\end{align}
		where $\lambda_{0}^{(2k)}$, $|k|=0,1,2$, are smooth functions on $\mc{M}$, real-valued for $k=0$, complex-valued for $k \neq 0$. By Proposition \ref{prop:PetrovIII}, we already know that this is true with $\lambda_{0}^{(4)} = 0$ when $\langle \wt{k} \rangle$ is a triple PND of $\wt{\Weyl}$, and Proposition \ref{prop:PetrovIII2Bach} tells us that this is consistent with \eqref{eq:Bach_deg}. Let us now assume that $\langle \wt{k} \rangle$ is a double PND of $\wt{\Weyl}$. This means that the component $\wt{\Psi}_2$ of $\wt{\Weyl}$ defined in \eqref{eq:NP-Psi} does not vanish. Using \eqref{eq:RePsi2} and \eqref{eq:ImPsi2} and taking further derivatives with respect to $\phi$, and plugging the resulting expressions into \eqref{eq:BachPsi2}, we find that condition \eqref{eq:Bach_deg} is equivalent to the fourth-order linear ordinary differential equation
		\begin{align*}
			\ddddot{\wt{\lambda}}_0 + 20 \ddot{\wt{\lambda}}_0 + 64 \wt{\lambda}_0 & = 64 \left( \i  \nabla_{\alpha} \lambda^{\alpha}_{(0)} -   \i \nabla^{\alpha} \lambda_{\alpha}^{(0)} 
			+ 3 \lambda^{\alpha}_{(2)} \lambda_{\alpha}^{(-2)}   
			+ \Rho  \right) \, ,
		\end{align*}
		on $\wt{\lambda}_0$. The general solution is given by \eqref{eq:solambd0} where
		\begin{align*}
			\lambda_0^{(0)} & = \i  \nabla_{\alpha} \lambda^{\alpha}_{(0)} -   \i \nabla^{\alpha} \lambda_{\alpha}^{(0)} 
			+ 3 \lambda^{\alpha}_{(2)} \lambda_{\alpha}^{(-2)}   
			+ \Rho  \, ,
		\end{align*}
		and $\lambda_0^{(-2)}$, $\lambda_0^{(-4)}$ are arbitrary complex-valued functions on $(\mc{M},H,J)$. The end of the proof is just as that of Proposition \ref{prop:PetrovIII}.
	\end{proof}

	\section{Distinguished almost Lorentzian scales}\label{sec:al_Lor_sc}
	In contrast with Section \ref{sec:algsp}, we now examine the consequences of the existence of a distinguished almost Lorentzian scale $\wt{\sigma}$ on a four-dimensional optical geometry with twisting non-shearing congruence of null geodesics. The idea is to impose an increasing number of differential conditions on $\wt{\sigma}$, which correspond to subsystems of the Einstein equations as defined in Section \ref{sec:confsc}. These will be shown to reduce to invariant differential equations on the CR leaf space, and, in the context of perturbed Fefferman spaces, on the CR data. In passing, we will derive sufficient conditions on $\wt{\sigma}$ for  an optical geometry to be conformally isometric to a perturbed Fefferman space --- see Corollary \ref{cor:al_half-Einstein}.
	
	We will also draw conclusions regarding the analytic properties of the underlying CR structure as a result of these almost Lorentzian scales. Proposition \ref{prop:wkhlfEinsc}, in particular, is a conceptual approach to \emph{Kerr's complex coordinate}, while Theorems \ref{thm:realisableCRMax} and \ref{thm:realisableCRPetrovII} formalise the results of \cite{Lewandowski1990,Hill2008} regarding CR embeddability.
	
	\subsection{Preliminary results}\label{sec:prel_zero_set}
	We start from a four-dimensional optical geometry $(\wt{\mc{M}},\wt{\mbf{c}},\wt{k}) \rightarrow (\mc{M},H,J)$ with twisting non-shearing congruence of null geodesics $\wt{\mc{K}}$  --- as before, the associated Robinson structure will be denoted $(\wt{N},\wt{K})$, so that $\langle \wt{k} \rangle$ is a PND of the Weyl tensor, and for each pseudo-Hermitian structure $\theta$ with Levi form $h$, the corresponding metric $\wt{g}_{\theta} \in \accentset{n.e.}{\wt{\mbf{c}}}$ takes the form \eqref{eq:can_met} for some adapted coframe.
	
	Let $\wt{\sigma}$ be an almost Lorentzian scale that satisfies
	\begin{align}\label{eq:RickkscFef}
		\wt{k}{}^{a} \wt{k}{}^{b} \left( \wt{\nabla}_{a} \wt{\nabla}_{b} \wt{\sigma} + \wt{\Rho}_{a b} \wt{\sigma} \right) & = 0 \, .
	\end{align}
	By \cite[Proposition~6.1]{TaghaviChabert2023b}, and working on a suitable subset $\wt{\mc{U}}$ of $\wt{\mc{M}}$, we find
	\begin{align*}
		\wt{\sigma} & = \cos \phi \cdot \wt{\sigma}_{\theta} \, ,
	\end{align*}
	where $\phi$ is an affine parameter along the geodesics of $\wt{\mc{K}}$ and $\wt{\sigma}_{\theta}$ is the Lorentzian scale of a metric $\wt{g}_{\theta}$ in $\accentset{n.e.}{\wt{\mbf{c}}}$ for some contact form $\theta$ for $(H,J)$. The zero set $\wt{\mc{Z}}$ of $\wt{\sigma}$ is thus the union of sections of $\varpi : \wt{\mc{U}} \rightarrow \mc{U} := \varpi (\wt{\mc{U}}) \subset \mc{M}$ parametrised by the integers $\Z$, and off $\wt{\mc{Z}}$, the metric $\wt{g} = \wt{\sigma}{}^{-2} \wt{\bm{g}}$ takes the form
	\begin{align}\label{eq:gsec2gth}
		\wt{g} & = \sec^2 \phi \cdot \wt{g}_{\theta} \, , & \mbox{for $\phi \neq \tfrac{2 k + 1}{2}\pi$, $k\in \Z$,}
	\end{align}
	and its Ricci tensor satisfies $\wt{\Ric}(\wt{k},\wt{k}) = 0$. For the purpose of our subsequent computations, we shall express the metric $\wt{g}_{\theta}$ in the form \eqref{eq:can_met}.
	
	\begin{rem}
		When $(\wt{\mc{M}},\wt{\mbf{c}},\wt{k}) \rightarrow (\mc{M},H,J)$ is a perturbed Fefferman space, we can refine the above statements. By \cite[Proposition 6.4]{TaghaviChabert2023b}, any choice of nowhere vanishing density $\sigma \in \Gamma(\mc{E}(1,0))$ on $\mc{M}$ determines an almost Lorentzian scale $\wt{\sigma}$ that satisfies \eqref{eq:RickkscFef}, and conversely, any such $\wt{\sigma}$ arises from a unique nowhere vanishing density $\sigma \in \Gamma(\mc{E}(1,0))$. In this case, the zero set of $\wt{\sigma}$ consists of the union $\wt{\mc{Z}} = \wt{\mc{Z}}_+ \cup \wt{\mc{Z}}_-$, where $\wt{\mc{Z}}_\pm := [\pm \i \sigma^{-1}]: \mc{M} \rightarrow \wt{\mc{M}}$, i.e.\ two cross-sections of Fefferman's circle bundle.
	\end{rem}

	For simplicity, we temporarily assume that $\wt{k}$ is a repeated PND of the Weyl tensor of $\wt{\mbf{c}}$, so that by Lemma \ref{lem:PetrovII}, $\wt{\lambda}$ appearing in \eqref{eq:can_met2} is given by
	\begin{align}\label{eq:lambda_PetrovII}
		\wt{\lambda} & = \d \phi + \left( \lambda_{\alpha}^{(0)} + \lambda_{\alpha}^{(-2)} \e^{-2\i \phi} \right) \theta^{\alpha} + \left( \lambda_{\bar{\alpha}}^{(0)} + \lambda_{\bar{\alpha}}^{(2)} \e^{2\i \phi} \right) \ol{\theta}{}^{\bar{\alpha}} + \wt{\lambda}_{0} \theta \, ,
	\end{align}
	for some complex-valued functions $\lambda_{\alpha}^{(0)}$, $\lambda_{\alpha}^{(-2)}$ on $\mc{M}$, and real-valued function $\wt{\lambda}_{0}$ on $\wt{\mc{M}}$. Suppose that $\wt{\sigma}$ satisfies
	\begin{align}\label{eq:deg_sp_L_sc}
		\wt{k}{}^{a} \wt{v}{}^{b} \left( \wt{\nabla}_{a} \wt{\nabla}_{b} \wt{\sigma} + \wt{\Rho}_{a b} \wt{\sigma} \right) & = 0 \, , & \mbox{for any $\wt{v} \in \Gamma (\langle \wt{k} \rangle^\perp)$.}
	\end{align}
	Then off $\wt{\mc{Z}}$, the Ricci tensor satisfies $\wt{\Ric}(\wt{k},\wt{v}) = 0$ for any $\wt{v} \in \Gamma(\langle \wt{k} \rangle^\perp)$. This, under the assumption that $\langle \wt{k} \rangle$ is a repeated PND of the Weyl tensor, means that
	\begin{align}\label{eq:step1}
		2 \lambda_\alpha ^{(-2)} =  \lambda_{\alpha}^{(0)} =: \lambda_\alpha \, ,
	\end{align}
	as was computed in \cite{TaghaviChabert2022} (see \cite[Remark~8.2]{TaghaviChabert2022} and also \cite{Lewandowski1990}), and likewise, this condition is sufficient for $\wt{\sigma}$ to satisfy \eqref{eq:deg_sp_L_sc}. Summarizing:
	\begin{lem}\label{lem:deg_sp_L_sc}
		Let $(\wt{\mc{M}},\wt{\mbf{c}},\wt{k}) \rightarrow (\mc{M},H,J)$ be a four-dimensional optical geometry with twisting non-shearing congruence of null geodesics. Suppose that $\langle \wt{k} \rangle$ is a repeated PND of the Weyl tensor. Let $\wt{\sigma}$ be an almost Lorentzian scale that satisfies \eqref{eq:RickkscFef}, so that off the zero set $\wt{\mc{Z}}$ of $\wt{\sigma}$, $\wt{g} = \wt{\sigma}^{-2} \wt{\bm{g}}$ is a smooth metric. Then $\wt{\sigma}$ satisfies \eqref{eq:deg_sp_L_sc} if and only if off $\wt{\mc{Z}}$, the metric $\wt{g}$ takes the form
		\begin{align}\label{eq:strongNS_2}
			\wt{g} & = \sec^2 \phi \cdot \left( 4 \theta \odot \left( \d \phi + (1 + \tfrac{1}{2} \e^{-2 \i \phi}  )\lambda_{\alpha} \theta^{\alpha} + (1 + \tfrac{1}{2} \e^{2 \i \phi}  )\lambda_{\bar{\alpha}} \ol{\theta}{}^{\bar{\alpha}}  + \wt{\lambda}_{0} \theta \right) + h \right) \, ,
		\end{align}
		for some $\lambda_{\alpha} \in \Gamma(\mc{E}_{\alpha})$ and real-valued function $\wt{\lambda}_{0}$ on $\wt{\mc{M}}$, and where $\theta$ is a pseudo-Hermitian structure with Levi form $h$, and $\phi$ is an affine parameter along $\wt{\mc{K}}$.
	\end{lem}
	
	\subsection{Almost weakly half-Einstein scales}
	Next, we assume that our almost Lorentzian scale $\wt{\sigma}$ is almost weakly half-Einstein, i.e.\ $\wt{\sigma}$ satisfies \eqref{eq:alwkEins}. By Theorem \ref{thm:GSalpha}, the condition that $\langle \wt{k} \rangle$ is a repeated PND follows immediately, so need not be added. In particular, the premises of Lemma \ref{lem:deg_sp_L_sc} hold. In addition, off $\wt{\mc{Z}}$, we have $\wt{\Ric}(\wt{w},\wt{w}) = 0$ for any $\wt{w} \in \Gamma(\wt{N})$, which means that
	\begin{align}\label{eq:step2}
		\nabla_{\alpha} \lambda_\beta - \i \lambda_\alpha \lambda_\beta - \WbA_{\alpha \beta} = 0 \, ,
	\end{align}
	as follows from \cite[equation~(A.19)]{TaghaviChabert2022}. Again, equation \eqref{eq:step2} will guarantee that $\wt{\sigma}$ satisfies \eqref{eq:alwkEins}.
	We now conclude:
	
	\begin{prop}\label{prop:alwkhalfE}
		Let $(\wt{\mc{M}},\wt{\mbf{c}},\wt{k}) \longrightarrow (\mc{M},H,J)$ be a four-dimensional optical geometry with twisting non-shearing congruence of null geodesics. Let $\wt{\sigma}$ be an almost Lorentzian scale that satisfies \eqref{eq:RickkscFef}, so that off the zero set $\wt{\mc{Z}}$ of $\wt{\sigma}$, $\wt{g} = \wt{\sigma}^{-2} \wt{\bm{g}}$ is a smooth metric. Then $\wt{\sigma}$ is an almost weakly half-Einstein scale if and only if off $\wt{\mc{Z}}$, the metric $\wt{g}$ takes the form \eqref{eq:strongNS_2} where $\lambda_{\alpha}$ satisfies \eqref{eq:step2}.
	\end{prop}
	
	As a corollary, we obtain the next proposition, which is a reformulation of a classical result that has its roots in the work of Robinson, Trautman and Kerr, to name but a few, but was subsequently interpreted in the language of CR geometry in \cite{Lewandowski1990a,Hill2008}.
	\begin{prop}\label{prop:wkhlfEinsc}
		Locally, a four-dimensional optical geometry $(\wt{\mc{M}},\wt{\mbf{c}},\wt{k}) \longrightarrow (\mc{M},H,J)$ with twisting non-shearing congruence of null geodesics admits an almost weakly half-Einstein scale if and only if $(\mc{M},H,J)$ admits a strongly independent CR function.
	\end{prop}
	
	\begin{proof}
		If $(\wt{\mc{M}},\wt{\mbf{c}},\wt{k})$ admits an almost weakly half-Einstein scale, then off its zero set $\wt{\mc{Z}}$, the Ricci tensor of the metric $\wt{g} = \wt{\sigma}^{-2}\wt{\bm{g}}$ satisfies \eqref{eq:wk_hlf_Einstein}. By Proposition \ref{prop:alwkhalfE}, there exists $\lambda_{\alpha} \in \Gamma(\mc{E}_{\alpha})$ that satisfies \eqref{eq:step2}, which is none other than the Webster--Weyl equation \eqref{eq:Webster--Weyl}. By Theorem \ref{thm:Webster--Weyl}, we can conclude that $(\mc{M},H,J)$ admits a strongly independent CR function. The converse works just analogously.
	\end{proof}

	Let us now specialise to the case where we start from a perturbed Fefferman space. This will allow us to determine the nature of the Webster--Weyl equation in this context.
	\begin{thm}\label{thm:wkhlfEinsc}
		A four-dimensional perturbed Fefferman space $(\wt{\mc{M}},\wt{\mbf{c}}_{\wt{\xi}},\wt{k}) \longrightarrow (\mc{M},H,J)$ admits an almost weakly half-Einstein scale $\wt{\sigma}$ if and only if $\wt{\xi}$ is determined by some CR data $\left( \bm{\xi}_{\alpha}^{(-2)},  \bm{\xi}_{\alpha}^{(0)}, [\nabla,\bm{\xi}_{0}^{(k)}] \right)_{k \in \mc{I}}$ for some subset $\mc{I} \subset \Z_{\geq0}$, and there exists a nowhere vanishing density $\sigma \in \Gamma(\mc{E}(1,0))$ that satisfies
		\begin{align}
			& \accentset{\xi}{\nabla}_{\bar{\alpha}} \sigma + 2\i \xi_{\bar{\alpha}}^{(2)} \sigma = 0 \, , \label{eq:xi2_a} \\
			& \accentset{\xi}{\nabla}_{\alpha} \accentset{\xi}{\nabla}_{\beta} \sigma + \i  \WbA_{\alpha \beta} \sigma = 0 \, , \label{eq:exactWW_Ein_a}	
		\end{align}
		where $\xi^{(-2)}_{\alpha} = \sigma \ol{\sigma}^{-1} \bm{\xi}_{\alpha}^{(-2)}$ and $\accentset{\xi}{\nabla}_{\alpha}$ denotes any partial Webster connection gauged by $\bm{\xi}_{\alpha}^{(0)}$.
		
		Moreover, the zero set $\wt{\mc{Z}}$ of $\wt{\sigma}$ consists of the union $\wt{\mc{Z}}_+ \cup \wt{\mc{Z}}_-$ of the sections $\wt{\mc{Z}}_{\pm} = [\pm \i \sigma^{-1}] : \mc{M} \rightarrow \wt{\mc{M}}$. Off $\wt{\mc{Z}}$, $\wt{\sigma}$ determines a weakly half-Einstein metric $\wt{g} = \sec^2 \phi \cdot \wt{g}_{\theta, \wt{\xi}}$, where $\phi$ is the fibre coordinate determined by $\sigma$ and $\wt{g}_{\theta, \wt{\xi}}$ is the perturbed Fefferman metric associated to the contact form $\theta = (\sigma \ol{\sigma})^{-1} \bm{\theta}$.
	\end{thm}
	
	\begin{proof}
		We have two immediate consequences of the existence of an almost weakly half-Einstein scale $\wt{\sigma}$. On the one hand, by Theorem \ref{thm:GSalpha}, $\langle \wt{k} \rangle$ is a repeated PND of the Weyl tensor of $\wt{\mbf{c}}_{\wt{\xi}}$, and so, by Corollary \ref{cor:PetrovII}, the perturbation one-form $\wt{\xi}$ is determined by some CR data $\left( \bm{\xi}_{\alpha}^{(-2)},  \bm{\xi}_{\alpha}^{(0)}, [\nabla,\bm{\xi}_{0}^{(k)}] \right)_{k \in \mc{I}}$ for some subset $\mc{I} \subset \Z_{\geq0}$. On the other hand, by \cite[Proposition~6.4]{TaghaviChabert2023b}, there exists a nowhere vanishing density $\sigma$ of weight $(1,0)$, and thus a pseudo-Hermitian structure $\theta = (\sigma \ol{\sigma})^{-1} \bm{\theta}$ together with a fibre coordinate $\phi$ such that the metric $\wt{g} = \wt{\sigma}^{-2} \wt{\bm{g}}$ takes the form
		\begin{align}\label{eq:secpFeff}
			\wt{g} & = \sec^2 \phi \cdot \wt{g}_{\theta,\wt{\xi}} \, , & \mbox{for $\phi \neq \pm \tfrac{1}{2}\pi$,}
		\end{align}
		where $\wt{g}_{\theta,\wt{\xi}}$ is the perturbed Fefferman metric associated to $\theta$ and $\wt{\xi}$. The zero set $\wt{\mc{Z}}$ can be identified as the cross-sections $[\pm \i \sigma^{-1}]$. The metric $\wt{g}$ was already obtained in Proposition \ref{prop:alwkhalfE} for more general optical geometries, and it suffices to apply relations \eqref{eq:lmb2xi} to \eqref{eq:step1} and \eqref{eq:step2} to derive
		\begin{align*}
			\accentset{\xi}{\nabla}_{\alpha} \sigma +  2 \i \xi_\alpha ^{(-2)} \sigma & = 0 \, ,  &
			\accentset{\xi}{\nabla}_{\alpha} \accentset{\xi}{\nabla}_{\beta} \sigma + \i \WbA_{\alpha \beta} \sigma = 0 \, ,
		\end{align*}
		respectively. The second of these equations is just \eqref{eq:exactWW_Ein_a}, and since we are using the Webster connection compatible with $\sigma \ol{\sigma}$, the complex conjugate of the first is just \eqref{eq:xi2_a}. This determines the densities $\bm{\xi}^{(-2)}_{\alpha}$ and $\bm{\xi}^{(0)}_{\alpha}$. We are then free to choose the remaining densities $\bm{\xi}^{(k)}_{0}$, where $k \in \mc{I}$ for any choice of subset $\mc{I} \subset \Z_{\geq0}$.
		
		For the converse, we assume that we have the CR data as given and a density $\sigma$ of weight $(1,0)$ that satisfies \eqref{eq:xi2_a} and \eqref{eq:exactWW_Ein_a}. By \cite[Proposition~6.4]{TaghaviChabert2023b}, $\sigma$ gives rise to an almost Lorentzian scale $\wt{\sigma}$ satisfying \eqref{eq:RickkscFef} with zero set $\wt{\mc{Z}}$ determined by $[\pm \i \sigma^{-1}]$. Off $\wt{\mc{Z}}$, the corresponding metric $\wt{g} = \wt{\sigma}^{-2} \wt{\bm{g}}$ is of the form \eqref{eq:secpFeff}, where $\phi$ and $\wt{g}_{\theta,\wt{\xi}}$ are the fibre coordinate and perturbed Fefferman metric determined by $\sigma$ respectively.  Choose $\nabla$ to be compatible with $\theta = (\sigma \ol{\sigma})^{-1} \bm{\theta}$, and trivialise \eqref{eq:xi2_a} and \eqref{eq:exactWW_Ein_a} to obtain \eqref{eq:step1} and \eqref{eq:step2}. Using the identification \eqref{eq:lmb2xi} and applying Proposition \ref{prop:alwkhalfE} shows that $\wt{\sigma}$ is an almost weakly half-Einstein scale $\wt{\sigma}$ as required.
	\end{proof}
	
	In two special cases, we obtain the following consequences:
	\begin{cor}\label{cor:wkhlfEinsc}
		Let $(\wt{\mc{M}},\wt{\mbf{c}}_{\wt{\xi}},\wt{k}) \longrightarrow (\mc{M},H,J)$ be a four-dimensional perturbed Fefferman space.
		\begin{enumerate}
			\item Suppose that $\wt{\xi}$ is determined by some CR data $\left(\bm{\xi}_{\alpha}^{(0)}, [\nabla,\bm{\xi}_{0}^{(k)}] \right)_{k \in \mc{I}}$ for some subset $\mc{I} \subset \Z_{\geq0}$. Then locally, $(\wt{\mc{M}},\wt{\mbf{c}}_{\wt{\xi}},\wt{k})$ admits an almost weakly half-Einstein scale $\wt{\sigma}$ if and only if $(\mc{M},H,J)$ is realisable.
			\item Suppose that $\wt{\xi}$ is determined by some CR data $\left(\bm{\xi}_{0}^{(k)} \right)_{k \in \mc{I}}$ for some subset $\mc{I} \subset \Z_{\geq0}$. Then locally, $(\wt{\mc{M}},\wt{\mbf{c}}_{\wt{\xi}},\wt{k})$ admits an almost weakly half-Einstein scale $\wt{\sigma}$ if and only if $(\mc{M},H,J)$ is CR--Einstein.
		\end{enumerate} 
	\end{cor}
	
	\begin{proof}
		For the first part, \eqref{eq:exactWW_Ein_a} and \eqref{eq:xi2_a}, with $\bm{\xi}_{\alpha}^{(-2)} = 0$, are simply \eqref{eq:linWW1d3} and \eqref{eq:linWW2d3} respectively with $\tau = \sigma$. By Theorem \ref{thm:real2dens_d3} the result follows. For the second part, if in addition $\bm{\xi}_{\alpha}^{(0)} = 0$, \eqref{eq:exactWW_Ein_a} and \eqref{eq:xi2_a} are the defining equations for a CR--Einstein structure --- see Remark \ref{rem:CR-Eins}
	\end{proof}
	
	\begin{rem}
		The second part of Corollary \ref{cor:wkhlfEinsc} also tells us that the underlying CR structure is flat. Weaker conditions are certainly possible if one assumes that $\xi_{\bar{\alpha}}^{(0)}$ or $\xi_{\bar{\alpha}}^{(2)}$ are `exact'. We leave the details to the reader.
	\end{rem}
	
	\subsection{Almost weakly half-Einstein scales and null vacuum Maxwell two-forms}
	Drawing from the terminology used in mathematical physics, a \emph{vacuum Maxwell two-form} on a four-dimensional Lorentzian conformal manifold $(\wt{\mc{M}},\wt{\mbf{c}})$ is a two-form $\wt{F}$ that is both closed and co-closed, i.e.\
	\begin{align}\label{eq:vacMax}
		\d \wt{F} & = 0 \, , & \d \wt{\star} \wt{F} & = 0 \, .
	\end{align}
	Here, $\wt{\star}$ is the (conformally invariant) Hodge dual operator on two-forms on $\mc{M}$. Such a two-form will be said to be \emph{null} or \emph{algebraically special} if there exists a null vector field $\wt{k}$ such that
	\begin{align}\label{eq:null2f}
		\wt{k} \hook \wt{F} & = 0 \, , & \wt{\bm{\kappa}} \wedge \wt{F} & = 0 \, .
	\end{align}
	where $\wt{\bm{\kappa}} = \wt{\bm{g}} (\wt{k} , \cdot )$. In this case, we shall refer to the span of $\wt{k}$ as a \emph{repeated principal null direction (PND)} of $\wt{F}$. Both \eqref{eq:vacMax} and \eqref{eq:null2f} are conformally invariant.
	
	Condition \eqref{eq:null2f} tells us that we can write $\wt{F} = \wt{F}^+ + \wt{F}^-$ where $\wt{F}^\pm$ are simple complex two-forms satisfying $\wt{F}^\pm = \ol{\wt{F}^\mp}$ and $\wt{\star} \wt{F}^\pm = \pm \i \wt{F}^\pm$. More precisely, a PND $\wt{k}$ of $\wt{F}$ defines an optical structure $\wt{K} := \langle \wt{k} \rangle$ on $(\wt{\mc{M}},\wt{\mbf{c}})$ with almost Robinson structure $(\wt{N},\wt{K})$, and $\wt{F}^+ \in \Gamma(\bigwedge^2 \Ann(\wt{N}))$. By \eqref{eq:vacMax}, $\wt{F}^\pm$ are closed, so $\wt{N}$ is involutive, which is equivalent to $\wt{k}$ being tangent to a non-shearing congruence of null geodesics $\wt{\mc{K}}$. This is essentially the content of \emph{Mariot's Theorem} \cite{Mariot1954}: a null vacuum Maxwell two-form locally always gives rise to a non-shearing congruence of null geodesics. In particular, the leaf space of $\wt{\mc{K}}$ is a CR three-manifold $(\mc{M},H,J)$, contact if the congruence is twisting, and by \eqref{eq:vacMax} and \eqref{eq:null2f}, using a suitable orientation on $\wt{\mc{M}}$, $\wt{F}^+$ is the pullback of a closed section of the canonical bundle of $(\mc{M},H,J)$.
	
	This already suggests that a converse of Mariot's Theorem will not in general be possible: starting from a twisting non-shearing congruence of null geodesics, a null vacuum Maxwell two-form would have to arise from a closed section of the canonical bundle of a contact CR three-manifold. But the existence is in general not guaranteed unless $(\mc{M},H,J)$ is realisable. In particular, if one works in the \emph{analytic} category, the CR three-manifold is certainly realisable and admits nowhere vanishing closed sections, and one can construct analytic null vacuum Maxwell two-forms. This is known as \emph{Robinson's Theorem} \cite{Robinson1961}. Tafel \cite{Tafel1985,Tafel1986} showed that Robinson's result may fail otherwise, and indeed, highlighted its failure in terms of a Lewy operator on the CR three-manifold.
	
	The following result is a reformulation of \cite{Hill2008} and \cite{Schmalz2019}.
	\begin{thm}\label{thm:realisableCRMax}
		Let $(\wt{\mc{M}},\wt{\mbf{c}},\wt{k}) \rightarrow (\mc{M},H,J)$ be a four-dimensional optical geometry with twisting non-shearing congruence of null geodesics. Then locally, $\wt{\mbf{c}}$ admits an almost weakly half-Einstein scale and $\langle \wt{k} \rangle$ is a repeated PND of a nowhere vanishing null vacuum Maxwell two-form if and only if $(\mc{M},H,J)$ is realisable.
	\end{thm}
	
	\begin{proof}
		Suppose that $(\wt{\mc{M}},\wt{\mbf{c}}, \wt{k})$ admits an almost weakly half-Einstein scale. Then Proposition \ref{prop:wkhlfEinsc} provides a CR function. On the other hand, the nowhere vanishing null Maxwell two-form  yields a nowhere vanishing closed section of the canonical bundle. By Theorem \ref{thm:Jacobowitz}, realisability follows.
		
		Conversely, suppose $(\mc{M},H,J)$ is realisable. Then by Proposition \ref{prop:wkhlfEinsc}, $(\wt{\mc{M}},\wt{\mbf{c}}, \wt{k})$ admits an almost weakly half-Einstein scale. In addition, there exist nowhere vanishing closed sections of $\mc{C}$ on $(\mc{M},H,J)$, and any of them pull back to nowhere vanishing null vacuum Maxwell two-forms on $(\wt{\mc{M}},\wt{\mbf{c}}, \wt{k})$.
	\end{proof}

	\subsection{Almost half-Einstein scales}
	Let us restrict our almost Lorentzian scale $\wt{\sigma}$ on $(\wt{\mc{M}},\wt{\mbf{c}},\wt{k}) \rightarrow (\mc{M},H,J)$ further by assuming that $\wt{\sigma}$ is now an almost half-Einstein scale, i.e.\ it satisfies \eqref{eq:alwkEins} and \eqref{eq:alcstRicSc}. Then, in particular, Proposition \ref{prop:alwkhalfE} must hold, i.e.\ off the zero set $\wt{\mc{Z}}$ of $\wt{\sigma}$, the metric $\wt{g} = \wt{\sigma}^{-2} \wt{\bm{g}}$ is given by \eqref{eq:strongNS_2} with $\lambda_{\alpha}$ satisfying \eqref{eq:step2}. By Proposition \ref{prop:scale2metric}, condition \eqref{eq:alcstRicSc} tells us that off $\wt{\mc{Z}}$, the Ricci scalar is constant, with $\wt{\Sc} = 4 \wt{\Lambda}$, which is equivalent to a second-order ordinary differential equation on the component $\wt{\lambda}_{0}$:
	\begin{equation}\label{eq:ODE1}
		\cos^2 \phi \ddot{\wt{\lambda}}_0 + 6 \cos \phi \sin \phi \dot{\wt{\lambda}}_0 + \left( 12 - 8 \cos^2 \phi \right) \wt{\lambda}_{0} 
		+ A_2 \cos^2 \phi 
		+ A_1 \cos \phi \sin \phi  + A_0 = 0 \, ,
	\end{equation}
	where
	\begin{subequations}
		\begin{align}\label{eq:coeffODE1}
			A_0 & = -  4  \wt{\Lambda} - 6   \lambda_{\alpha} \lambda^{\alpha} \, , \\
			A_1 & = 3 \left( \nabla_{\alpha} \lambda^{\alpha} + \nabla^{\alpha} \lambda_{\alpha} \right) \, , \\
			A_2 & =  8  \Rho  + 5  \i \left( \nabla_{\alpha} \lambda^{\alpha} - \nabla^{\alpha} \lambda_{\alpha} \right) - 6 \lambda_{\alpha} \lambda^{\alpha}  \, .
		\end{align}
	\end{subequations}
	Again, these equations can be derived from the computation of the curvature given in \cite{TaghaviChabert2022}.
	The general solution $\wt{\lambda}_{0}$ is given by \eqref{eq:ODE1sol} with coefficients \eqref{eq:coeff_lbd0} in the next proposition.
	
	\begin{prop}\label{prop:al_half-Einstein}
		Let $(\wt{\mc{M}},\wt{\mbf{c}},\wt{k}) \longrightarrow (\mc{M},H,J)$ be a four-dimensional optical geometry with twisting non-shearing congruence of null geodesics.  Let $\wt{\sigma}$ be an almost Lorentzian scale that satisfies \eqref{eq:RickkscFef}, so that off the zero set $\wt{\mc{Z}}$ of $\wt{\sigma}$, $\wt{g} = \wt{\sigma}^{-2} \wt{\bm{g}}$ is a smooth metric. Then $\wt{\sigma}$ is an an almost half-Einstein scale if and only if off $\wt{\mc{Z}}$, the metric $\wt{g}$ takes the form \eqref{eq:strongNS_2}, where $\lambda_{\alpha}$ satisfies \eqref{eq:step2}, and $\wt{\lambda}_{0}$ is given by
		\begin{align}\label{eq:ODE1sol}
			\wt{\lambda}_0 & =  \lambda_0^{(-4)} \e^{-4 \i \phi} + \lambda_0^{(-2)} \e^{-2 \i \phi}  + \lambda_0^{(0)} + \lambda_0^{(2)} \e^{2 \i \phi} + \lambda_0^{(4)} \e^{4 \i \phi}  \, ,
		\end{align}
		where
		\begin{subequations}\label{eq:coeff_lbd0}
			\begin{align}
				\lambda_{0}^{(0)} 
				& = \i \left(  \nabla_{\alpha} \lambda^{\alpha} - \nabla^{\alpha} \lambda_{\alpha} \right) + \tfrac{3}{4} \lambda_{\alpha} \lambda^{\alpha}  +  \Rho  + 6 \Re \left( \mu \right)  \, , \label{eq:step3a}\\
				\lambda_{0}^{(2)}
				& = 2 \lambda_{0}^{(4)} +  \tfrac{1}{4}  \i  \nabla_{\alpha} \lambda^{\alpha} + \tfrac{1}{2} \lambda_{\alpha} \lambda^{\alpha} + 2 \Re \left( \mu \right) \, , \label{eq:step3b}\\
				\lambda_{0}^{(4)}
				& = \tfrac{1}{8} \left(   4  \Rho  + 3  \i \left( \nabla_{\alpha} \lambda^{\alpha} -  \nabla^{\alpha} \lambda_{\alpha} \right) - 3 \lambda_{\alpha} \lambda^{\alpha}  - \tfrac{4}{3}  \wt{\Lambda}   \right) + \mu \, , \label{eq:step3c}
			\end{align}
		\end{subequations}
		for some arbitrary complex-valued function $\mu$, and where $\wt{\Lambda} = \tfrac{1}{4} \wt{\Sc}$ is constant.
	\end{prop}
	
	\begin{rem}
		We note that the function $\mu$ can also be written as
		\begin{align}\label{eq:const}
			\mu
			& := \tfrac{1}{16} \left( 2 \i \nabla^{\alpha} \lambda_{\alpha} - 4 \lambda_{\alpha} \lambda^{\alpha}  + c - \i b  \right) \, ,
		\end{align}
		for some arbitrary real-valued functions $b$ and $c$ on $\mc{M}$.
	\end{rem}
	
	Since $\wt{\lambda}_{\alpha} =  \lambda_{\alpha}^{(0)} + \lambda_{\alpha}^{(-2)} \e^{-2\i \phi}$ (see display \eqref{eq:lambda_PetrovII}) and $\wt{\lambda}_{0}$ is of the form \eqref{eq:ODE1sol}, applying \cite[Lemma~5.7]{TaghaviChabert2023b} leads to the following:
	\begin{cor}\label{cor:al_half-Einstein}
		Let $(\wt{\mc{M}},\wt{\mbf{c}},\wt{k}) \longrightarrow (\mc{M},H,J)$ be a four-dimensional optical geometry with twisting non-shearing congruence of null geodesics $\wt{\mc{K}}$. Suppose that $\wt{\mbf{c}}$ admits an almost half-Einstein scale. Then $(\wt{\mc{M}},\wt{\mbf{c}},\wt{k})$ is locally conformally isometric to a perturbed Fefferman space.
	\end{cor}
	
	The next theorem describes the CR data of the perturbed Fefferman space alluded to in Proposition \ref{prop:al_half-Einstein} and Corollary \ref{cor:al_half-Einstein}.
	\begin{thm}\label{thm:hlfEinsc}
		A four-dimensional perturbed Fefferman space $(\wt{\mc{M}},\wt{\mbf{c}}_{\wt{\xi}},\wt{k}) \longrightarrow (\mc{M},H,J)$ admits an almost half-Einstein scale $\wt{\sigma}$ if and only if $\wt{\xi}$ is determined by some CR data\linebreak $\left(\bm{\xi}_{\alpha}^{(-2)}, \bm{\xi}_{\alpha}^{(0)}, [\nabla , \bm{\xi}_{0}^{(0)}, \bm{\xi}_{0}^{(2)}], \bm{\xi}_{0}^{(4)} \right)$, and there exists a nowhere vanishing density $\sigma \in \Gamma(\mc{E}(1,0))$ that satisfies
		\begin{align}
			& \accentset{\xi}{\nabla}_{\bar{\alpha}} \sigma + 2\i \xi_{\bar{\alpha}}^{(2)} \sigma = 0 \, , \label{eq:xi2_b} \\
			& \accentset{\xi}{\nabla}_{\alpha} \accentset{\xi}{\nabla}_{\beta} \sigma + \i  \WbA_{\alpha \beta} \sigma = 0 \, , \label{eq:exactWW_Ein_b}	
		\end{align}
		where $\xi^{(-2)}_{\alpha} = \sigma \ol{\sigma}^{-1} \bm{\xi}_{\alpha}^{(-2)}$, $\accentset{\xi}{\nabla}_{\alpha}$ denotes any partial Webster connection gauged by $\bm{\xi}_{\alpha}^{(0)}$, and
		\begin{subequations}\label{eq:xi0}
			\begin{align}
				& \bm{\xi}_{0}^{(0)} =   \i \left(  \nabla_{\alpha} \bm{\xi}^{\alpha}_{(0)} -  \nabla^{\alpha} \bm{\xi}_{\alpha}^{(0)} \right) + 3 \bm{\xi}_{\alpha}^{(-2)} \bm{\xi}^{\alpha}_{(2)} + 6 \Re \left( \mu \right) \sigma^{-1} \ol{\sigma}^{-1}  \, , \label{eq:xi00} \\
				& \bm{\xi}_{0}^{(2)} =  \tfrac{1}{2}  \i  \accentset{\xi}{\nabla}_{\alpha} \bm{\xi}^{\alpha}_{(2)} + 2 \bm{\xi}_{0}^{(4)} \sigma^{-1} \ol{\sigma} + 2 \Re \left( \mu \right) \ol{\sigma}^{-2}  \, ,  \label{eq:xi02} \\
				& \bm{\xi}_{0}^{(4)} =  \tfrac{1}{2}  \left(  \tfrac{3}{4}  \left( \nabla_{\alpha} \left( \ol{\sigma}^{-1} \accentset{\xi}{\nabla}^{\alpha} \ol{\sigma} \right) + \nabla^{\alpha} \left( \sigma^{-1} \accentset{\xi}{\nabla}_{\alpha} \sigma \right) - \left(\sigma^{-1} \accentset{\xi}{\nabla}_{\alpha} \sigma \right) \left( \ol{\sigma}^{-1} \accentset{\xi}{\nabla}^{\alpha} \ol{\sigma} \right) \right) \right. \label{eq:xi04} \\
				& \qquad \qquad \left. - \tfrac{1}{2} \left( \nabla_{\alpha} \nu^{\alpha} + \nabla^{\alpha} \nu_{\alpha}  - \nu_{\alpha} \nu^{\alpha} \right) + \Rho \right)  \sigma^2 \ol{\sigma}^{-2} + \left( - \tfrac{1}{6}  \wt{\Lambda} + \mu \right) \sigma \ol{\sigma}^{-3} \, , \nonumber
			\end{align}
		\end{subequations}
		for some real constant $\wt{\Lambda}$ and complex-valued function $\mu$ on $(\mc{M},H,J)$, and where $\nabla_{\alpha} (\sigma \ol{\sigma}) = \nu_{\alpha} \sigma \ol{\sigma}$.
		
		Moreover, the zero set $\wt{\mc{Z}}$ of $\wt{\sigma}$ consists of the union $\wt{\mc{Z}}_+ \cup \wt{\mc{Z}}_-$ of the sections $\wt{\mc{Z}}_{\pm} = [\pm \i \sigma^{-1}] : \mc{M} \rightarrow \wt{\mc{M}}$. Off $\wt{\mc{Z}}$, $\wt{\sigma}$ determines a half-Einstein metric $\wt{g} = \sec^2 \phi \cdot \wt{g}_{\theta, \wt{\xi}}$, with Ricci scalar $4 \wt{\Lambda}$, where $\phi$ is the fibre coordinate determined by $\sigma$ and $\wt{g}_{\theta, \wt{\xi}}$ is the perturbed Fefferman metric associated to the contact form $\theta = (\sigma \ol{\sigma})^{-1} \bm{\theta}$.
	\end{thm}
	
	\begin{proof}
		Since $\wt{\sigma}$ is an almost half-Einstein scale, we already know by Proposition \ref{prop:al_half-Einstein} that the perturbation one-form $\wt{\xi}$ is determined by some CR data \begin{equation*}
			\left(\bm{\xi}_{\alpha}^{(0)}, \bm{\xi}_{\alpha}^{(-2)}, [\nabla , \bm{\xi}_{0}^{(0)}, \bm{\xi}_{0}^{(2)}], \bm{\xi}_{0}^{(4)} \right) \, .
		\end{equation*}
		Further the existence of a nowhere vanishing density $\sigma \in \Gamma(\mc{E}(1,0))$ that satisfies \eqref{eq:xi2_b} and \eqref{eq:exactWW_Ein_b} is guaranteed by Theorem \ref{thm:wkhlfEinsc}. To establish the precise form of the remaining densities, we look at the next relevant conditions on the Ricci curvature, which tells us by Proposition \ref{prop:al_half-Einstein} that the component $\wt{\lambda}_{0}$ is determined by \eqref{eq:ODE1sol} with \eqref{eq:coeff_lbd0}. For convenience, using \eqref{eq:step1}, we relabel the latter as
		\begin{align*}
			\lambda_{0}^{(0)} & =  \i \left(  \nabla_{\alpha} \lambda^{\alpha}_{(0)} - \nabla^{\alpha} \lambda_{\alpha}^{(0)} \right) + 3 \lambda_{\alpha}^{(-2)} \lambda^{\alpha}_{(2)}  +  \Rho + 6 \Re \left( \mu \right) \, , \\
			\lambda_{0}^{(2)} & =  2 \lambda_{0}^{(4)} +  \tfrac{1}{2}  \i  \nabla_{\alpha} \lambda^{\alpha}_{(2)} + \lambda_{\alpha}^{(0)} \lambda^{\alpha}_{(2)} + 2 \Re \left( \mu \right) \, , \\
			\lambda_{0}^{(4)} & = \tfrac{1}{8} \left(   4  \Rho  + 3  \i \left( \nabla_{\alpha} \lambda^{\alpha}_{(0)} -   \nabla^{\alpha} \lambda_{\alpha}^{(0)} \right) - 3 \lambda_{\alpha}^{(0)} \lambda^{\alpha}_{(0)}  - \tfrac{4}{3}  \wt{\Lambda}   \right) + \mu \, ,
		\end{align*}
		for some real constant $\wt{\Lambda}$ and complex-valued function $\mu$ on $(\mc{M},H,J)$. Using the relations \eqref{eq:lmb2xi}, these become
		\begin{subequations}
			\begin{align}
				\xi_{0}^{(0)}
				& =   \i \left(  \nabla_{\alpha} \xi^{\alpha}_{(0)} -  \nabla^{\alpha} \xi_{\alpha}^{(0)} \right) + 3 \xi_{\alpha}^{(-2)} \xi^{\alpha}_{(2)} + 6 \Re \left( \mu \right) \, , \label{eq:xi00t}\\
				\xi_{0}^{(2)} & =  2 \xi_{0}^{(4)} +  \tfrac{1}{2}  \i  \nabla_{\alpha} \xi^{\alpha}_{(2)} +  \i \sigma^{-1} \accentset{\xi}{\nabla}_{\alpha} \sigma \xi^{\alpha}_{(2)} + 2 \Re \left( \mu \right) \, , \label{eq:xi02t} \\
				\xi_{0}^{(4)} & = \tfrac{1}{2} \left( \tfrac{3}{4}  \left( \nabla_{\alpha} \left( \ol{\sigma}^{-1} \accentset{\xi}{\nabla}^{\alpha} \ol{\sigma} \right) + \nabla^{\alpha} \left( \sigma^{-1} \accentset{\xi}{\nabla}_{\alpha} \sigma \right) \right) \right.  \label{eq:xi04t} \\
				& \qquad \qquad \left. - \tfrac{3}{4} \left( \sigma^{-1} \accentset{\xi}{\nabla}_{\alpha} \sigma \right) \left( \ol{\sigma}^{-1} \accentset{\xi}{\nabla}^{\alpha} \ol{\sigma} \right)  - \tfrac{1}{3}  \wt{\Lambda}  + \Rho   \right) + \mu \, .  \nonumber
			\end{align}
		\end{subequations}
		Next, we multiply \eqref{eq:xi00t}, \eqref{eq:xi02t} and \eqref{eq:xi04t} through by $\sigma^{-1} \ol{\sigma}^{-1}$, $\ol{\sigma}^{-2}$ and $\sigma \ol{\sigma}^{-3}$ respectively, so that using \eqref{eq:gD_density}, we obtain \eqref{eq:xi00}, \eqref{eq:xi02} and \eqref{eq:xi04}, where we have added terms involving  $\nu_{\alpha} = (\sigma \ol{\sigma})^{-1} \nabla_{\alpha} (\sigma \ol{\sigma})$ to ensure CR invariance.
		
		The converse is just as in the proof of Theorem \ref{thm:wkhlfEinsc}, and we leave the details to the reader.
	\end{proof}

	\subsection{Almost pure radiation scales}
	While the Bach tensor of a conformal four-manifold is an obstruction to the existence of an almost Einstein scale, its vanishing is not necessary for the existence of an almost pure radiation scale. For our purposes, we shall need the following weaker condition.
	
	\begin{prop}\label{prop:W+E2B}
		Let $(\wt{\mc{M}},\wt{\mbf{c}},\wt{k})$ be a four-dimensional optical geometry with twisting non-shearing congruence of null geodesics that admits an almost pure radiation scale. Then $\langle \wt{k} \rangle$ is a repeated PND of the Weyl tensor, and the Bach tensor satisfies $\wt{\Bach}(\wt{k},\wt{k}) = 0$.
	\end{prop}
	
	\begin{proof}
		Let $\wt{\sigma}$ be an almost pure radiation scale. Then $\langle \wt{k} \rangle$ is a repeated PND of the Weyl tensor by Theorem \ref{thm:GSalpha}. Let $\wt{g} = \wt{\sigma}^{-2} \wt{\bm{g}}$ be the pure radiation metric off the zero set of $\wt{\sigma}$. By \eqref{eq:puradvac}, the Schouten tensor of $\wt{g}$ can be written as
		\begin{align}\label{eq:Schouten_redEin}
			\wt{\Rho}_{a b} & = \tfrac{1}{2}\wt{\Phi} \wt{\kappa}_{a} \wt{\kappa}_{b} + \tfrac{1}{6} \wt{\Lambda} \wt{g}_{a b} \, ,
		\end{align}
		for some smooth function $\wt{\Phi}$ on $\wt{\mc{M}}$, constant $\wt{\Lambda}$ and where $\wt{\kappa} = \wt{g}(\wt{k},\cdot)$. By \eqref{eq:Bach}, we must compute the expression
		\begin{align*}
			\wt{k}^a \wt{k}^b \wt{\Bach}_{a b} & = - \wt{k}^a \wt{k}^b \wt{\nabla}^{c} \wt{\Cot}_{a b c} + \wt{k}^a \wt{k}^b \wt{\Rho}^{c d} \wt{\Weyl}_{a c b d} \, ,
		\end{align*}
		The second term of the RHS of this expression is then clearly zero, and we are left with the first term, which, by the product rule, can be re-arranged as
		\begin{align}\label{eq:kkDY}
			\wt{k}^a \wt{k}^b \wt{\nabla}^{c} \wt{\Cot}_{a b c} & = \wt{\nabla}^{c} \left( \wt{k}^a \wt{k}^b \wt{\Cot}_{a b c} \right) - \left( \wt{\nabla}^{c} \wt{k}^a \right) \wt{k}^b \wt{\Cot}_{a b c} - \left( \wt{\nabla}^{c} \wt{k}^b \right) \wt{k}^a \wt{\Cot}_{a b c} \, .
		\end{align}
		Now, using equation \eqref{eq:Cotton_Rho} for the definition of the Cotton tensor, the product rule and the fact that $\wt{\Rho}_{a b} \wt{k}^{b} = \tfrac{1}{6} \wt{\Lambda}$ and the congruence generated by $\wt{k}$ is non-shearing, we get $\wt{k}^a \wt{k}^b \wt{\Cot}_{a b c} = 0$. Thus, the first term of the RHS of \eqref{eq:kkDY} must be zero. Next, using \eqref{eq:Cotton_Rho} again and substituting in \eqref{eq:Schouten_redEin}, we find $\left( \wt{\nabla}^{c} \wt{k}^a \right) \wt{k}^b \wt{\Cot}_{a b c} = 0$, so that the second term of the right-hand side of \eqref{eq:kkDY} vanishes, and so does the third term by the symmetries of $\wt{\Cot}_{a b c}$. Hence, $\wt{\Bach}(\wt{k},\wt{k}) = 0$ off the zero set of $\wt{\sigma}$. These conformally invariant conditions extend smoothly to the zero set of $\wt{\sigma}$.
	\end{proof}

	By Proposition \ref{prop:W+E2B} and Theorem \ref{thm:PetrovII+Bach}, we immediately obtain the following corollary:
	\begin{cor}\label{cor:puradsc}
		Let $(\wt{\mc{M}},\wt{\mbf{c}},\wt{k})\longrightarrow (\mc{M},H,J)$ be a four-dimensional optical geometry with twisting non-shearing congruence of null geodesics that admits an almost pure radiation scale. Then it is locally conformally isometric to a perturbed Fefferman space $(\wt{\mc{M}}',\wt{\mbf{c}}'_{\wt{\xi}},\wt{k}')\longrightarrow (\mc{M},H,J)$ where the perturbation one-form $\wt{\xi}$ is determined by some CR data $\left(\bm{\xi}_{\alpha}^{(-2)}, \bm{\xi}_{\alpha}^{(0)}, [\nabla,\bm{\xi}_{0}^{(0)}, \bm{\xi}_{0}^{(-2)}], \bm{\xi}_{0}^{(-4)} \right)$ with
		\begin{align*}
			\bm{\xi}_{0}^{(0)} & = \i  \nabla_{\alpha} \bm{\xi}^{\alpha}_{(0)} -   \i \nabla^{\alpha} \bm{\xi}_{\alpha}^{(0)} 
			+ 3 \bm{\xi}^{\alpha}_{(2)} \bm{\xi}_{\alpha}^{(-2)}  \, .
		\end{align*} 
	\end{cor}
	
	We can make our statement regarding the CR data of the previous corollary more precise.
	\begin{thm}\label{thm:pertFeffpurad}
		A four-dimensional perturbed Fefferman space $(\wt{\mc{M}},\wt{\mbf{c}}_{\wt{\xi}},\wt{k}) \longrightarrow (\mc{M},H,J)$ admits an almost pure radiation scale $\wt{\sigma}$ if and only if $\wt{\xi}$ is determined by some CR data\linebreak $\left(\bm{\xi}_{\alpha}^{(-2)}, \bm{\xi}_{\alpha}^{(0)}, [\nabla,\bm{\xi}_{0}^{(0)}, \bm{\xi}_{0}^{(-2)}], \bm{\xi}_{0}^{(-4)} \right)$  and there exists a nowhere vanishing density $\sigma \in \Gamma(\mc{E}(1,0))$ that satisfies
		\begin{align}
			& \accentset{\xi}{\nabla}_{\bar{\alpha}} \sigma + 2\i \xi_{\bar{\alpha}}^{(2)} \sigma = 0 \, , \label{eq:xi2} \\
			& \accentset{\xi}{\nabla}_{\alpha} \accentset{\xi}{\nabla}_{\beta} \sigma + \i  \WbA_{\alpha \beta} \sigma = 0 \, , \label{eq:exactWW_Ein}	
		\end{align}
		where $\xi^{(-2)}_{\alpha} = \sigma \ol{\sigma}^{-1} \bm{\xi}_{\alpha}^{(-2)}$, $\accentset{\xi}{\nabla}_{\alpha}$ denotes any partial Webster connection gauged by $\bm{\xi}_{\alpha}^{(0)}$, and
		\begin{subequations}\label{eq:xi0b}
			\begin{align}
				& \bm{\xi}_{0}^{(0)} = \i  \nabla_{\alpha} \bm{\xi}^{\alpha}_{(0)} -   \i \nabla^{\alpha} \bm{\xi}_{\alpha}^{(0)} 	+ 3 \, \bm{\xi}^{\alpha}_{(2)} \bm{\xi}_{\alpha}^{(-2)}   \, , \label{eq:xi00b} \\
				& \bm{\xi}_{0}^{(2)} =  - \tfrac{1}{2} \i \accentset{\xi}{\nabla}^{\alpha} \bm{\xi}_{\alpha}^{(-2)} + 4 \sigma \ol{\sigma}^{-1} \bm{\xi}_{0}^{(-4)}  \label{eq:xi02b} \, , \\
				& \bm{\xi}_{0}^{(4)} =  \tfrac{1}{2}  \left(  \tfrac{3}{4}  \left( \nabla_{\alpha} \left( \ol{\sigma}^{-1} \accentset{\xi}{\nabla}^{\alpha} \ol{\sigma} \right) + \nabla^{\alpha} \left( \sigma^{-1} \accentset{\xi}{\nabla}_{\alpha} \sigma \right) - \left(\sigma^{-1} \accentset{\xi}{\nabla}_{\alpha} \sigma \right) \left( \ol{\sigma}^{-1} \accentset{\xi}{\nabla}^{\alpha} \ol{\sigma} \right) \right) \right.  \label{eq:xi04b} \\
				& \qquad \qquad  \left. - \tfrac{1}{2} \left( \nabla_{\alpha} \nu^{\alpha} + \nabla^{\alpha} \nu_{\alpha}  - \nu_{\alpha} \nu^{\alpha} \right) + \Rho \right)  \sigma^2 \ol{\sigma}^{-2} + \left( - \tfrac{1}{6}  \wt{\Lambda} + \mu \right) \sigma \ol{\sigma}^{-3} \, , \nonumber
			\end{align}
		\end{subequations}
		for some real constant $\wt{\Lambda}$ and complex-valued function $\mu$ on $(\mc{M},H,J)$ satisfying $\ol{\mu}= -\mu$, and where $\nabla_{\alpha} (\sigma \ol{\sigma}) = \nu_{\alpha} \sigma \ol{\sigma}$, and in addition
		\begin{align}\label{eq:2ndCR}
			& \accentset{\xi}{\nabla}_{\alpha} \bm{\xi}_0^{(4)} - \left( \nu_{\alpha} - 2 \i \xi_{\alpha}^{(-2)} \right) \bm{\xi}^{(4)}_0 = 0 \, . 
		\end{align}
		
		Moreover, the zero set $\wt{\mc{Z}}$ of $\wt{\sigma}$ consists of the union $\mc{Z}_+ \cup \mc{Z}_-$ of the sections $\mc{Z}_{\pm} = [\pm \i \sigma^{-1}] : \mc{M} \rightarrow \wt{\mc{M}}$. Off $\wt{\mc{Z}}$, $\wt{\sigma}$ determines a pure radiation metric  $\wt{g} = \sec^2 \phi \cdot \wt{g}_{\theta, \wt{\xi}}$, with Ricci scalar $4 \wt{\Lambda}$, where $\phi$ is the fibre coordinate determined by $\sigma$ and $\wt{g}_{\theta, \wt{\xi}}$ is the perturbed Fefferman metric associated to the contact form $\theta = (\sigma \ol{\sigma})^{-1} \bm{\theta}$.
	\end{thm}

	\begin{proof}
		Since $\wt{\sigma}$ is an almost pure radiation scale, and in particular, an almost half-Einstein scale, we know by Theorem \ref{thm:hlfEinsc} that the perturbation one-form is determined by some CR data $\left(\bm{\xi}_{\alpha}^{(-2)}, \bm{\xi}_{\alpha}^{(0)}, [\nabla , \bm{\xi}_{0}^{(0)}, \bm{\xi}_{0}^{(2)}], \bm{\xi}_{0}^{(4)} \right)$, that there exists a nowhere vanishing density $\sigma \in \Gamma(\mc{E}(1,0))$ that satisfies \eqref{eq:xi2} and \eqref{eq:exactWW_Ein}, and that $\bm{\xi}_{0}^{(0)}$, $\bm{\xi}_{0}^{(2)}$ and $\bm{\xi}_{0}^{(4)}$ must be given by \eqref{eq:xi00}, \eqref{eq:xi02} and \eqref{eq:xi04} respectively.
		
		In comparison with the case where $\wt{\sigma}$ is merely an almost half-Einstein scale, we have another two additional requirements. The first one is that $\wt{\Ric}(\wt{k}, \cdot) = \wt{\Lambda} \wt{g}(\wt{k},\cdot)$. A computation using \cite{TaghaviChabert2022} will show that this is equivalent to the component $\wt{\lambda}_0$ being a solution to 
		\begin{multline}\label{eq:ODE2}
			\cos^2 \phi \ddot{\wt{\lambda}}_0 + 3 \cos \phi \sin \phi \dot{\wt{\lambda}}_0 +  \left( 3 + 4 \cos^2 \phi \right) \wt{\lambda}_{0}  \\
			+ B_4 \cos^4 \phi  
			+ A_1 \cos^3 \phi \sin \phi
			- \tfrac{1}{2} A_2	\cos^2 \phi 
			+ \tfrac{1}{4} A_0 = 0 \, .
		\end{multline}
		where $A_1$, $A_2$ and $A_3$ are given by \eqref{eq:coeffODE1}, and
		\begin{align*}
			B_4 = \left(   3 \i \left(  \nabla^{\alpha} \lambda_{\alpha} - \nabla_{\alpha} \lambda^{\alpha} \right) - 12  \lambda_{\alpha} \lambda^{\alpha}  \right) \, .
		\end{align*} Plugging the solution \eqref{eq:ODE1sol} with coefficients \eqref{eq:coeff_lbd0} into \eqref{eq:ODE2} tells us that the function $c$ given in \eqref{eq:const} is fully determined by $c =  \i \left(  \nabla_{\alpha} \lambda^{\alpha} - \nabla^{\alpha} \lambda_{\alpha} \right) + 4  \lambda_{\alpha} \lambda^{\alpha}$. Hence $\mu$ satisfies $\Re\left(\mu\right) = 0$, and the coefficients \eqref{eq:xi0} reduce to \eqref{eq:xi0b}.

		The second requirement comes from the condition $\wt{\Ric}(\wt{v},\cdot) = \wt{\Lambda} \wt{g}(\wt{v},\cdot)$ for any $\wt{v} \in \Gamma(\langle \wt{k} \rangle^\perp)$. A brute force computation, involving our previous assumptions and manipulations, eventually shows that this is equivalent to
		\begin{align}\label{eq:step6}
			\nabla_{\alpha} \lambda_0^{(4)} - 3  \i  \lambda_{\alpha}  \lambda_0^{(4)} = 0 \, .
		\end{align}
		We now turn \eqref{eq:step6} into \eqref{eq:2ndCR} by setting $\xi_{0}^{(4)} = \lambda_{0}^{(4)}$ and $\lambda_{\alpha} = \i \sigma^{-1} \accentset{\xi}{\nabla}_{\alpha} \sigma$, multiplying it through by $\sigma \ol{\sigma}^{-3}$ and adding a term involving $\nu_{\alpha} = (\sigma \ol{\sigma})^{-1} \nabla_{\alpha} (\sigma \ol{\sigma})$ to ensure CR invariance.
		
		The converse is just as was explained in the proofs of Theorems \ref{thm:wkhlfEinsc} and \ref{thm:hlfEinsc}. 
	\end{proof}

	\begin{rem}
		We note that in the case where our conformal manifold is the canonical Fefferman space, equations \eqref{eq:xi2} and \eqref{eq:exactWW_Ein} reduce to the CR--Einstein equations \eqref{eq:linWW1} and \eqref{eq:linWW2} with $\tau = \sigma$, which is consistent with what is found in \cite{Cap2008}.
	\end{rem}

	\begin{lem}\label{lem:purad_Ric}
		Let $(\wt{\mc{M}},\wt{\mbf{c}}_{\wt{\xi}},\wt{k}) \rightarrow (\mc{M},H,J)$ be a four-dimensional perturbed Fefferman space that admits an almost pure radiation scale $\wt{\sigma}$, i.e.\ $\wt{\sigma}$ satisfies \eqref{eq:alwkEins} and \eqref{eq:alcstRicSc} for some $\wt{\Phi}_{a b}$ of the form \eqref{eq:purad_alt}. Then $\wt{\Phi}_{a b}$ is given by
		\begin{align*}
			\wt{\Phi} & = 4 \cos^2 \phi \cdot \bm{\Phi} \cdot \bm{\theta} \otimes \bm{\theta} \, ,
		\end{align*}
		for some real-valued density $\bm{\Phi}$ of weight $(-2,-2)$ on $(\mc{M},H,J)$.  In particular, $\wt{\Phi}_{a b}=0$ on the zero set of $\wt{\sigma}$.
	\end{lem}
	
	\begin{proof}
		We trivialise \eqref{eq:alcstRicSc} where the pure radiation field is given by \eqref{eq:purad_alt} with the Lorentzian scale $\wt{\sigma}_{\theta}$, and since $\Upsilon = \tan \phi \, \d \phi$, we find that
		\begin{align*}
			\dot{\wt{\Phi}} + 2 \sin \phi \cdot \cos \phi \cdot \wt{\Phi} & = 0 \, ,
		\end{align*}
		which clearly has solution $\wt{\Phi} = \cos^2 \phi \cdot \Phi$, for some real-valued function $\Phi$ on $(\mc{M},H,J)$. By CR invariance, we must deduce that $\Phi$ arises from a density $\bm{\Phi}$ of weight $(-2,-2)$.
	\end{proof}
	
	\begin{rem}
		We have omitted the rather involved expression for the CR density $\bm{\Phi}$ of weight $(-2,-2)$ which determines the pure radiation field of the metric. But Lemma \ref{lem:purad_Ric} tells us that for the metric $\wt{g}$ to satisfy the full Einstein equations, $(\wt{\Ric}_{a b})_\circ = 0$, we need only one more real constraint on the CR data, namely, the vanishing of $\bm{\Phi}$.
	\end{rem}
	
	\subsection{CR realisability and Petrov types II and D}
	We end this section by reformulating a result due to \cite{Lewandowski1990a} and \cite{Hill2008} regarding the connection between the realisability of the CR structure and the Petrov type II or D condition. We provide an alternative proof based on the ideas of Section \ref{sec:CR_geom}.
	\begin{thm}\label{thm:realisableCRPetrovII}
		Let $(\wt{\mc{M}},\wt{\mbf{c}}_{\wt{\xi}},\wt{k}) \rightarrow (\mc{M},H,J)$ be a four-dimensional perturbed Fefferman space that admits an almost pure radiation or Einstein scale. Suppose $\wt{\mbf{c}}_{\wt{\xi}}$ is of Petrov type II or D but no more degenerate. Then $(\mc{M},H,J)$ is locally realisable.
	\end{thm}
	
	\begin{proof}
		Suppose that $(\wt{\mc{M}},\wt{\mbf{c}}_{\wt{\xi}},\wt{k})$ admits an almost pure radiation scale. By Theorem \ref{thm:pertFeffpurad}, there exists a nowhere vanishing density $\sigma \in \Gamma(\mc{E}(1,0))$ and CR data $\left(\bm{\xi}_{\alpha}^{(-2)}, \bm{\xi}_{\alpha}^{(0)}, [\nabla, \bm{\xi}_{0}^{(0)}, \bm{\xi}_{0}^{(2)}], \bm{\xi}_{0}^{(4)} \right)$ such that \eqref{eq:xi2}, \eqref{eq:exactWW_Ein} and \eqref{eq:2ndCR} hold, where $\nabla_{\alpha} (\sigma	\ol{\sigma}) = \nu_{\alpha} \sigma	\ol{\sigma}$ and $\xi_{\alpha}^{(-2)} = \sigma \ol{\sigma}^{-1} \bm{\xi}_{\alpha}^{(-2)}$. The condition that $\wt{\mbf{c}}_{\wt{\xi}}$ is of Petrov II or D implies that $\bm{\xi}_{0}^{(4)}$ is non-zero --- see \cite{Lewandowski1990} or Proposition \ref{prop:Weyl} below. Using the definition of $\nu_{\alpha}$, \eqref{eq:xi2} implies
		\begin{align}\label{eq:xi0m2}
			\accentset{\xi}{\nabla}_{\alpha} \sigma - \left( \nu_{\alpha} - 2\i \xi_{\alpha}^{(-2)} \right) \sigma & = 0 \, .
		\end{align}
		Hence, by \eqref{eq:xi0m2} and \eqref{eq:2ndCR}, we have that 
		\begin{align*}
			\accentset{\xi}{\nabla}_{\alpha} \left( \sigma^{-1} \bm{\xi}_0^{(4)} \right) & = 	\left( \accentset{\xi}{\nabla}_{\alpha} \sigma^{-1} \right) \bm{\xi}_0^{(4)} + 	\sigma^{-1} \left( \accentset{\xi}{\nabla}_{\alpha} \bm{\xi}_0^{(4)} \right) \\
			& = - \left( \nu_{\alpha} - 2 \i \xi_{\alpha}^{(-2)} \right) \sigma^{-1} \bm{\xi}^{(4)}_0 + \left( \nu_{\alpha} - 2 \i \xi_{\alpha}^{(-2)} \right) \sigma^{-1} \bm{\xi}^{(4)}_0 = 0 \, .
		\end{align*}
		Choose $\tau \in \Gamma(\mc{E}(1,0))$ such that $\tau^{-3} = \ol{\sigma}^{-1} \bm{\xi}_0^{(-4)} \in \Gamma(\mc{E}(-3,0))$. Then applying Theorem \ref{thm:real2dens_d3} proves realisability of $(\mc{M},H,J)$.
	\end{proof}
	
	\begin{rem}
		Suppose that $(\mc{M},H,J)$ is realisable. Then Theorem \ref{thm:real2dens_d3} tells us that there exist nowhere vanishing densities $\sigma$ and $\tau$ of weight $(1,0)$ that satisfy \eqref{eq:linWW1d3} and \eqref{eq:linWW2d3} for some gauge $\xi_{\alpha}^{(0)}$ respectively. Define $\nu_{\alpha}$ by $\nabla_{\alpha} (\sigma \ol{\sigma}) = \nu_{\alpha} \sigma \ol{\sigma}$. Then $\bm{\xi}_0^{(4)} := \sigma \ol{\tau}^{-3} \in \Gamma(\mc{E}(1,-3))$ satisfies \eqref{eq:2ndCR}. We can then define $\bm{\xi}^{(-2)}_{\alpha}$ by \eqref{eq:xi2}, and $\bm{\xi}^{(0)}_{0}$ and $\bm{\xi}^{(2)}_{0}$ by \eqref{eq:xi00b} and \eqref{eq:xi02b} respectively. A converse to Theorem \ref{thm:realisableCRPetrovII} would then boil down to whether these definitions are compatible with the remaining equation \eqref{eq:xi04b}.
	\end{rem}
	
	\section{Properties of the Weyl tensor and asymptotics}\label{sec:Weyl}
	In this final section, we investigate the properties of the conformal curvature of a perturbed Fefferman space whose CR data takes the form
	\begin{equation*}
		(\bm{\xi}^{(-2)}_{\alpha}, \bm{\xi}^{(0)}_{\alpha}, [\nabla, \bm{\xi}^{(0)}_{0}, \bm{\xi}^{(2)}_{0}], \bm{\xi}^{(4)}_{0}) \, .
	\end{equation*}
	The reason for such a choice is primarily dictated by the fact that the CR data of the optical geometry described by Theorem \ref{thm:PetrovII+Bach}, that is, with algebraically special Weyl and Bach tensors, is restricted in this fashion. More importantly, optical geometries that admit almost half-Einstein and almost pure radiation scales also fall into this category --- see Theorems \ref{thm:hlfEinsc} and \ref{thm:pertFeffpurad} --- a fact that we will eventually exploit to describe the conformal infinities of these geometries in terms of CR invariants.
	
	\subsection{Fourier expansion of the conformal curvature}
	The following proposition follows from a brute force computation using \cite[Appendix~A]{TaghaviChabert2022}.
	\begin{prop}\label{prop:Weyl}
		Let $(\wt{\mc{M}},\wt{\mbf{c}}_{\wt{\xi}},\wt{k}) \longrightarrow (\mc{M},H,J)$ be a perturbed Fefferman space, where the perturbation one-form $\wt{\xi}$ is determined by some CR data \begin{equation*}
			(\bm{\xi}^{(-2)}_{\alpha}, \bm{\xi}^{(0)}_{\alpha}, [\nabla, \bm{\xi}^{(0)}_{0}, \bm{\xi}^{(2)}_{0}], \bm{\xi}^{(4)}_{0}) \, ,
		\end{equation*}
		so that $\langle \wt{k} \rangle$ is a repeated PND of the Weyl tensor $\wt{\Weyl}$. In particular, $\wt{\Weyl}$ is algebraically special. Choose a nowhere vanishing density $\sigma$ of weight $(1,0)$ to trivialise $\wt{\mc{M}} \rightarrow \mc{M}$ with fibre coordinate $\phi$. Let $\wt{g}_{\theta,\wt{\xi}}$ be a perturbed Fefferman metric in $\wt{\mbf{c}}_{\wt{\xi}}$ for the contact form $\theta = (\sigma \ol{\sigma})^{-1} \bm{\theta}$. Then the Fourier expansions of the components of the self-dual part of the Weyl tensor defined by \eqref{eq:NP-Psi} are found to be
		\begin{subequations}\label{eq:Weyl_Fourier}
			\begin{align}
				\wt{\Psi}_2 & = \Psi_2^{(-4)} \e^{-4 \i \phi} + \Psi_2^{(-2)} \e^{-2 \i \phi} + \Psi_2^{(0)} \, , \\
				\wt{\Psi}_3 & = \Psi_3^{(-6)} \e^{-6 \i \phi} + \Psi_3^{(-4)} \e^{-4 \i \phi} + \Psi_3^{(-2)} \e^{-2 \i \phi} + \Psi_3^{(0)} + \Psi_3^{(2)} \e^{2 \i \phi}  \, , \\
				\wt{\Psi}_4 & = \Psi_4^{(-8)} \e^{-8 \i \phi} + \Psi_4^{(-6)} \e^{-6 \i \phi} + \Psi_4^{(-4)} \e^{-4 \i \phi}  + \Psi_4^{(-2)} \e^{-2 \i \phi}  + \Psi_4^{(0)} + \Psi_4^{(2)} \e^{2 \i \phi} + \Psi_4^{(4)} \e^{4 \i \phi} \, , 
			\end{align}
		\end{subequations}
		for some tensors $\Psi_2^{(2k)}$, $(\Psi_3^{(2k)})_{\alpha}$ and $(\Psi_4^{(2k)})_{\alpha \beta}$ on $(\mc{M},H,J)$ that correspond to the densities
		\begin{align*}
			& \bm{\Psi}_2^{(2k)} = \sigma^{k-1} \ol{\sigma}^{-k-1} \Psi_2^{(2k)}   & \in \Gamma(\mc{E}(k-1,-k-1)) \, , & & -2 \leq k \leq  0  \, ,\\
			& (\bm{\Psi}_3^{(2k)})_{\alpha} = \sigma^{k-1} \ol{\sigma}^{-k-1}  (\Psi_3^{(2k)})_{\alpha} & \in \Gamma(\mc{E}_{\alpha}(k-1,-k-1)) \, , & & -3 \leq k \leq  1  \, ,\\
			& (\bm{\Psi}_4^{(2k)})_{\alpha \beta} =   \sigma^{k-1} \ol{\sigma}^{-k-1} (\Psi_4^{(2k)})_{\alpha \beta} & \in \Gamma(\mc{E}_{\alpha \beta}(k-1,-k-1)) \, , & & -4 \leq k \leq  2  \, ,
		\end{align*}
		given by
		\begingroup
		\allowdisplaybreaks
		\begin{align*}
			\bm{\Psi}_2^{(-4)} & =  4 \bm{\xi}_0^{(-4)} \, , \\
			\bm{\Psi}_2^{(-2)} & = \i \accentset{\xi}{\nabla}^{\gamma} \bm{\xi}^{(-2)}_{\gamma} + 2  \bm{\xi}_0^{(-2)}\, , \\
			\bm{\Psi}_2^{(0)} & = - \tfrac{4 \i}{3}  F_{\alpha}{}^{\alpha} - 2 \bm{\xi}^{\gamma}_{(2)} \bm{\xi}_{\gamma}^{(-2)}  \, , \\
			\bm{\Psi}_3^{(-6)}
			& = - \tfrac{3}{2} \bm{\xi}_{\alpha}^{(-2)} \bm{\Psi}_2^{(-4)} \, , \\
			\bm{\Psi}_3^{(-4)}
			& = - \tfrac{1}{2} \bm{\xi}_{\alpha}^{(-2)} \bm{\Psi}_2^{(-2)} + \tfrac{\i}{2}  \accentset{\xi}{\nabla}_{\alpha} \bm{\Psi}_2^{(-4)}  \, , \\
			\bm{\Psi}_3^{(-2)}
			& = \tfrac{\i}{4} \accentset{\xi}{\nabla}_{\alpha}\bm{\Psi}_{2}^{(-2)} - \i \left(  \accentset{\xi}{\nabla}_0\bm{\xi}_{\alpha}^{(-2)} - \tfrac{2 \i}{3} \Rho \bm{\xi}_{\alpha}^{(-2)} - \accentset{\xi}{\nabla}_{\alpha}\bm{\xi}_{0}^{(-2)} \right)  \, , \\
			\bm{\Psi}_3^{(0)}
			& = \tfrac{1}{2}\nabla_{\alpha} F_{\beta}{}^{\beta}
			- \tfrac{3 \i}{2}  F_{0 \alpha} 
			+\tfrac{3}{2} \bm{\xi}_{\alpha}^{(-2)} \ol{\bm{\Psi}}_2^{(2)} - \i  \bm{\xi}^{\beta}_{(2)} \accentset{\xi}{\nabla}_{\beta}\bm{\xi}_{\alpha}^{(-2)}   \, , \\
			\bm{\Psi}_3^{(2)}
			& = \tfrac{\i}{4} \accentset{\xi}{\nabla}_{\alpha} \ol{\bm{\Psi}}_2^{(2)}
			+ \tfrac{1}{2} \bm{\xi}_{\alpha}^{(-2)} \ol{\bm{\Psi}}_{2}^{(4)}  \, , \\
			\bm{\Psi}_4^{(-8)}
			& = \tfrac{3}{2} \bm{\Psi}_2^{(-4)} \bm{\xi}_{\alpha}^{(-2)} \bm{\xi}_{\beta}^{(-2)} \, , \\
			\bm{\Psi}_4^{(-6)}
			& = -\tfrac{\i}{4} \bm{\Psi}_{2}^{(-4)}  \accentset{\xi}{\nabla}_{\alpha}\bm{\xi}_{\beta}^{(-2)} - \tfrac{3 \i}{4} \bm{\xi}_{\alpha}^{(-2)} \accentset{\xi}{\nabla}_{\beta} \bm{\Psi}_{2}^{(-4)} \, , \\
			\bm{\Psi}_4^{(-4)}
			& = -\tfrac{1}{8} \left( \accentset{\xi}{\nabla}_{\alpha} \accentset{\xi}{\nabla}_{\beta} \bm{\Psi}_{2}^{(-4)}
			-3 \i \WbA_{\alpha \beta} \bm{\Psi}_{2}^{(-4)} \right)
			- \bm{\xi}_{\alpha}^{(-2)} \left( (\bm{\Psi}_3^{(-2)})_{\beta}
			- \tfrac{\i}{4} \accentset{\xi}{\nabla}_{\beta}\bm{\Psi}_{2}^{(-2)} \right) \, , \\
			\bm{\Psi}_4^{(-2)}
			& =  
			\tfrac{1}{8} \left(\accentset{\xi}{\nabla}_{\alpha}  \accentset{\xi}{\nabla}_{\beta}\bm{\Psi}_{2}^{(-2)} + 4 \i \WbA_{\alpha \beta} \bm{\Psi}_{2}^{(-2)}\right) 
			+\tfrac{\i}{2}\accentset{\xi}{\nabla}_{\alpha} (\bm{\Psi}_3^{(-2)})_{\beta}
			\, , \\
			\bm{\Psi}_4^{(0)}
			& = \tfrac{\i}{4} \CQ_{\alpha \beta} + \nabla_{\alpha} F_{0\beta} + \WbA_{\alpha \beta} F_{\gamma}{}^{\gamma} 
			+2 \i \bm{\xi}_{0}^{(2)} \accentset{\xi}{\nabla}_{\alpha}\bm{\xi}_{\beta}^{(-2)} \\
			& \qquad \qquad \qquad \qquad \qquad \qquad
			+3 \i \bm{\xi}_{\alpha}^{(-2)} \left( \accentset{\xi}{\nabla}_{\beta}\bm{\xi}_{0}^{(2)}
			- \WbA_{\beta \gamma} \bm{\xi}^{\gamma}_{(2)} - \tfrac{\i}{2} \bm{\xi}_{\beta}^{(-2)} \ol{\bm{\Psi}}_{2}^{(4)} \right) 
			\, , \\
			\bm{\Psi}_4^{(2)}
			& = -\tfrac{1}{2}\accentset{\xi}{\nabla}_{\alpha} \accentset{\xi}{\nabla}_{\beta} \bm{\xi}_{0}^{(2)}
			+ \tfrac{3 \i}{4} \ol{\bm{\Psi}}_{2}^{(4)} \accentset{\xi}{\nabla}_{\alpha}\bm{\xi}_{\beta}^{(-2)} 
			+ \tfrac{5 \i}{4} \bm{\xi}_{\alpha}^{(-2)} \accentset{\xi}{\nabla}_{\beta}\ol{\bm{\Psi}}_{2}^{(4)} 
			 +\tfrac{1}{2} \bm{\xi}^{\gamma}_{(2)} \nabla_{\gamma} \WbA_{\alpha \beta} 
			+  \WbA_{\alpha \beta} \accentset{\xi}{\nabla}_{\gamma}\bm{\xi}^{\gamma}_{(2)} \, , \\
			\bm{\Psi}_4^{(4)}
			& = -\tfrac{1}{8} \left( \accentset{\xi}{\nabla}_{\alpha} \accentset{\xi}{\nabla}_{\beta}\ol{\bm{\Psi}}_{2}^{(4)}
			+ \i \WbA_{\alpha \beta} \ol{\bm{\Psi}}_{2}^{(4)} \right) \, .
		\end{align*}
		\endgroup
		Here, for $i=2,3,4$, we have $\ol{\bm{\Psi}}_i^{(2k)}= \ol{\bm{\Psi}_i^{(-2k)}}$ for any $k$, $\accentset{\xi}{\nabla}$ denotes the gauged Webster connection defined by \eqref{eq:gconn} with $\xi := \xi^{(0)}_{0} \theta + \xi^{(0)}_{\alpha} \theta^{\alpha} + \xi^{(0)}_{\bar{\alpha}} \ol{\theta}{}^{\bar{\alpha}}$, and we have defined
		\begin{align}\label{eq:F}
			\begin{aligned}
				F_{\alpha \bar{\beta}} & := \tfrac{1}{2} \left( \nabla_{\alpha} \bm{\xi}_{\bar{\beta}}^{(0)} - \nabla_{\bar{\beta}} \bm{\xi}_{\alpha}^{(0)} - \i \bm{h}_{\alpha \bar{\beta}} \bm{\xi}_{0}^{(0)} \right) \, , \\
				F_{0 \alpha} & := \tfrac{1}{2} \left( \nabla_{0} \bm{\xi}_{\alpha}^{(0)} - \nabla_{\alpha} \bm{\xi}_{0}^{(0)} +   \WbA_{\alpha \beta} \bm{\xi}_{(0)}^{\beta} \right) \, .
			\end{aligned}
		\end{align}
	\end{prop}
	
	\begin{proof}
		We only sketch the rather extensive computations. As usual, we work with a perturbed Fefferman metric and an adapted coframe as explained in Section \ref{sec-optical}. We first express the Weyl curvature, computed from \cite[Appendix~A]{TaghaviChabert2022}, in terms of the Fourier coefficients $\lambda_{\alpha}^{(2k)}$, $\lambda_{\bar{\alpha}}^{(2k)}$ and $\lambda_{0}^{(2k)}$ of the coframe one-form $\lambda$. Then, using the relations \eqref{eq:lmb2xi}, these are eliminated in favour of the coefficients $\xi_{\alpha}^{(2k)}$, $\xi_{\bar{\alpha}}^{(2k)}$ and $\xi_{0}^{(2k)}$ of the perturbation one-form $\xi$, and eventually in terms of their weighted analogues $\bm{\xi}_{\alpha}^{(2k)}$, $\bm{\xi}_{\bar{\alpha}}^{(2k)}$ and $\bm{\xi}_{0}^{(2k)}$ using the definitions \eqref{eq:Fourier_coef}.
		
		It is in fact easier to reverse the steps by starting from the densities to arrive at the unweighted components of the Weyl tensor. To illustrate the procedure, let us consider the densities $\bm{\Psi}_{3}^{(-2)}$ and $\bm{\Psi}_{4}^{(0)}$ for instance. For clarity, we rewrite
		\begin{align*}
			\bm{\Psi}_3^{(-2)}
			& = -\tfrac{1}{4} \accentset{\xi}{\nabla}_{\beta}\accentset{\xi}{\nabla}^{\beta}\bm{\xi}_{\alpha}^{(-2)} + \tfrac{3 \,\i }{2} \accentset{\xi}{\nabla}_{\alpha}\bm{\xi}_{0}^{(-2)} -\i \left( \accentset{\xi}{\nabla}_0\bm{\xi}_{\alpha}^{(-2)} - \tfrac{2\i}{3} \Rho \bm{\xi}_{\alpha}^{(-2)} \right) \, .
		\end{align*}
		First, we choose a nowhere vanishing density $\sigma$ of weight $(1,0)$, to trivialise them as $\Psi_{3}^{(-2)}$ and $\Psi_{4}^{(0)}$ respectively, i.e.\ $\bm{\Psi}_3^{(-2)} = \sigma^{-2} \Psi_3^{(-2)}$ and  $\bm{\Psi}_4^{(0)} = \sigma^{-1} \ol{\sigma}^{-1}  \Psi_4^{(0)}$. Next, working with a Webster connection $\nabla$ compatible with $\sigma \ol{\sigma}$, we note that for any tensor-valued density $\bm{f}$ of weight $(w,w')$ with trivialisation $f = \sigma^{-w} \ol{\sigma}^{-w'} \bm{f}$, we have
		\begin{align*}
			\accentset{\xi}{\nabla}_{\alpha} \bm{f} & = \sigma^{w} \ol{\sigma}^{w'} \left( \nabla_{\alpha} f - \i (w - w)'\lambda_{\alpha}^{(0)} f \right) \, , \\
			\accentset{\xi}{\nabla}_{0} \bm{f} & = \sigma^{w} \ol{\sigma}^{w'} \left( \nabla_{0} f - \i (w - w)' \left( \lambda_{0}^{(0)} + \tfrac{1}{3} \Rho \right) f \right) \, ,
		\end{align*}
		where we have used \eqref{eq:gD_density}. Similarly, we find
		\begin{multline*}
			\accentset{\xi}{\nabla}_{\alpha} \accentset{\xi}{\nabla}_{\bar{\beta}} \bm{f} = \sigma^{w} \ol{\sigma}^{w'} \left( \nabla_{\alpha} \nabla_{\bar{\beta}} f 
			 - \i (w - w)'\left( \nabla_{\alpha} \lambda_{\bar{\beta}}^{(0)} f + \lambda_{\alpha}^{(0)} \nabla_{\bar{\beta}} f + \lambda_{\bar{\beta}}^{(0)} \nabla_{\alpha} f \right) \right. \\ 
			 \left. - (w - w')^2 \lambda_{\alpha}^{(0)} \lambda_{\bar{\beta}}^{(0)} f \right) \, .
		\end{multline*}
		In addition, we note that
		\begin{align*}
			F_{\alpha}{}^{\alpha} & = \tfrac{1}{2} \left( \nabla_{\alpha} \lambda^{\alpha}_{(0)} - \nabla^{\alpha} \lambda_{\alpha}^{(0)} + \i \left( \lambda_{0}^{(0)} - \Rho \right) \right) \, , \\
			F_{0 \alpha} & = \tfrac{1}{2} \left( \nabla_{0} \lambda_{\alpha}^{(0)} - \nabla_{\alpha} \lambda_{0}^{(0)} +   \WbA_{\alpha \beta} \lambda_{(0)}^{\beta} - \CT_{\alpha} \right) \, ,
		\end{align*}
		where $\CT_{\alpha}$ is defined in equation \eqref{eq:Cartan_tensor}. It is then a simple matter of using these identities to show that the densities $\bm{\Psi}_{3}^{(-2)}$ and $\bm{\Psi}_{4}^{(0)}$ given in the proposition reduce to
		\begin{multline*}
			\Psi_{3}^{(-2)} =
			-\tfrac{1}{4}\nabla_{\beta}{\nabla^{\beta}{\lambda^{(-2)}_{\alpha}}}
			-\tfrac{\i}{2}  \lambda^{(-2)}_{\alpha} \nabla_{\beta}{\lambda_{(0)}^{\beta}}
			-\tfrac{\i}{2}\lambda^{(0)}_{\beta} \nabla^{\beta}{\lambda^{(-2)}_{\alpha}}
			-\tfrac{\i}{2}\lambda_{(0)}^{\beta} \nabla_{\beta}{\lambda^{(-2)}_{\alpha}} 
			\\ + \lambda^{(0)}_{\alpha} \lambda_{(0)}^{\beta} \lambda^{(-2)}_{\beta} 
			+ \tfrac{3 \,\i}{2} \nabla_{\alpha}{\lambda^{(-2)}_{0}}
			-3 \lambda^{(0)}_{\alpha} \lambda^{(-2)}_{0}
			-\i \nabla_{0}{\lambda^{(-2)}_{\alpha}}
			+2 \lambda^{(-2)}_{\alpha} \lambda^{(0)}_{0}
			\, ,
		\end{multline*}
		and
		\begin{multline*}
			\Psi_{4}^{(0)} = 
			- \tfrac{1}{4}  \nabla_{0} \WbA_{\alpha \beta}
			+ \tfrac{1}{2}\nabla_{\alpha}{\nabla_{0}{\lambda^{(0)}_{\beta}}}
			- \tfrac{1}{2}\nabla_{\alpha}{\nabla_{\beta}{\lambda^{(0)}_{0}}}
			+ \tfrac{1}{2}\lambda_{(0)}^{\gamma} \nabla_{\gamma}{\WbA_{\alpha \beta}} 
			\\
			+ \WbA_{\alpha \beta} \nabla_{\gamma}{\lambda_{(0)}^{\gamma}} 
			- \tfrac{1}{2}\WbA_{\alpha \beta} \nabla^{\gamma}{\lambda^{(0)}_{\gamma}}
			+ \tfrac{\i}{2} \WbA_{\alpha \beta} \lambda^{(0)}_{0} 
			+ 2 \,\i \lambda^{(2)}_{0} \nabla_{\alpha}{\lambda^{(-2)}_{\beta}} 
			+ 2 \lambda^{(0)}_{\alpha} \lambda^{(2)}_{0} \lambda^{(-2)}_{\beta} \\
			+ 3 \,\i \lambda^{(-2)}_{\beta} \nabla_{\alpha}{\lambda^{(2)}_{0}}  
			- 3 \,\i \WbA_{\alpha \beta} \lambda_{(2)}^{\gamma} \lambda^{(-2)}_{\gamma} 
			+ 6 \lambda^{(4)}_{0} \,\lambda^{(-2)}_{\alpha} \lambda^{(-2)}_{\beta} 
			\, ,
		\end{multline*}
		respectively. Omitting the details, these expressions can be shown to be precisely those found on the basis of \cite[Appendix~A]{TaghaviChabert2022}. The other cases work in the same way. 
	\end{proof}
	
	\begin{rem}
		The tensorial quantities \eqref{eq:F} should be understood as the components of the exterior derivative of the perturbation one-form $\wt{\xi}$ --- see expression \eqref{eq:curv_xi0}.
	\end{rem}
	
	\begin{rem}
		In the case of a Fefferman space, i.e.\ $\wt{\xi} =0$, or even $\wt{\xi} = \xi$ with $\d \xi = 0$, there is only one component left, namely, $\bm{\Psi}_4^{(0)} = \tfrac{\i}{4} \CQ_{\alpha \beta}$ where $\CQ_{\alpha \beta}$ is the Cartan invariant \eqref{eq:Cartan_tensor} obstructing CR flatness. Thus, we have that the Weyl tensor is of type N, and the Fefferman conformal structure is flat if and only if the CR structure is flat, as was already demonstrated in \cite{Lewandowski1988}.
	\end{rem}
	
	From \eqref{eq:Psichange}, under a change of perturbed Fefferman metric (or equivalently, change of contact form), these transform as
	\begin{subequations}\label{eq_Psidenschange}
		\begin{align}
			\wh{\bm{\Psi}}_2^{(2k)} & = \bm{\Psi}_2^{(2k)} \, ,   & - 2 \leq k \leq 0 \, ,\\
			(\wh{\bm{\Psi}}_3^{(2k)})_{\alpha} & = (\bm{\Psi}_3^{(2k)})_{\alpha} - \tfrac{3 \i}{2} \Upsilon_{\alpha} \bm{\Psi}_2^{(2k)} \, , & - 3 \leq k \leq 1 \, ,\\
			(\wh{\bm{\Psi}}_4^{(2k)})_{\alpha \beta} & = (\bm{\Psi}_4^{(2k)})_{\alpha \beta} - 2 \i \Upsilon_{\alpha} (\bm{\Psi}_3^{(2k)})_{\beta} - \tfrac{3}{2} \Upsilon_{\alpha} \Upsilon_{\beta} \bm{\Psi}_2^{(2k)} \, , & - 4 \leq k \leq 2  \, .
		\end{align}
	\end{subequations}
	These can be used to prove:
	\begin{lem}\label{lem:PetrovD}
		Let $(\wt{\mc{M}},\wt{\mbf{c}}_{\wt{\xi}}, \wt{k} ) \longrightarrow (\mc{M},H,J)$ be a perturbed Fefferman space as in Proposition \ref{prop:Weyl}. Then \begin{align}\label{eq:CRPetrovD}
			\bm{\Psi}_{\alpha \beta}^{(2k)} & := \sum_{i+j =k} \left( 4 (\bm{\Psi}_3^{(2i)} )_{\alpha} (\bm{\Psi}_3^{(2j)})_{\beta} - 6 \bm{\Psi}_2^{(2i)} (\bm{\Psi}_4^{(2j)})_{\alpha \beta} \right) \, , & - 6 \leq k \leq 2 \, ,
		\end{align}
		are CR invariant sections of $\mc{E}_{\alpha \beta}(2k-2,-2k-2)$. Further, $\bm{\Psi}^{(-12)}_{\alpha \beta}=0$.
		
		In particular, the Weyl tensor is of Petrov type D if and only if $\bm{\Psi}_{\alpha \beta}^{(2k)}=0$ for all $k$.
	\end{lem}
	
	\begin{rem}
		The transformations \eqref{eq_Psidenschange} allow us to determine which of these are CR invariant under various conformal curvature conditions such as Petrov types and so on. For instance, the density $\bm{\Psi}_4^{(4)}$ is a CR invariant for Petrov type II or D, and is nothing but the gauged Webster--Weyl operator acting on a density of weight $(1,-3)$, namely $\ol{\bm{\Psi}}_{2}^{(4)} = 4 \bm{\xi}_{0}^{(4)}$, which is known to be CR invariant in the first place.
	\end{rem}
	
	Lemma \ref{lem:PetrovD} provides relations between the densities $\bm{\Psi}_{i}^{(2k)}$ under the Petrov type D assumption. In the following proposition, which can be easily proved by inspection of the densities given in Proposition \ref{prop:Weyl}, we extend our findings to the other Petrov types --- see also Proposition \ref{prop:PetrovIII} for Petrov type III:
	\begin{prop}\label{prop:PetrovIII-N-O}
		Let $(\wt{\mc{M}},\wt{\mbf{c}}_{\wt{\xi}}, \wt{k} ) \longrightarrow (\mc{M},H,J)$ be a perturbed Fefferman space as in Proposition \ref{prop:Weyl}.
		\begin{itemize}
			\item  The Weyl tensor is of Petrov type III if and only if
			\begin{align}\label{eq:xi_PetIII}
				\bm{\xi}_{0}^{(4)} & = 0 \, , & F_{\alpha}{}^{\alpha} & = \tfrac{3\i}{2} \bm{\xi}^{\alpha}_{(2)} \bm{\xi}_{\alpha}^{(-2)} \, , & \bm{\xi}_0^{(-2)} & = -\tfrac{\i}{2} \accentset{\xi}{\nabla}^{\gamma} \bm{\xi}^{(-2)}_{\gamma} \, ;
			\end{align}
			In addition, $\bm{\Psi}_{3}^{(-6)} = \bm{\Psi}_{3}^{(-4)} = \bm{\Psi}_{3}^{(2)} = 0$ and $\bm{\Psi}_{4}^{(-8)} = \bm{\Psi}_{4}^{(-6)} = \bm{\Psi}_{4}^{(4)} = 0$.
			\item  The Weyl tensor is of Petrov type N if and only if conditions \eqref{eq:xi_PetIII} hold and
			\begin{subequations}\label{eq:xi_PetN}
				\begin{align}
					& F_{0 \alpha} = - \tfrac{1}{6} \bm{\xi}^{\beta}_{(2)} \accentset{\xi}{\nabla}_{\alpha}\bm{\xi}_{\beta}^{(-2)} + \tfrac{1}{2} \bm{\xi}^{(-2)}_{\beta} \accentset{\xi}{\nabla}_{\alpha} \bm{\xi}^{\beta}_{(2)} \, , \\
					& \accentset{\xi}{\nabla}_{\alpha} \accentset{\xi}{\nabla}^{\gamma} \bm{\xi}^{(-2)}_{\gamma} - 2\i \accentset{\xi}{\nabla}_0\bm{\xi}_{\alpha}^{(-2)} - \tfrac{4}{3} \Rho \bm{\xi}_{\alpha}^{(-2)}   = 0 \, ;
				\end{align}
				In addition, $\bm{\Psi}_{4}^{(-8)} = \bm{\Psi}_{4}^{(-6)} = \bm{\Psi}_{4}^{(-4)} = \bm{\Psi}_{4}^{(-2)} = \bm{\Psi}_{4}^{(4)} = 0$.
			\end{subequations}
			\item  The Weyl tensor vanishes identically if and only if conditions \eqref{eq:xi_PetIII} and \eqref{eq:xi_PetN} hold, and
			\begin{subequations}\label{eq:xi_PetO}
				\begin{align}
					& \bm{\xi}^{(-2)}_{\gamma} \accentset{\xi}{\nabla}_{\alpha} \accentset{\xi}{\nabla}_{\beta} \bm{\xi}^{\gamma}_{(2)} + \tfrac{1}{6} \bm{\xi}^{\gamma}_{(2)} \accentset{\xi}{\nabla}_{\alpha} \accentset{\xi}{\nabla}_{\beta} \bm{\xi}^{(-2)}_{\gamma} + \tfrac{2}{3} (\accentset{\xi}{\nabla}_{\alpha} \bm{\xi}_{\gamma}^{(-2)}) (\accentset{\xi}{\nabla}_{\beta} \bm{\xi}^{\gamma}_{(2)}) \\
					& \qquad \qquad \qquad \qquad \qquad \qquad \qquad + \tfrac{3\i}{2} \WbA_{\alpha \beta} \bm{\xi}^{\gamma}_{(2)} \bm{\xi}_{\gamma}^{(-2)} - \tfrac{\i}{4} \CQ_{\alpha \beta}  = 0 \, , \nonumber \\
					&  \accentset{\xi}{\nabla}_{\alpha} \accentset{\xi}{\nabla}_{\beta}  \accentset{\xi}{\nabla}_{\gamma} \bm{\xi}_{(2)}^{\gamma}
					+ 4 \i \WbA_{\alpha \beta} \accentset{\xi}{\nabla}_{\gamma}\bm{\xi}^{\gamma}_{(2)} + 2 \i \bm{\xi}^{\gamma}_{(2)} \nabla_{\gamma} \WbA_{\alpha \beta} 
					= 0 \, .
				\end{align}
			\end{subequations}
		\end{itemize}
	\end{prop}
	
	\begin{rem}
		The CR invariance of equations \eqref{eq:xi_PetIII}, \eqref{eq:xi_PetN} and \eqref{eq:xi_PetO} can be readily verified.
	\end{rem}
	
	Another property easy to check is that $\bm{\Psi}_2^{(0)}$ is real-valued. The interpretation of its vanishing is given next.
	\begin{prop}\label{prop:BachPsi}
		Let $(\wt{\mc{M}},\wt{\mbf{c}}_{\wt{\xi}}, \wt{k} ) \longrightarrow (\mc{M},H,J)$ be a perturbed Fefferman space as in Proposition \ref{prop:Weyl}. Then the Bach tensor satisfies $\wt{\Bach}(\wt{k},\wt{k}) = 0$ if and only if $\bm{\Psi}_2^{(0)} = 0$.
	\end{prop}
	
	\begin{proof}
		This is immediate on comparing the expression for $\bm{\Psi}_2^{(0)}$ in Proposition \ref{prop:Weyl} with equation \eqref{eq:xi00PetrovII+Bach} in Theorem \ref{thm:PetrovII+Bach}.
	\end{proof}
	
	\subsection{Distinguished almost Lorentzian scales}
	If we now assume that the perturbed Fefferman space admits a distinguished almost Lorentzian scale as discussed in Section \ref{sec:al_Lor_sc}, the expressions for the densities of the Weyl curvature given in Proposition \ref{prop:Weyl} may simplify or vanish altogether. We shall make no attempt to provide a full account, and focus on the two specific cases below.
	
	\begin{prop}\label{prop:hEinasym}
		Let $(\wt{\mc{M}},\wt{\mbf{c}}_{\wt{\xi}}, \wt{k} ) \longrightarrow (\mc{M},H,J)$ be a perturbed Fefferman space as in Proposition \ref{prop:Weyl}, and suppose it admits an almost half-Einstein scale. Let $\sigma \in \Gamma(\mc{E}(1,0))$ be the nowhere vanishing density given in Theorem \ref{thm:hlfEinsc}. Then
		\begin{align}
			\sigma \ol{\sigma} \ol{\bm{\Psi}}_{2}^{(0)} - \ol{\sigma}^2 \ol{\bm{\Psi}}_{2}^{(2)} + \sigma^{-1} \ol{\sigma}^3  \ol{\bm{\Psi}}_{2}^{(4)} & = 0 \, . \label{eq:Psi2_hEins}
		\end{align}
	\end{prop}
	
	\begin{proof}
		Combining \eqref{eq:xi00} and \eqref{eq:xi02}, we obtain
		\begin{equation*}
			\sigma \ol{\sigma} \left( \bm{\xi}_{0}^{(0)} - \i \left(  \nabla_{\alpha} \bm{\xi}^{\alpha}_{(0)} -  \nabla^{\alpha} \bm{\xi}_{\alpha}^{(0)} \right) -  3 \bm{\xi}_{\alpha}^{(-2)} \bm{\xi}^{\alpha}_{(2)} \right)
			- 3 \ol{\sigma}^2 \left( \bm{\xi}_{0}^{(2)} - \tfrac{1}{2}  \i  \accentset{\xi}{\nabla}_{\alpha} \bm{\xi}^{\alpha}_{(2)} \right) + 6 \sigma^{-1} \ol{\sigma}^3 \bm{\xi}_{0}^{(4)}   = 0 \, . 
		\end{equation*}
		Hence, on comparing this equation with the expressions for $\bm{\Psi}_{2}^{(0)}$, $\bm{\Psi}_{2}^{(-2)}$ and $\bm{\Psi}_{2}^{(-4)}$ in Proposition \ref{prop:Weyl}, we deduce \eqref{eq:Psi2_hEins}.
	\end{proof}

	\begin{prop}\label{prop:pureradasym}
		Let $(\wt{\mc{M}},\wt{\mbf{c}}_{\wt{\xi}}, \wt{k} ) \longrightarrow (\mc{M},H,J)$ be a perturbed Fefferman space as in Proposition \ref{prop:Weyl}, and suppose it admits an almost pure radiation scale. Let $\sigma \in \Gamma(\mc{E}(1,0))$ be the nowhere vanishing density given in Theorem \ref{thm:pertFeffpurad}. Then
		\begin{subequations}\label{eq:Psi2_purad}
			\begin{align}
				\ol{\bm{\Psi}}_{2}^{(0)} & = 0 \, , \label{eq:Psi2_purad1} \\
				\ol{\bm{\Psi}}_{2}^{(2)} & = \sigma^{-1} \ol{\sigma}  \ol{\bm{\Psi}}_{2}^{(4)}  \, . \label{eq:Psi2_purad2}
			\end{align}
		\end{subequations}
	\end{prop}
	
	\begin{proof}
		On inspection of $\bm{\Psi}_{2}^{(0)}$, $\bm{\Psi}_{2}^{(-2)}$ and $\bm{\Psi}_{2}^{(-4)}$ in Proposition \ref{prop:Weyl}, conditions \eqref{eq:Psi2_purad1} and \eqref{eq:Psi2_purad2} are seen to follow from \eqref{eq:xi00b} and \eqref{eq:xi02b} respectively. Equation \eqref{eq:Psi2_purad1} can also be obtained from Proposition \ref{prop:BachPsi} and Proposition \ref{prop:W+E2B}.
	\end{proof}
	
	\subsection{Asymptotics}\label{sec:asym}
	We now investigate the behaviour of the Weyl tensor on the zero set of some distinguished almost Lorentzian scale.
	\begin{prop}\label{prop:Wasym}
		Let $(\wt{\mc{M}},\wt{\mbf{c}}_{\wt{\xi}}, \wt{k} ) \longrightarrow (\mc{M},H,J)$ be a perturbed Fefferman space as in Proposition \ref{prop:Weyl}. Choose any nowhere vanishing density  $\sigma$ of weight $(1,0)$ and let $\wt{\mc{Z}}$ be the union $\wt{\mc{Z}}_+ \cup \wt{\mc{Z}}_-$ of the sections $\wt{\mc{Z}}_{\pm} := [\pm \i \sigma^{-1}] : \mc{M} \rightarrow \wt{\mc{M}}$. Then the Weyl tensor is
		\begin{enumerate}
			\item of Petrov type D or more degenerate on $\wt{\mc{Z}}$ if and only if 
			\begin{align}\label{eq:PetrovD_asymp}
				&	\sum_{k=-5}^{2} (-1)^{k} \sigma^{-k+1} \ol{\sigma}^{k+1} \bm{\Psi}_{\alpha \beta}^{(2k)} = 0\, ;
			\end{align}
			\item of Petrov type III or more degenerate on $\wt{\mc{Z}}$ if and only if 
			\begin{align}\label{eq:PetrovIII_asymp}
				& \sigma^3 \ol{\sigma}^{-1} \bm{\Psi}_{2}^{(-4)} - \sigma^2 \bm{\Psi}_{2}^{(-2)} +	\sigma \ol{\sigma} \bm{\Psi}_{2}^{(0)}  = 0 \, ;
			\end{align}
			\item of Petrov type N or more degenerate on $\wt{\mc{Z}}$ if and only if \eqref{eq:PetrovIII_asymp} holds, and
			\begin{align}\label{eq:PetrovN_asymp}
				& - \sigma^{4} \ol{\sigma}^{-2} \bm{\Psi}_3^{(-6)}  + \sigma^{3} \ol{\sigma}^{-1} \bm{\Psi}_3^{(-4)} - \sigma^{2}  \bm{\Psi}_3^{(-2)} + \sigma \ol{\sigma} \bm{\Psi}_3^{(0)} -  \ol{\sigma}^{2} \bm{\Psi}_3^{(2)} = 0 \, ;
			\end{align}
			\item of Petrov type O on $\wt{\mc{Z}}$ if and only if \eqref{eq:PetrovIII_asymp} and  \eqref{eq:PetrovN_asymp} holds, and
			\begin{equation}\label{eq:PetrovO_asymp}
				\sigma^{5} \ol{\sigma}^{-3} \bm{\Psi}_4^{(-8)} - \sigma^{4} \ol{\sigma}^{-2} \bm{\Psi}_4^{(-6)} + \sigma^{3} \ol{\sigma}^{-1} \bm{\Psi}_4^{(-4)} - \sigma^{2} \bm{\Psi}_4^{(-2)} \\
				+ \sigma \ol{\sigma} \bm{\Psi}_4^{(0)} - \ol{\sigma}^{2} \bm{\Psi}_4^{(2)} + \sigma \ol{\sigma}^{3} \bm{\Psi}_4^{(4)} = 0 \, .
			\end{equation}
		\end{enumerate}
	\end{prop}  
	
	\begin{proof}
		In terms of the fibre coordinate, $\wt{\mc{Z}}$ consists of the hypersurfaces $\phi = \pm \tfrac{\pi}{2}$. The expressions \eqref{eq:Weyl_Fourier} evaluated at $\phi = \pm \tfrac{\pi}{2}$ simply correspond to \eqref{eq:PetrovIII_asymp}, \eqref{eq:PetrovN_asymp} and \eqref{eq:PetrovO_asymp}. The computation of  \eqref{eq:PetrovD_asymp} is similar, and follows from formula \eqref{eq:CRPetrovD}.
	\end{proof}
	
	Our next step is to identify $\wt{\mc{Z}}$ as given in Proposition \ref{prop:Wasym} as the zero set of some almost Lorentzian scale $\wt{\sigma}$ --- this is possible by \cite[Proposition~6.1]{TaghaviChabert2023b}, see also Section \ref{sec:prel_zero_set}. Assuming that $\wt{\sigma}$ is in fact an almost half-Einstein scale, and comparing equation \eqref{eq:Psi2_hEins} from Proposition \ref{prop:hEinasym} and equation \eqref{eq:PetrovIII_asymp} from Proposition \ref{prop:Wasym}, we obtain:
	\begin{cor}
		Let $(\wt{\mc{M}},\wt{\mbf{c}}_{\wt{\xi}}, \wt{k} ) \longrightarrow (\mc{M},H,J)$ be a perturbed Fefferman space as in Proposition \ref{prop:Weyl}, and suppose it admits an almost half-Einstein scale with zero set $\wt{\mc{Z}}$. Then the Weyl tensor is of Petrov type III or more degenerate on $\wt{\mc{Z}}$.
	\end{cor}
	
	The causal properties of the zero set of an almost half-Einstein scale are given in \cite[Proposition~7.1]{TaghaviChabert2023b}. This allows us to prove the following:
	\begin{thm}\label{thm:asymconfflat}
		Suppose $(\wt{\mc{M}},\wt{\mbf{c}}_{\wt{\xi}},\wt{k})$ admits an almost pure radiation scale $\wt{\sigma}$ with zero set $\wt{\mc{Z}}$. Then $\wt{\mbf{c}}_{\wt{\xi}}$ is conformally flat on $\wt{\mc{Z}}$, i.e.\ $\wt{\Weyl}=0$ on $\wt{\mc{Z}}$.
	\end{thm}
	
	\begin{proof}
		Let us take another covariant derivative of the defining equation \eqref{eq:alwkEins} for $\wt{\sigma}$ with \eqref{eq:purad_alt} and \eqref{eq:alcstRicSc}. Then commuting the covariant derivatives, we derive the integrability condition
		\begin{align*}
			\wt{\Weyl}_{a b c d} \wt{\nu}^d + \wt{\Cot}_{c a b} \wt{\sigma} & = \wt{T}_{a b c} \wt{\sigma} - \wt{\Phi} \wt{\kappa}_{c} \wt{\kappa}_{[a} \wt{\nu}_{b]} \, ,
		\end{align*}
		for some trace-free tensor $\wt{T}_{a b c}$ regular on $\wt{\mc{M}}$, and where $\wt{\nu}_{a} = \wt{\nabla}_{a} \wt{\sigma}$. By Lemma \ref{lem:purad_Ric}, we already know that $\wt{\Phi} = 0$ on $\wt{\mc{Z}}$. Hence
		\begin{align}\label{eq:WeylonZ}
			\wt{\Weyl}_{a b c d} \wt{\nu}^d & = 0 \, , & \mbox{on $\wt{\mc{Z}}$.}
		\end{align}
		There are two cases to consider:
		\begin{itemize}
			\item  If $\wt{\Lambda} \neq 0$, then by \cite[Proposition~7.1]{TaghaviChabert2023b}, $\wt{\mc{Z}}$ is spacelike or timelike, and so $\wt{\nu}_a$ is non-null. It is easy to verify that in dimension four, at any point, given such a non-null vector, the Weyl tensor generically takes the form $\wt{\Weyl}_{a b c d} = \left(\wt{\nu}_{[a} \wt{\Weyl}_{b] c d} + \wt{\nu}_{[c} \wt{\Weyl}_{d] a b} \right)_\circ$ for some trace-free $\wt{\Weyl}_{a b c}=\wt{\Weyl}_{a [b c]}$ satisfying $\wt{\Weyl}_{[a b c]} = 0$. Condition \eqref{eq:WeylonZ} now implies that $\wt{\Weyl}=0$ on $\wt{\mc{Z}}$.
			\item If $\wt{\Lambda} = 0$, then by \cite[Proposition~7.1]{TaghaviChabert2023b}, $\wt{\mc{Z}}$ is null, and so is $\wt{\nu}_a$. Hence \eqref{eq:WeylonZ} implies that $\wt{\nu}_a$ is a quadruple PND of $\wt{\Weyl}$ on $\wt{\mc{Z}}$. This would mean that $\wt{k}$ degenerates to a quadruple PND of $\wt{\Weyl}$ and so would be proportional to $\wt{\nu}^a$ on $\wt{\mc{Z}}$. But this is impossible since $\wt{\nu}^a$ is tangent to the null hypersurface $\wt{\mc{Z}}$ while $\wt{k}^{a}$ is transverse to it. Hence $\wt{\Weyl}=0$ on $\wt{\mc{Z}}$.
		\end{itemize}
		This ends the proof.
	\end{proof}
	
	As a consequence of Proposition \ref{prop:pureradasym}, Theorem \ref{thm:asymconfflat} and Proposition \ref{prop:Wasym}, we obtain:
	\begin{prop}\label{prop:confEins}
		Let $(\wt{\mc{M}},\wt{\mbf{c}}_{\wt{\xi}}, \wt{k} ) \longrightarrow (\mc{M},H,J)$ be a perturbed Fefferman space as in Proposition \ref{prop:Weyl}, and suppose it admits an almost pure radiation scale $\wt{\sigma}$. Then
		\begin{subequations}\label{eq:strasympEin}
			\begin{equation}
				\bm{\Psi}_2^{(0)} = 0 \, , 
			\end{equation}
			\begin{equation}
				\sigma^3 \ol{\sigma}^{-1} \bm{\Psi}_{2}^{(-4)} - \sigma^2 \bm{\Psi}_{2}^{(-2)}  = 0 \, , 
			\end{equation}
			\begin{equation}
				- \sigma^{4} \ol{\sigma}^{-2} \bm{\Psi}_3^{(-6)}  + \sigma^{3} \ol{\sigma}^{-1} \bm{\Psi}_3^{(-4)} - \sigma^{2}  \bm{\Psi}_3^{(-2)} + \sigma \ol{\sigma} \bm{\Psi}_3^{(0)} -  \ol{\sigma}^{2} \bm{\Psi}_3^{(2)} = 0 \, , 
			\end{equation}
			\begin{equation}
				\sigma^{5} \ol{\sigma}^{-3} \bm{\Psi}_4^{(-8)} - \sigma^{4} \ol{\sigma}^{-2} \bm{\Psi}_4^{(-6)} + \sigma^{3} \ol{\sigma}^{-1} \bm{\Psi}_4^{(-4)} - \sigma^{2} \bm{\Psi}_4^{(-2)} 
				+ \sigma \ol{\sigma} \bm{\Psi}_4^{(0)} - \ol{\sigma}^{2} \bm{\Psi}_4^{(2)} + \sigma \ol{\sigma}^{3} \bm{\Psi}_4^{(4)} = 0 \, ,
			\end{equation}
		\end{subequations}
		where $\sigma$ is the nowhere vanishing density of weight $(1,0)$ given in Theorem \ref{thm:pertFeffpurad}.
	\end{prop}
	Following Penrose \cite{Penrose1986}, we shall call the system of equations \eqref{eq:strasympEin} the \emph{strong asymptotic Einstein condition}, regardless of whether our manifold admits an almost Einstein scale or not.
	
	\begin{rem}
		Drawing from \cite{Curry2018}, we may think of $(\wt{\mc{M}},\wt{\mbf{c}}_{\wt{\xi}},\wt{k})$ admitting an almost half-Einstein scale $\wt{\sigma}$ with zero set $\wt{\mc{Z}}$ as the conformal compactification of some algebraically special optical geometry $(\wt{\mc{M}}', \wt{g},\wt{k})$ with twisting non-shearing congruence of null geodesics, where $\wt{\mc{M}}' = \wt{\mc{M}} \setminus \wt{\mc{Z}}$ and $\wt{g}=\wt{\sigma}^{-2} \wt{\bm{g}}$ is a half-Einstein metric. We can thus view $\wt{\mc{M}}'$ as the \emph{bulk} of some manifold $\wt{\mc{M}}$ with boundary $\wt{\mc{Z}}$, at least locally.
		
		Note however that these results are very much \emph{local}, and assumptions on the spacetime $(\wt{\mc{M}}',\wt{g},\wt{k})$ such as being \emph{asymptotically simple} \cite{Penrose1963} may be too strong. Indeed, it is demonstrated in \cite{Mason1998} that this condition might be too stringent in our context: the only algebraically special asymptotically simple Ricci-flat Lorentzian four-manifold is Minkowski space.
		
		Nevertheless, one may assume that the geodesics of the congruence generated by $\wt{k}$ reach only some portions of conformal infinity.  One could then use the above results to describe the asymptotic properties of such spacetimes, but only along the congruence. Thus, from Theorem \ref{thm:asymconfflat}, we may say that $\wt{\mc{M}}'$ is \emph{asymptotically conformally flat along the geodesics of $\wt{k}$}, and by \cite[Proposition~7.1]{TaghaviChabert2023b}, more specifically \emph{asymptotically flat, de Sitter or anti Sitter along the geodesics of $\wt{k}$} depending on the sign of the Ricci scalar. For instance, the Taub--NUT metric \cite{Taub1951,Newman1963} is well-known to be \emph{globally} non-asymptotically flat since the topology of its conformal infinity is $S^3$, but it is asymptotically flat along the geodesics of its twisting non-shearing congruences of null geodesics.
	\end{rem}

	\appendix
	
	\section{Adapted CR coframes of the second kind}\label{app:coframe2}
	The aim of this appendix is to explain the relation between the two kinds of adapted CR coframes used in Section \ref{sec:CR3} in the context of a contact CR three-manifold $(\mc{M},H,J)$. Here, any choice of nowhere vanishing section ${f}_{\alpha}$ of ${H}^{(1,0)}$ determines a unique nowhere vanishing real-valued vector field ${f}_0$ by the relation
	\begin{align}\label{eq:ffbf0}
		[{f}_{\alpha}, \overline{{f}}_{\bar{\beta}}] & = - \i {\delta}_{\alpha \bar{\beta}} {f}_0 \, .
	\end{align}
	We remind the reader that the indices $\alpha$,$\bar{\beta}$ only take the value $1$, which is what allows us to make the normalisation in \eqref{eq:ffbf0} with $\delta_{1 \bar{1}} = 1$. The triple $( {f}_0, {f}_{\alpha}, \overline{{f}}_{\bar{\alpha}})$ will be referred to as an \emph{adapted frame of the second kind}. Its dual coframe, which we shall denote $( \omega^0, {\omega}^{\alpha}, \overline{{\omega}}^{\bar{\alpha}})$, will be referred to as an \emph{adapted coframe of the second kind}. Writing $\omega := \omega^0$, the commutation relation \eqref{eq:ffbf0} is then equivalent to the structure equations
	\begin{align*}
		\d \omega & = \i \delta_{\alpha \bar{\beta}} {\omega}^{\alpha} \wedge \overline{{\omega}}{}^{\bar{\beta}} +  \omega \wedge \left( {\Gamma}_{\alpha}{\omega}^{\alpha}  + {\Gamma}_{\bar{\alpha}} {\omega}^{\bar{\alpha}} \right) \, , \\
		\d {\omega}^{\alpha}
		& = \Xi_{\beta}^{\alpha} \omega \wedge {\omega}^{\beta} + \Xi_{\bar{\beta}}^{\alpha} \omega \wedge \overline{{\omega}}{}^{\bar{\beta}} \, , \\
		\d \ol{{\omega}}^{\bar{\alpha}}
		& = \Xi_{\bar{\beta}}^{\bar{\alpha}} \omega \wedge \ol{{\omega}}^{\bar{\beta}} + \Xi_{\beta}^{\bar{\alpha}} \omega \wedge {\omega}{}^{\beta} \, ,
	\end{align*}
	for some complex-valued functions $\Gamma_{\alpha}$, $\Xi_{\beta}^{\alpha}$, $\Xi_{\bar{\beta}}^{\alpha}$, where $\Gamma_{\bar{\alpha}} = \ol{\Gamma_{\alpha}}$, $\Xi_{\bar{\beta}}^{\bar{\alpha}} = \ol{\Xi_{\beta}^{\alpha}}$ and $\Xi_{\beta}^{\bar{\alpha}} = \ol{\Xi_{\bar{\beta}}^{\alpha}}$. Note that by definition, $({f}_{\alpha})$ and $({\omega}^{\alpha})$  are necessarily unitary with respect to the Levi form of $\omega$. As before, we raise and lower the indices freely using $\delta_{\alpha \bar{\beta}}$ and its inverse, i.e.\ $v^{\alpha} = \delta^{\alpha \bar{\beta}} v_{\bar{\beta}}$, and so on.

	Now define
	\begin{align*}
		\theta & = \omega \, , &
		{\theta}^{\alpha} & = {\omega}^{\alpha} - \i {\Gamma}^{\alpha} \theta \, , &
		\ol{{\theta}}{}^{\bar{\alpha}} & = \ol{{\omega}}{}^{\bar{\alpha}} + \i \ol{{\Gamma}}{}^{\bar{\alpha}} \theta \, ,
	\end{align*}
	or equivalently,
	\begin{align*}
		\ell & = {f}_0 + \i {\Gamma}^{\alpha} {f}_{\alpha} - \i \ol{{\Gamma}}{}^{\bar{\alpha}} {f}_{\bar{\alpha}} \, , &
		{e}_{\alpha} & = {f}_{\alpha} \, , & \ol{{e}}_{\bar{\alpha}} & = \ol{{f}}_{\bar{\alpha}} \, .
	\end{align*}
	Then $(\theta, {\theta}^{\alpha}, \ol{{\theta}}{}^{\bar{\alpha}})$ is a \emph{unitary} coframe of the first kind with dual $(\ell, {e}_{\alpha}, \ol{{e}}_{\bar{\alpha}})$, and by inspection of the structure equations \eqref{eq:structure_CR} with $h_{\alpha \bar{\beta}} = \delta_{\alpha \bar{\beta}}$, we obtain the relation
	\begin{align*}
		\Gamma_{\alpha} & = \Gamma_{\alpha \bar{\beta}}{}^{\bar{\beta}} \, , & 	\Gamma_{\bar{\alpha}} & = \Gamma_{\bar{\alpha} \beta}{}^{\beta} \, , \\
		\Xi_{\beta}^{\alpha} & = -   {\Gamma}_{0 \gamma}{}^{\gamma} \delta_{\beta}^{\alpha} -  \i e_{\beta} ({\Gamma}^{\alpha}{}_{\gamma}{}^{\gamma})  \, , & \Xi_{\bar{\beta}}^{\bar{\alpha}} & = - {\Gamma}_{0 \bar{\gamma}}{}^{\bar{\gamma}} \delta_{\bar{\beta}}^{\bar{\alpha}} +  \i e_{\bar{\beta}} ({\Gamma}^{\bar{\alpha}}{}_{\bar{\gamma}}{}^{\bar{\gamma}})  \, , \\
		\Xi_{\bar{\beta}}^{\alpha} & =  {\WbA}^{\alpha}{}_{\bar{\beta}} -  \i e_{\bar{\beta}}({\Gamma}^{\alpha}{}_{\gamma}{}^{\gamma})  \, , & \Xi_{\beta}^{\bar{\alpha}} & = {\WbA}^{\bar{\alpha}}{}_{\beta} +  \i e_\beta ({\Gamma}^{\bar{\alpha}}{}_{\gamma}{}^{\gamma})   \, . 
	\end{align*}
	
	The above transformation, from a coframe of the second kind to a unitary coframe of the first kind, is clearly invertible. In addition, any such coframe determines the nowhere vanishing section $\theta \wedge \theta^1 = \omega \wedge \omega^1$ of the canonical bundle. Conversely, any nowhere vanishing local section $\zeta$ of the canonical bundle determines a unique contact form $\theta$ by the volume normalisation condition \eqref{eq:volnorm}. The remaining freedom is to choose the admissible coframe $(\theta^\alpha)$ such that $\zeta = \theta \wedge \theta^1$.
	
	In summary:
	\begin{lem}
		Locally, there is a one-to-one correspondence between
		\begin{itemize}
			\item unitary adapted coframes of the first kind,
			\item adapted coframes of the second kind,
			\item nowhere vanishing sections of the canonical bundle.
		\end{itemize}
	\end{lem}

	\section{Properties of the Bach tensor}\label{app:Bach_prop}
	The computations necessary for establishing Proposition \ref{prop:PetrovIII2Bach} and Theorem \ref{thm:PetrovII+Bach} are collected in the present appendix. Let $(\wt{\mc{M}},\wt{\mbf{c}},\wt{k})$ be an optical geometry of dimension four equipped with a twisting non-shearing congruence of null geodesics. We refer the reader to Section \ref{sec-optical} for the general setup induced from such a geometry, that is, we work with a metric $\wt{g}$ in $\accentset{n.e}{\wt{\mbf{c}}}$. This means that $\wt{\kappa} = \wt{g}(\wt{k},\cdot)$ is the pullback of a contact form on the leaf space $(\mc{M}, H, J)$ and
	\begin{align}\label{eq:spNkap}
		\wt{\nabla}_{a} \wt{\kappa}_b & = \wt{\tau}_{a b} + 2 \wt{E}_{(a} \wt{\kappa}_{b)} \, ,
	\end{align}
	for some two-form $\wt{\tau}_{a b}$ and one-form $\wt{E}_{a}$ such that $\wt{k}^{a} \wt{\tau}_{a b}= \wt{k}^{a} \wt{E}_{a} = 0$. In particular,
	\begin{align}
		\wt{k}^{b} \wt{\nabla}_{b} \wt{\kappa}_{a} & = 0 \, , \label{eq:affgeod} \\
		\wt{\nabla}^{a} \wt{\kappa}_{a} & = 0 \, , \label{eq:divfree}
	\end{align}
	as expected since the geodesics of $\wt{k}$ are affinely parametrised and the congruence is non-expanding.
	
	We assume from the outset that the Weyl tensor satisfies
	\begin{align}\label{eq:reprincW0}
		\wt{\Weyl}(\wt{k},\wt{v},\wt{k},\cdot) & = 0 \, , & \mbox{for any $\wt{v} \in \Gamma( \langle \wt{k} \rangle^\perp)$,}
	\end{align}
	which is equivalently to
	\begin{align}\label{eq:reprincW}
		\wt{\Weyl}(\wt{k},\wt{v},\wt{u},\wt{w}) & = 0 \, , & \mbox{for any $\wt{u},\wt{v},\wt{w} \in \Gamma( \langle \wt{k} \rangle^\perp )$.}
	\end{align}
	This can be re-expressed as $\wt{k}^{d} \wt{\Weyl}_{d a b c} = \wt{\kappa}_{a} \wt{\phi}_{[b c]} - \wt{\kappa}_{[b} \wt{\phi}_{c] a}$ for some $\wt{\phi}_{a b}$ satisfying $\wt{k}^{a} \wt{\phi}_{a [b} \wt{\kappa}_{c]} = 0$, and this tells us that
	\begin{align}\label{eq:reprincW2}
		\wt{k}^{d} \wt{\Weyl}_{d a b c} \wt{k}^{e} \wt{\Weyl}_{e}{}^{a b c} & = 0 \, .
	\end{align}
	In addition, since $\wt{\nabla}_{[a} \wt{\tau}_{b c]} = 0$, equation \eqref{eq:spNkap} gives
	\begin{align}
		\wt{\nabla}_{a} \wt{\tau}_{b c}
		& = \wt{\Riem}_{b c}{}^{d}{}_{a} \wt{\kappa}_{d} + 2 \wt{\tau}_{b c} \wt{E}_{a} + 2 \wt{\nabla}_{[b} \wt{E}_{c]} \wt{\kappa}_{a} \, . \label{eq:Ntau}
	\end{align}
	
	Using the definition of the Bach tensor \eqref{eq:Bach}, we are now in the position to compute
	\begin{align}\label{eq:Bkk}
		\wt{k}^a \wt{k}^b	\wt{\Bach}_{a b} & = - \wt{k}^a \wt{k}^b \wt{\nabla}^{c} \wt{\Cot}_{a b c} + \wt{k}^a \wt{k}^b \wt{\Rho}^{c d} \wt{\Weyl}_{a c b d} \, ,
	\end{align}
	under our given assumptions. We break down the computation of the first term of \eqref{eq:Bkk} into a number of steps. We first note that
	\begin{align}
		\wt{k}^a \wt{k}^b \wt{\nabla}^{c} \wt{\Cot}_{a b c} & = 	\wt{\nabla}^{c} \left( \wt{k}^a \wt{k}^b \wt{\Cot}_{a b c} \right) - \left( \wt{\nabla}^{c} \wt{k}^a \right) \wt{k}^b \wt{\Cot}_{a b c} - \left( \wt{\nabla}^{c} \wt{k}^b \right) \wt{k}^a \wt{\Cot}_{a b c} \, , \label{eq:kkNY}
	\end{align}
	Using \eqref{eq:Y=divW}, the product rule, multiple instances of \eqref{eq:spNkap} and the fact that $\langle \wt{k} \rangle$ is a repeated PND of the Weyl tensor, we find
	\begin{align}
		\wt{k}^a \wt{k}^b \wt{\Cot}_{a b c}
		& = \wt{\nabla}^d (\wt{k}^a \wt{k}^b \wt{\Weyl}_{d a b c}) - \tfrac{3}{2} \wt{\tau}^{d b} \wt{k}^a \wt{\Weyl}_{d b a c} \, . \label{eq:kkY}
	\end{align}
	By \eqref{eq:reprincW}, we get $\wt{k}^a \wt{k}^b \wt{\Cot}_{a b [c} \wt{\kappa}_{d]} = 0$.
	This fact, together with \eqref{eq:spNkap} and \eqref{eq:kkY}, tells us that
	\begin{align}
		\left( \wt{\nabla}^{c} \wt{k}^b \right) \wt{k}^a \wt{\Cot}_{a b c}
		& = - \wt{\tau}^{b c} \wt{k}^a \wt{\Cot}_{a b c} \, , \label{eq:NkkY1}
	\end{align}
	and using the fact that $\wt{\Cot}_{[a b c]} = 0$,
	\begin{align}
		\left( \wt{\nabla}^{c} \wt{k}^a \right) \wt{k}^b \wt{\Cot}_{a b c}
		& = - \tfrac{1}{2} \wt{\tau}^{b c} \wt{k}^a \wt{\Cot}_{a b c}  \, . \label{eq:NkkY2}
	\end{align}
	Hence, equation \eqref{eq:kkNY} simplifies to
	\begin{align}
		\wt{k}^a \wt{k}^b \wt{\nabla}^{c} \wt{\Cot}_{a b c} & = 	\wt{\nabla}^{c} \left( \wt{k}^a \wt{k}^b \wt{\Cot}_{a b c} \right) + \tfrac{3}{2} \wt{\tau}^{b c} \wt{k}^a \wt{\Cot}_{a b c} \, . \label{eq:kkNY2}
	\end{align}
	For the last term, we obtain
	\begin{align*}
		\wt{\tau}^{b c} \wt{k}^a \wt{\Cot}_{a b c}
		& = \wt{\nabla}^{d} (\wt{\tau}^{b c} \wt{k}^a \wt{\Weyl}_{d a b c}) + 2  \wt{\Rho}^{b d} \wt{k}^{c} \wt{k}^a \wt{\Weyl}_{d a b c} - \wt{\tau}^{d a} \wt{\tau}^{b c} \wt{\Weyl}_{d a b c} \, ,
	\end{align*}
	where we have used \eqref{eq:Y=divW}, the product rule, equations \eqref{eq:Ntau}, \eqref{eq:spNkap}, \eqref{eq-Riem_decomp},  and \eqref{eq:reprincW2}. Plugging this expression and \eqref{eq:kkY} into \eqref{eq:kkNY2} yields
	\begin{equation*}
		\wt{k}^a \wt{k}^b \wt{\nabla}^{c} \wt{\Cot}_{a b c}
		= 	\wt{\nabla}^{c} \wt{\nabla}^d (\wt{k}^a \wt{k}^b \wt{\Weyl}_{d a b c}) + 3 \wt{\nabla}^{d} (\wt{\tau}^{b c} \wt{k}^a \wt{\Weyl}_{d a b c}) 
		+ 3  \wt{\Rho}^{b d} \wt{k}^{c} \wt{k}^a \wt{\Weyl}_{d a b c} - \tfrac{3}{2} \wt{\tau}^{d a} \wt{\tau}^{b c} \wt{\Weyl}_{d a b c} \, ,
	\end{equation*}
	and therefore, equation \eqref{eq:Bkk} becomes
	\begin{equation}
		\wt{k}^a \wt{k}^b	\wt{\Bach}_{a b} = - \wt{\nabla}^{c} \wt{\nabla}^d (\wt{k}^a \wt{k}^b \wt{\Weyl}_{d a b c}) - 3 \wt{\nabla}^{d} ( \wt{\tau}^{b c} \wt{k}^a \wt{\Weyl}_{ d a b c} ) 
		+ \tfrac{3}{2} \wt{\tau}^{d a} \wt{\tau}^{b c} \wt{\Weyl}_{d a b c}  - 2 \wt{k}^a \wt{k}^b \wt{\Rho}^{c d} \wt{\Weyl}_{a c b d} \, . \label{eq:Bkk2}
	\end{equation}
	Using the symmetries of the Weyl tensor, and \eqref{eq:NP-Psi} for the definition of $\wt{\Psi}_2$, we find
	\begin{align*}
		\wt{k}^a \wt{k}^b \wt{\Weyl}_{d a b c} & =
		- \wt{\kappa}_d \wt{\kappa}_c (\wt{\Psi}_2 + \ol{\wt{\Psi}}_2) \, , &
		\wt{\tau}^{b c} \wt{k}^a \wt{\Weyl}_{d a b c} & =
		-2 \i \wt{\kappa}_{d} (\wt{\Psi}_2 - \ol{\wt{\Psi}}_2) \, , \\
		\wt{\tau}^{d a} \wt{\tau}^{b c} \wt{\Weyl}_{d a b c} & =
		- 4 (\wt{\Psi}_2 + \ol{\wt{\Psi}}_2) \, .
	\end{align*}
	Finally, substituting these into \eqref{eq:Bkk2} and using the product rule together with \eqref{eq:affgeod} and \eqref{eq:divfree}, we finally conclude 
	\begin{align}\label{eq:BachPsi2}
		\wt{k}^a \wt{k}^b	\wt{\Bach}_{a b}
		& = (\ddot{\wt{\Psi}}_2 + \ddot{\ol{\wt{\Psi}}}_2) + 6 \i (\dot{\wt{\Psi}}_2 - \dot{\ol{\wt{\Psi}}}_2)  - 8 (\wt{\Psi}_2 + \ol{\wt{\Psi}}_2)  \, ,
	\end{align}
	where we have made use of the fact that $\wt{\Rho}^{a b} \wt{\kappa}_a \wt{\kappa}_b = 1$.
	
	\section{Relation to the Debney--Kerr--Schild metric}\label{app:others}
	In this section, we provide the appropriate coframe and coordinate transformations to relate a perturbed Fefferman space $(\wt{\mc{M}},\wt{\mbf{c}}_{\wt{\xi}},\wt{k}) \rightarrow (\mc{M},H,J)$ that admits an almost half-Einstein scale, to the `standard' metric of \cite{Debney1969,Stephani2003}. Note that the relation between the latter and the form of the metric obtained in \cite{Hill2008} is already given in \cite{Zhang2013}. In that reference, it is rightly pointed out that the equations on the CR base space are simplest when given in the Debney--Kerr--Schild coordinates. But these lack obvious invariance. There is therefore a trade-off when expressing the Einstein equations in an invariant manner.
	
	The reader should refer to Theorem \ref{thm:hlfEinsc} for the notation used in the next paragraphs. Since $\i \sigma^{-1} \accentset{\xi}{\nabla}_{\alpha} \sigma$ is a solution to \eqref{eq:exactWW_Ein_b}, there exists, by Theorem \ref{thm:lambda_exact}, a density $\wh{\sigma}$ of weight $(1,0)$ such that
	\begin{align*}
		\i \sigma^{-1} \accentset{\xi}{\nabla}_{\alpha} \sigma & =  \i \left( \wh{\sigma}^{-1} \ol{\wh{\sigma}}{}^{2} \right) \nabla_{\alpha} \left( \wh{\sigma} \ol{\wh{\sigma}}{}^{-2} \right) \, ,
	\end{align*}
	Using the definition of $\accentset{\xi}{\nabla}_{\alpha}$ and \eqref{eq:xi2_b}, this allows us to express $\xi_{\alpha}^{(0)}$ and $\xi_{\alpha}^{(-2)}$ as
	\begin{align*}
		\xi_{\alpha}^{(0)} & = \i \left( \wh{\sigma}{}^{-1} \ol{\wh{\sigma}}{}^{2} \sigma \right) \nabla_{\alpha} \left( \wh{\sigma} \ol{\wh{\sigma}}{}^{-2} \sigma{}^{-1} \right) \, , \\
		\xi_{\alpha}^{(-2)}
		& = \tfrac{\i}{2} \left( \wh{\sigma}{}^{-1} \ol{\wh{\sigma}}{}^{2} \sigma \ol{\sigma} \right) \nabla_{\alpha} \left( \wh{\sigma} \ol{\wh{\sigma}}{}^{-2} \sigma{}^{-1} \ol{\sigma}^{-1}\right) \, .
	\end{align*}
	Note that $\wh{\sigma} \ol{\wh{\sigma}}{}^{-2} \sigma{}^{-1} \in \Gamma(\mc{E}(0,-2))$ and $\wh{\sigma} \ol{\wh{\sigma}}{}^{-2} \sigma{}^{-1} \ol{\sigma}^{-1} \in \Gamma(\mc{E}(0,-3))$. Substituting these into \eqref{eq:xi0}, we can derive similar expressions for $\bm{\xi}_{0}^{(0)}$, $\bm{\xi}_{0}^{(2)}$ and $\bm{\xi}_{0}^{(4)}$.
	
	Let us write
	\begin{align*}
		\wh{\sigma} & = w \sigma \, , & \mbox{where} & &  w & = \mr{\varrho} \e^{\i \mr{\phi}} \, ,
	\end{align*}
	for some smooth real-valued functions $\mr{\varrho} > 0$ and $\mr{\phi} \in [-\pi,\pi)$. This induces a change of contact form $\wh{\theta} = \mr{\varrho}^{-2} \theta$, and, with reference to Appendix \ref{app:coframe2}, there are unitary adapted coframes $(\wh{\theta},\wh{\theta}^{1},\ol{\wh{\theta}}{}^{\bar{1}})$ and $(\wh{\omega},\wh{\omega}^{1},\ol{\wh{\omega}}{}^{\bar{1}})$ of the first and second kinds respectively, related by
	\begin{align*}
		\wh{\theta} & = \wh{\omega} \, , &
		\wh{\theta}^{1} & = \wh{\omega}^1 - \i \wh{\Gamma}^{1} \wh{\omega} \, , 
	\end{align*}
	where $\wh{\Gamma}_{\bar{1}}$ is the connection one-form of the (partial) Webster connection preserving $\wh{\theta}$. In addition, as guaranteed by Theorem \ref{thm:lambda_exact},
	\begin{align*}
		\wh{\omega}^1 & = \d \zeta \, ,
	\end{align*}
	for some CR function $\zeta$.  We can re-express
	\begin{align*}
		\wh{\xi}_{\alpha}^{(0)} = 2 \i \wh{\Gamma}_{\alpha} + \i w^{-1} \wh{\nabla}_{\alpha} w \, ,
	\end{align*}
	and similar formulae can be derived for the remaining CR data.
	
	We now trivialise the Fefferman bundle with $\wh{\sigma}$. This induces a change of perturbed Fefferman metric $\wt{g}_{\wh{\theta},\wt{\xi}} = \mr{\varrho}^{-2} \wt{g}_{\theta,\wt{\xi}}$, a change of fibre coordinate $\wh{\phi} = \phi + \mr{\phi}$, and a change of one-form $\wt{\lambda}=\wt{\omega}_{\theta} + \wt{\xi}$ as
	\begin{align*}
		\wh{\wt{\lambda}} & = \wt{\lambda} - \tfrac{1}{2} \i \Upsilon_{1} \theta^{1} + \tfrac{1}{2} \i \Upsilon_{\bar{1}} \ol{\theta}{}^{\bar{1}} - \tfrac{1}{2} \Upsilon_{1} \Upsilon^{1} \theta \, ,
	\end{align*}
	where $\Gamma_1 = \nabla_{1} \log \mr{\varrho}$. Defining
	\begin{align*}
		\wh{\wt{\lambda}}_{\omega} & = \wh{\wt{\lambda}} + \tfrac{1}{2} \i \Gamma_1 \wh{\omega}^1 - \tfrac{1}{2} \i \Gamma_{\bar{1}} \wh{\omega}^{\bar{1}} + \tfrac{1}{2} \Gamma_1 \Gamma^1 \wh{\omega} \, ,
	\end{align*}
	we can then write out our metric in one of the forms given in \cite{Hill2008} as
	\begin{align*}
		\wt{g} & = \sec^2 (\wh{\phi} - \mr{\phi}) \cdot \mr{\varrho}^2 \cdot \left( 4 \wh{\theta} \odot \wh{\wt{\lambda}}_{\wt{\omega}} + 2 \wh{\omega}^1 \odot \ol{\wh{\omega}}{}^{\bar{1}} \right) \, ,
	\end{align*}
	with some obvious changes in the notation.
	
	Our next step is to introduce a radial coordinate by the formula
	\begin{align*}
		r & = \mr{\varrho} \tan (\wh{\phi} - \mr{\phi}) \, , & \mbox{i.e.} & & \frac{\mr{\varrho} + \i r}{\sqrt{r^2 + \mr{\varrho}^2}} & = \e^{\i (\wh{\phi} - \mr{\phi})} \, ,
	\end{align*}
	so that
	\begin{align*}
		r^2 + \mr{\varrho}^2 & = \mr{\varrho}^2 \sec^2 ( \wh{\phi} - \mr{\phi}) \, , &
		\d  \wh{\phi} - \d \mr{\phi} & = \frac{1}{r^2 + \mr{\varrho}^2} \left( \mr{\varrho} \d r - r \d \mr{\varrho} \right) \, .
	\end{align*}
	It is then a straightforward matter to re-express the metric in terms of the radial coordinate $r$. We may also check that the asymptotic properties of the `unphysical' metric $\wt{g}_{\wh{\theta}, \wt{\xi}}$ as $r$ tends to $\pm \infty$ are consistent with the findings of Section \ref{sec:asym}.

	To make contact with the Debney--Kerr--Schild form of the metric \cite{Debney1969,Stephani2003}, we rescale the one-forms $\wh{\wt{\lambda}}$ and $\wh{\wt{\xi}}$ as 
	\begin{align*}
		\wh{\wt{\lambda}}{}'_{\omega} & = \mr{\varrho} \cdot \sec^2 ( \wh{\phi} - \mr{\phi}) \cdot \wt{\wh{\lambda}}_{\omega} \, , & 	\wt{\xi}' & = \mr{\varrho} \sec^2 ( \wh{\phi} - \mr{\phi}) \wt{\xi} \, . 
	\end{align*}
	so that
	\begin{align*}
		\wh{\wt{\lambda}}{}'_{\omega} & = \d r - r \d \log \mr{\varrho}  + \frac{r^2 + \mr{\varrho}^2}{\mr{\varrho}} \left( \omega_{\wh{\theta}} - \tfrac{1}{3} \wh{\Rho} \wh{\theta} \right) + \wt{\xi}'  \, ,
	\end{align*}
	Then the metric $\wt{g}$ takes the form
	\begin{align*}
		\wt{g} & = 4 \mr{\varrho} \cdot \wh{\omega} \odot \wh{\wt{\lambda}}{}'_{\omega} + 2 (r^2 + \mr{\varrho}^2)  \wh{\omega}^{1} \odot \ol{\wh{\omega}}{} ^{\bar{1}}  \, .
	\end{align*}
	We may also choose a local transverse coordinate $u$ and complex-valued function $Z$ on $\mc{M}$ so that
	\begin{align*}
		\wh{\omega} = \frac{1}{2 \mr{\varrho} } \left(\d u + Z \d \zeta + \ol{Z} \d \ol{\zeta} \right) \, .
	\end{align*}
	Dropping the hats, we can now recognise metric \eqref{eq:Debmet} given in Section \ref{sec:intro}.
	
	\section*{Acknowledgments}
	The author is grateful to Jerzy Lewandowski, who sadly passed away in October 2024, for his support, both as a colleague and as a friend. The author also would like to thank Rod Gover for helpful and stimulating conversations. Finally, the author is thankful to the anonymous referee for their careful review and pertinent comments, notably for suggesting a simpler proof of Theorem \ref{thm:real2dens_d3} and for bringing references \cite{Curry2019a,Curry2021} to his attention. 
	
	\bibliographystyle{amsplain}

\begin{thebibliography}{10}
		
		\bibitem{Akahori1987}
		Takao Akahori, \emph{A new approach to the local embedding theorem of
			{CR}-structures for {$n\geq 4$} (the local solvability for the operator
			{$\overline\partial_b$} in the abstract sense)}, Mem. Amer. Math. Soc.
		\textbf{67} (1987), no.~366, xvi+257. \MR{888499}
		
		\bibitem{Bailey1994}
		T.~N. Bailey, M.~G. Eastwood, and A.~R. Gover, \emph{Thomas's structure bundle
			for conformal, projective and related structures}, Rocky Mountain J. Math.
		\textbf{24} (1994), no.~4, 1191--1217. \MR{1322223}
		
		\bibitem{Baston1989}
		Robert~J. Baston and Michael~G. Eastwood, \emph{The {P}enrose transform},
		Oxford Mathematical Monographs, The Clarendon Press Oxford University Press,
		New York, 1989, Its interaction with representation theory, Oxford Science
		Publications.
		
		\bibitem{Bernstein1975}
		I.~N. Bernstein, I.~M. Gelfand, and S.~I. Gelfand, \emph{Differential operators
			on the base affine space and a study of {{\(\mathfrak g\)}}-modules}, Lie
		{Groups} {Represent}., {Proc}. {Summer} {Sch}. {Bolyai} {Janos} {Math}.
		{Soc}., {Budapest} 1971, 21-64 (1975)., 1975.
		
		\bibitem{Cap2014}
		A.~{\v{C}}ap, A.~R. Gover, and M.~Hammerl, \emph{Holonomy reductions of
			{C}artan geometries and curved orbit decompositions}, Duke Math. J.
		\textbf{163} (2014), no.~5, 1035--1070. \MR{3189437}
		
		\bibitem{Cap2008a}
		Andreas {\v{C}}ap, \emph{Infinitesimal automorphisms and deformations of
			parabolic geometries}, J. Eur. Math. Soc. (JEMS) \textbf{10} (2008), no.~2,
		415--437. \MR{2390330}
		
		\bibitem{Cap2008}
		Andreas {\v{C}}ap and A.~Rod Gover, \emph{C{R}-tractors and the {F}efferman
			space}, Indiana Univ. Math. J. \textbf{57} (2008), no.~5, 2519--2570.
		\MR{2463976}
		
		\bibitem{Cap2010}
		\bysame, \emph{A holonomy characterisation of {F}efferman spaces}, Ann. Global
		Anal. Geom. \textbf{38} (2010), no.~4, 399--412. \MR{2733370}
		
		\bibitem{Cap2001}
		Andreas {\v{C}}ap, Jan Slov{\'{a}}k, and Vladim{\'{i}}r Sou{\v{c}}ek,
		\emph{Bernstein-{G}elfand-{G}elfand sequences}, Ann. of Math. (2)
		\textbf{154} (2001), no.~1, 97--113. \MR{1847589}
		
		\bibitem{Curry2019a}
		Sean~N. Curry and Peter Ebenfelt, \emph{Bounded strictly pseudoconvex domains
			in {$\Bbb{C}^2$} with obstruction flat boundary {II}}, Adv. Math.
		\textbf{352} (2019), 611--631. \MR{3964157}
		
		\bibitem{Curry2021}
		\bysame, \emph{Bounded strictly pseudoconvex domains in {$\Bbb C^2$} with
			obstruction flat boundary}, Amer. J. Math. \textbf{143} (2021), no.~1,
		265--306. \MR{4201785}
		
		\bibitem{Curry2018}
		Sean~N. Curry and A.~Rod Gover, \emph{An introduction to conformal geometry and
			tractor calculus, with a view to applications in general relativity},
		Asymptotic analysis in general relativity, London Math. Soc. Lecture Note
		Ser., vol. 443, Cambridge Univ. Press, Cambridge, 2018, pp.~86--170.
		\MR{3792084}
		
		\bibitem{Debney1969}
		G.~C. Debney, R.~P. Kerr, and A.~Schild, \emph{Solutions of the {E}instein and
			{E}instein-{M}axwell equations}, J. Math. Phys. \textbf{10} (1969),
		1842--1854.
		
		\bibitem{Fefferman1976a}
		Charles~L. Fefferman, \emph{Monge-{A}mp\`ere equations, the {B}ergman kernel,
			and geometry of pseudoconvex domains}, Ann. of Math. (2) \textbf{103} (1976),
		no.~2, 395--416. \MR{407320}
		
		\bibitem{Fino2023}
		Anna Fino, Thomas Leistner, and Arman Taghavi-Chabert, \emph{Almost {R}obinson
			geometries}, Lett. Math. Phys. \textbf{113} (2023), no.~3, 56. \MR{4589302}
		
		\bibitem{Fino2023a}
		\bysame, \emph{{Optical geometries}}, Ann. Scuola Norm. Sup. Pisa Cl. Sci. (5)
		\textbf{26} (2023), no.~(1), 341--396.
		
		\bibitem{Goldberg1962}
		J.~N. Goldberg and R.~K. Sachs, \emph{A theorem on {P}etrov types}, Acta Phys.
		Polon. \textbf{22} (1962), no.~suppl, suppl, 13--23. \MR{156679}
		
		\bibitem{Gover2005a}
		A.~Rod Gover, \emph{Almost conformally {E}instein manifolds and obstructions},
		Differential geometry and its applications, Matfyzpress, Prague, 2005,
		pp.~247--260. \MR{2268937}
		
		\bibitem{Gover2005}
		A.~Rod Gover and C.~Robin Graham, \emph{C{R} invariant powers of the
			sub-{L}aplacian}, J. Reine Angew. Math. \textbf{583} (2005), 1--27.
		\MR{2146851}
		
		\bibitem{GoverHillNurowski11}
		A.~Rod Gover, C.~Denson Hill, and Pawe{\l} Nurowski, \emph{Sharp version of the
			{G}oldberg-{S}achs theorem}, Ann. Mat. Pura Appl. (4) \textbf{190} (2011),
		no.~2, 295--340. \MR{2786175 (2012j:53057)}
		
		\bibitem{Hill2008}
		C.~Denson Hill, Jerzy Lewandowski, and Pawe\l Nurowski, \emph{Einstein's
			equations and the embedding of 3-dimensional {CR} manifolds}, Indiana Univ.
		Math. J. \textbf{57} (2008), no.~7, 3131--3176. \MR{2492229}
		
		\bibitem{Jacobowitz1987}
		Howard Jacobowitz, \emph{The canonical bundle and realizable {CR}
			hypersurfaces}, Pacific J. Math. \textbf{127} (1987), no.~1, 91--101.
		\MR{876018}
		
		\bibitem{Jacobowitz1990}
		\bysame, \emph{An introduction to {CR} structures}, Mathematical Surveys and
		Monographs, vol.~32, American Mathematical Society, Providence, RI, 1990.
		\MR{1067341}
		
		\bibitem{Jacobowitz2020}
		\bysame, \emph{A conjecture of {T}rautman}, Lecture {N}otes of {S}eminario
		{I}nterdisciplinare di {M}atematica. {V}ol. {XV}, Lect. Notes Semin.
		Interdiscip. Mat., vol.~15, Semin. Interdiscip. Mat. (S.I.M.), Potenza, 2020,
		pp.~33--43. \MR{4394997}
		
		\bibitem{Jacobowitz1982}
		Howard Jacobowitz and Fran\c{c}ois Tr\`{e}ves, \emph{Nonrealizable {CR}
			structures}, Invent. Math. \textbf{66} (1982), no.~2, 231--249. \MR{656622}
		
		\bibitem{Jacobowitz1983}
		\bysame, \emph{Aberrant {CR} structures}, Hokkaido Math. J. \textbf{12} (1983),
		no.~3, 276--292. \MR{719968}
		
		\bibitem{Kuranishi1982}
		Masatake Kuranishi, \emph{Strongly pseudoconvex {CR} structures over small
			balls. {III}. {A}n embedding theorem}, Ann. of Math. (2) \textbf{116} (1982),
		no.~2, 249--330. \MR{672837}
		
		\bibitem{LeBrun1984}
		Claude~R. LeBrun, \emph{Twistor {CR} manifolds and three-dimensional conformal
			geometry}, Trans. Amer. Math. Soc. \textbf{284} (1984), no.~2, 601--616.
		\MR{743735}
		
		\bibitem{LeBrun1985}
		\bysame, \emph{Ambi-twistors and {E}instein's equations}, Classical Quantum
		Gravity \textbf{2} (1985), no.~4, 555--563. \MR{795101}
		
		\bibitem{Lee1986}
		John~M. Lee, \emph{The {F}efferman metric and pseudo-{H}ermitian invariants},
		Trans. Amer. Math. Soc. \textbf{296} (1986), no.~1, 411--429. \MR{837820}
		
		\bibitem{Lee1988}
		\bysame, \emph{Pseudo-{E}instein structures on {CR} manifolds}, Amer. J. Math.
		\textbf{110} (1988), no.~1, 157--178. \MR{926742}
		
		\bibitem{Lewandowski1988}
		Jerzy Lewandowski, \emph{On the {F}efferman class of metrics associated with a
			three-dimensional {CR} space}, Lett. Math. Phys. \textbf{15} (1988), no.~2,
		129--135. \MR{943984}
		
		\bibitem{Lewandowski1991}
		\bysame, \emph{Twistor equation in a curved spacetime}, Classical Quantum
		Gravity \textbf{8} (1991), no.~1, L11--L17. \MR{1088742}
		
		\bibitem{Lewandowski1990}
		Jerzy Lewandowski and Pawe\l Nurowski, \emph{Algebraically special twisting
			gravitational fields and {CR} structures}, Classical Quantum Gravity
		\textbf{7} (1990), no.~3, 309--328. \MR{1038612}
		
		\bibitem{Lewandowski1990b}
		\bysame, \emph{Cartan's chains and {L}orentz geometry}, J. Geom. Phys.
		\textbf{7} (1990), no.~1, 63--80. \MR{1094731}
		
		\bibitem{Lewandowski1990a}
		Jerzy Lewandowski, Pawe\l Nurowski, and Jacek Tafel, \emph{Einstein's equations
			and realizability of {CR} manifolds}, Classical Quantum Gravity \textbf{7}
		(1990), no.~11, L241--L246. \MR{1078890}
		
		\bibitem{Lewy1957}
		Hans Lewy, \emph{An example of a smooth linear partial differential equation
			without solution}, Ann. of Math. (2) \textbf{66} (1957), 155--158. \MR{88629}
		
		\bibitem{Mariot1954}
		Louis Mariot, \emph{Le champ \'{e}lectromagn\'{e}tique singulier}, C. R. Acad.
		Sci. Paris \textbf{238} (1954), 2055--2056. \MR{62543}
		
		\bibitem{Mason1998}
		L.~J. Mason, \emph{The asymptotic structure of algebraically special
			spacetimes}, Classical Quantum Gravity \textbf{15} (1998), no.~4, 1019--1030.
		\MR{1620249}
		
		\bibitem{Matsumoto2022}
		Yoshihiko Matsumoto, \emph{The {CR} {K}illing operator and
			{B}ernstein-{G}elfand-{G}elfand construction in {CR} geometry},
		arXiv:2205.11022 [math.DG] (2022).
		
		\bibitem{Newman1963}
		E.~Newman, L.~Tamburino, and T.~Unti, \emph{Empty-space generalization of the
			{S}chwarzschild metric}, J. Mathematical Phys. \textbf{4} (1963), 915--923.
		\MR{0152345 (27 \#2325)}
		
		\bibitem{Nirenberg1973}
		Louis Nirenberg, \emph{Lectures on linear partial differential equations},
		Conference Board of the Mathematical Sciences Regional Conference Series in
		Mathematics, vol. No. 17, American Mathematical Society, Providence, RI,
		1973, Expository Lectures from the CBMS Regional Conference held at the Texas
		Technological University, Lubbock, Tex., May 22--26, 1972. \MR{450755}
		
		\bibitem{Penrose1963}
		Roger Penrose, \emph{Asymptotic properties of fields and space-times}, Phys.
		Rev. Lett. \textbf{10} (1963), 66--68. \MR{149912}
		
		\bibitem{Penrose1986}
		Roger Penrose and Wolfgang Rindler, \emph{Spinors and space-time. {V}ol. 2},
		second ed., Cambridge Monographs on Mathematical Physics, Cambridge
		University Press, Cambridge, 1988, Spinor and twistor methods in space-time
		geometry. \MR{944085}
		
		\bibitem{Petrov1954}
		A.~Z. Petrov, \emph{Classification of spaces defining gravitational fields},
		Kazan. Gos. Univ. U\v{c}. Zap. \textbf{114} (1954), no.~8, 55--69.
		\MR{0076401}
		
		\bibitem{Robinson1961}
		Ivor Robinson, \emph{Null electromagnetic fields}, J. Mathematical Phys.
		\textbf{2} (1961), 290--291. \MR{127369}
		
		\bibitem{Robinson1985}
		Ivor Robinson and Andrzej Trautman, \emph{Integrable optical geometry}, Letters
		in Mathematical Physics \textbf{10} (1985), 179--182.
		
		\bibitem{Robinson1986}
		\bysame, \emph{Cauchy-{R}iemann structures in optical geometry}, Proceedings of
		the fourth {M}arcel {G}rossmann meeting on general relativity, {P}art {A},
		{B} ({R}ome, 1985), North-Holland, Amsterdam, 1986, pp.~317--324. \MR{879758}
		
		\bibitem{Robinson1989}
		\bysame, \emph{Optical geometry}, New Theories in Physics, Proc. of the XI
		Warsaw Symp. on Elementary Particle Physics, Kazimierz 23-27 May 1988
		(Pokorski~S. Ajduk, Z. and Trautman A., eds.), World Scientific, 1989,
		pp.~454--497.
		
		\bibitem{Schmalz2019}
		G.~Schmalz and M.~Ganji, \emph{A criterion for local embeddability of
			three-dimensional {CR} structures}, Ann. Mat. Pura Appl. (4) \textbf{198}
		(2019), no.~2, 491--503. \MR{3927166}
		
		\bibitem{Stephani2003}
		Hans Stephani, Dietrich Kramer, Malcolm MacCallum, Cornelius Hoenselaers, and
		Eduard Herlt, \emph{Exact solutions of {E}instein's field equations}, second
		ed., Cambridge Monographs on Mathematical Physics, Cambridge University
		Press, Cambridge, 2003. \MR{2003646}
		
		\bibitem{Tafel1985}
		Jacek Tafel, \emph{On the {R}obinson theorem and shearfree geodesic null
			congruences}, Lett. Math. Phys. \textbf{10} (1985), no.~1, 33--39.
		\MR{796997}
		
		\bibitem{Tafel1986}
		\bysame, \emph{Erratum: ``{O}n the {R}obinson theorem and shearfree geodesic
			null congruences'' [{L}ett. {M}ath. {P}hys. {\bf 10} (1985), no. 1, 33--39]},
		Lett. Math. Phys. \textbf{11} (1986), no.~2, 177. \MR{836073}
		
		\bibitem{TaghaviChabert2022}
		Arman Taghavi-Chabert, \emph{Twisting non-shearing congruences of null
			geodesics, almost {CR} structures and {E}instein metrics in even dimensions},
		Ann. Mat. Pura Appl. (4) \textbf{201} (2022), no.~2, 655--693. \MR{4386840}
		
		\bibitem{TaghaviChabert2023b}
		Arman Taghavi-Chabert, \emph{Perturbations of {F}efferman spaces over almost
			{CR} manifolds}, arXiv:2309.16986 [math.DG] (2023), To be published in
		\textit{Mathematische Zeitschrift}.
		
		\bibitem{Tanaka1975}
		Noboru Tanaka, \emph{A differential geometric study on strongly pseudo-convex
			manifolds}, Kinokuniya Book-Store Co., Ltd., Tokyo, 1975, Lectures in
		Mathematics, Department of Mathematics, Kyoto University, No. 9. \MR{0399517}
		
		\bibitem{Taub1951}
		A.~H. Taub, \emph{Empty space-times admitting a three parameter group of
			motions}, Ann. of Math. (2) \textbf{53} (1951), 472--490. \MR{41565}
		
		\bibitem{Trautman1984}
		Andrzej Trautman, \emph{Deformations of the {H}odge map and optical geometry},
		J. Geom. Phys. \textbf{1} (1984), no.~2, 85--95. \MR{794981}
		
		\bibitem{Trautman1985}
		\bysame, \emph{Optical structures in relativistic theories}, no. Num\'{e}ro
		Hors S\'{e}rie, 1985, The mathematical heritage of \'{E}lie Cartan (Lyon,
		1984), pp.~401--420. \MR{837209}
		
		\bibitem{Trautman1999}
		\bysame, \emph{Gauge and optical aspects of gravitation}, Classical Quantum
		Gravity \textbf{16} (1999), no.~12A, A157--A175. \MR{1728438}
		
		\bibitem{Trautman1999a}
		\bysame, \emph{On complex structures in physics}, On {E}instein's path ({N}ew
		{Y}ork, 1996), Springer, New York, 1999, pp.~487--501. \MR{1658884}
		
		\bibitem{Trautman2002}
		\bysame, \emph{Robinson manifolds and {C}auchy-{R}iemann spaces}, Classical
		Quantum Gravity \textbf{19} (2002), no.~2, R1--R10. \MR{1885472}
		
		\bibitem{Webster1978}
		S.~M. Webster, \emph{Pseudo-{H}ermitian structures on a real hypersurface}, J.
		Differential Geometry \textbf{13} (1978), no.~1, 25--41. \MR{520599}
		
		\bibitem{Webster1989}
		Sidney Webster, \emph{On the proof of {K}uranishi's embedding theorem}, Ann.
		Inst. H. Poincar\'e{} Anal. Non Lin\'eaire \textbf{6} (1989), no.~3,
		183--207. \MR{995504}
		
		\bibitem{Webster1989a}
		Sidney~M. Webster, \emph{On the local solution of the tangential
			{C}auchy-{R}iemann equations}, Ann. Inst. H. Poincar\'e{} Anal. Non
		Lin\'eaire \textbf{6} (1989), no.~3, 167--182. \MR{995503}
		
		\bibitem{Zhang2013}
		Xuefeng Zhang and Daniel Finley, \emph{C{R} structures and twisting vacuum
			spacetimes with two {K}illing vectors and cosmological constant: type {II}
			and more special}, Classical Quantum Gravity \textbf{30} (2013), no.~11,
		115006, 20. \MR{3055095}
		
	\end{thebibliography}
	
	\providecommand{\bysame}{\leavevmode\hbox to3em{\hrulefill}\thinspace}
	\providecommand{\MR}{\relax\ifhmode\unskip\space\fi MR }
	\providecommand{\MRhref}[2]{%
		\href{http://www.ams.org/mathscinet-getitem?mr=#1}{#2}
	}
	\providecommand{\href}[2]{#2}

\end{document}